\documentclass[11pt]{amsart}
\usepackage{amssymb,mathrsfs}
\title[Polarized Real Tori ]{\large Polarized Real Tori }

\usepackage{MnSymbol}

\begin{document}

\author{Jae-Hyun Yang}

\address{Department of Mathematics, Inha University,
Incheon 402-751, Korea}
\email{jhyang@inha.ac.kr }


\newtheorem{theorem}{Theorem}[section]
\newtheorem{lemma}{Lemma}[section]
\newtheorem{proposition}{Proposition}[section]
\newtheorem{remark}{Remark}[section]
\newtheorem{definition}{Definition}[section]
\newtheorem{corollary}{Corollary}[section]

\renewcommand{\theequation}{\thesection.\arabic{equation}}
\renewcommand{\thetheorem}{\thesection.\arabic{theorem}}
\renewcommand{\thelemma}{\thesection.\arabic{lemma}}
\newcommand{\BR}{\mathbb R}
\newcommand{\BQ}{\mathbb Q}
\newcommand{\BT}{\mathbb T}
\newcommand{\BM}{\mathbb M}
\newcommand{\bn}{\bf n}
\def\charf {\mbox{{\text 1}\kern-.24em {\text l}}}
\newcommand{\BC}{\mathbb C}
\newcommand{\BZ}{\mathbb Z}

\newcommand{\HG}{{\mathscr H}_g}
\newcommand{\GG}{\G_g^{\star}}
\newcommand{\XG}{{\mathscr X}_\BR^g}
\newcommand{\XGI}{{\mathscr X}_{(\lambda,i)}^g}

\newcommand{\OHG}{\overline{{\mathscr H}_g} }
\newcommand{\OXG}{\overline{{\mathscr X}_\BR^g} }

\newcommand\Mg{{\mathcal M}_g}
\newcommand\Rg{{\mathfrak R}_g}
\newcommand\PG{{\mathcal P}_g}

\thanks{
\indent polarized real tori, line bundles over a real torus, semi-abelian varieties, semi-tori.  \\
\indent This work was supported by Basic Science Program through the National Research Foundation\\
\indent of Korea(NRF) funded by the Ministry of Education, Science and Technology (43338-01) and \\
\indent partially supported by the Max-Planck-Institut f{\"u}r Mathematik in Bonn.}


\begin{abstract}
{For a fixed positive integer $g$, we let ${\mathcal P}_g = \,\big\{ Y\in {\mathbb R}^{(g,g)}\,|
\ Y=\,{}^tY>0\,\big\}$ be the open convex cone in the Euclidean space ${\mathbb R}^{g(g+1)/2}$. Then the general linear group
$GL(g,{\mathbb R})$ acts naturally on ${\mathcal P}_g$ by $A\star Y=\,AY\,{}^t\!A$ ($A\in GL(g,{\mathbb R}),\ Y\in
{\mathcal P}_g$). We introduce a notion of polarized real tori. We show that the open cone ${\mathcal P}_g$
parametrizes principally polarized real tori of dimension $g$ and that
the Minkowski domain ${\mathfrak R}_g=\,GL(g,{\mathbb Z})\backslash {\mathcal P}_g$  may be regarded as a moduli
space of principally polarized real tori of dimension $g$. We also study smooth line bundles on a polarized
real torus by relating them to holomorphic line bundles on its associated polarized real abelian variety.}
\end{abstract}

\maketitle

\newcommand\tr{\triangleright}
\newcommand\al{\alpha}
\newcommand\be{\beta}
\newcommand\g{\gamma}
\newcommand\gh{\Cal G^J}
\newcommand\G{\Gamma}
\newcommand\de{\delta}
\newcommand\e{\epsilon}
\newcommand\z{\zeta}
\newcommand\vth{\vartheta}
\newcommand\vp{\varphi}
\newcommand\om{\omega}
\newcommand\p{\pi}
\newcommand\la{\lambda}
\newcommand\lb{\lbrace}
\newcommand\lk{\lbrack}
\newcommand\rb{\rbrace}
\newcommand\rk{\rbrack}
\newcommand\s{\sigma}
\newcommand\w{\wedge}
\newcommand\fgj{{\frak g}^J}
\newcommand\lrt{\longrightarrow}
\newcommand\lmt{\longmapsto}
\newcommand\lmk{(\lambda,\mu,\kappa)}
\newcommand\Om{\Omega}
\newcommand\ka{\kappa}
\newcommand\ba{\backslash}
\newcommand\ph{\phi}
\newcommand\M{{\Cal M}}
\newcommand\bA{\bold A}
\newcommand\bH{\bold H}
\newcommand\D{\Delta}

\newcommand\Hom{\text{Hom}}
\newcommand\cP{\Cal P}

\newcommand\cH{\Cal H}

\newcommand\pa{\partial}

\newcommand\pis{\pi i \sigma}
\newcommand\sd{\,\,{\vartriangleright}\kern -1.0ex{<}\,}
\newcommand\wt{\widetilde}
\newcommand\fg{\frak g}
\newcommand\fk{\frak k}
\newcommand\fp{\frak p}
\newcommand\fs{\frak s}
\newcommand\fh{\frak h}
\newcommand\Cal{\mathcal}

\newcommand\fn{{\frak n}}
\newcommand\fa{{\frak a}}
\newcommand\fm{{\frak m}}
\newcommand\fq{{\frak q}}
\newcommand\CP{{\mathcal P}_g}
\newcommand\Hgh{{\mathbb H}_g \times {\mathbb C}^{(h,g)}}
\newcommand\BD{\mathbb D}
\newcommand\BH{\mathbb H}
\newcommand\CCF{{\mathcal F}_g}
\newcommand\CM{{\mathcal M}}
\newcommand\Ggh{\Gamma_{g,h}}
\newcommand\Chg{{\mathbb C}^{(h,g)}}
\newcommand\Yd{{{\partial}\over {\partial Y}}}
\newcommand\Vd{{{\partial}\over {\partial V}}}

\newcommand\Ys{Y^{\ast}}
\newcommand\Vs{V^{\ast}}
\newcommand\LO{L_{\Omega}}
\newcommand\fac{{\frak a}_{\mathbb C}^{\ast}}

\newcommand\OW{\overline{W}}
\newcommand\OP{\overline{P}}
\newcommand\OQ{\overline{Q}}
\newcommand\Dg{{\mathbb D}_g}
\newcommand\Hg{{\mathbb H}_g}

\newcommand\OBD{ {\overline{\BD}}_g }

\newcommand\La{\Lambda}
\newcommand\FA{\mathfrak A}
\newcommand\FL{\mathfrak L}
\newcommand\FT{\mathfrak T}

\newcommand\POB{ {{\partial}\over {\partial{\overline \Omega}}} }
\newcommand\PZB{ {{\partial}\over {\partial{\overline Z}}} }
\newcommand\PX{ {{\partial}\over{\partial X}} }
\newcommand\PY{ {{\partial}\over {\partial Y}} }
\newcommand\PU{ {{\partial}\over{\partial U}} }
\newcommand\PV{ {{\partial}\over{\partial V}} }
\newcommand\PO{ {{\partial}\over{\partial \Omega}} }
\newcommand\PZ{ {{\partial}\over{\partial Z}} }
\newcommand\PW{ {{\partial}\over{\partial W}} }
\newcommand\PWB{ {{\partial}\over {\partial{\overline W}}} }
\newcommand\OVW{\overline W}

\newcommand\PR{{\mathcal P}_g\times {\mathbb R}^{(h,g)}}
\newcommand\Rmn{{\mathbb R}^{(h,g)}}
\newcommand\Gnm{GL_{g,h}}

\begin{section}{{\large\bf Introduction}}
\setcounter{equation}{0}
\vskip 0.3cm
For a given fixed positive integer $g$,
we let
$${\mathbb H}_g=\,\{\,\Omega\in \BC^{(g,g)}\,|\ \Om=\,^t\Om,\ \ \ \text{Im}\,\Om>0\,\}$$
be the Siegel upper half plane of degree $g$ and let
$$Sp(g,\BR)=\{ M\in \BR^{(2g,2g)}\ \vert \ ^t\!MJ_gM= J_g\ \}$$
be the symplectic group of degree $g$, where $F^{(k,l)}$ denotes
the set of all $k\times l$ matrices with entries in a commutative
ring $F$ for two positive integers $k$ and $l$, $^t\!M$ denotes
the transpose matrix of a matrix $M$ and
$$J_g=\begin{pmatrix} 0&I_g\\
                   -I_g&0\end{pmatrix}.$$
Then $Sp(g,\BR)$ acts on $\BH_g$ transitively by
\begin{equation}
M\cdot\Om=(A\Om+B)(C\Om+D)^{-1},
\end{equation} where $M=\begin{pmatrix} A&B\\
C&D\end{pmatrix}\in Sp(g,\BR)$ and $\Om\in \BH_g.$ Let
$$\G_g=Sp(g,\BZ)=\left\{ \begin{pmatrix} A&B\\
C&D\end{pmatrix}\in Sp(g,\BR) \,\big| \ A,B,C,D\
\textrm{integral}\ \right\}$$ be the Siegel modular group of
degree $g$. This group acts on $\BH_g$ properly discontinuously.
C. L. Siegel investigated the geometry of $\BH_g$ and automorphic
forms on $\BH_g$ systematically. Siegel\,\cite{Sieg1} found a
fundamental domain ${\mathcal F}_g$ for $\G_g\ba\BH_g$ and
described it explicitly. Moreover he calculated the volume of
$\CCF.$ We also refer to \cite{Igu},\,\cite{Ma2},\,\cite{Sieg1} for
some details on $\CCF.$ Siegel's fundamental domain is now called
the Siegel modular variety and is usually denoted by ${\Cal A}_g$.
In fact, ${\Cal A}_g$ is one of the important arithmetic varieties
in the sense that it is regarded as the moduli of principally
polarized abelian varieties of dimension $g$. Suggested by Siegel,
I. Satake \cite{Sat1} found a canonical compactification, now
called the Satake compactification of ${\Cal A}_g$. Thereafter W.
Baily \cite{B1} proved that the Satake compactification of ${\Cal
A}_g$ is a normal projective variety. This work was generalized to
bounded symmetric domains by W. Baily and A. Borel \cite{BB}
around the 1960s. Some years later a theory of smooth
compactification of bounded symmetric domains was developed by
Mumford school \cite{AMRT}. G. Faltings and C.-L. Chai \cite{FC}
investigated the moduli of abelian varieties over the integers and
could give the analogue of the Eichler-Shimura theorem that
expresses Siegel modular forms in terms of the cohomology of local
systems on ${\Cal A}_g$. I want to emphasize that Siegel modular
forms play an important role in the theory of the arithmetic and
the geometry of the Siegel modular variety ${\Cal A}_g$.

\vskip 0.2cm
We let
$$\CP=\left\{\, Y\in\BR^{(g,g)}\,|\ Y=\,^tY>0\ \right\}$$
be an open convex cone in $\BR^N$ with $N=g(g+1)/2.$ The general linear
group $GL(g,\BR)$ acts on $\CP$ transitively by
\begin{equation}
A\circ Y:=\,AY\,^tA,\qquad A\in GL(g,\BR),\ Y\in \CP.
\end{equation}
We observe that the action (1.2) is naturally induced from the symplectic action (1.1).
Thus $\CP$ is a symmetric space diffeomorphic to $GL(g,\BR)/O(g).$
Let
$$GL(g,\BZ)=\,\left\{\,\g\in GL(g,\BR)\,|\ \g\ \textrm{is integral}\,\right\}$$
be an arithmetic discrete subgroup of $GL(g,\BR).$ Using the reduction theory
Minkowski \cite{Mink} found a fundamental domain ${\mathfrak R}_g$, the so-called
Minkowski domain for the action (1.2) of $GL(g,\BZ)$ on $\CP$. In fact, using the
Minkowski domain ${\mathfrak R}_g$ Siegel found his fundamental domain ${\mathcal F}_g.$
As in the case of $\BH_g$, automorphic forms on $\CP$ for $GL(g,\BZ)$ and geometry on $\CP$
have been studied by many people, e.g., Selberg \cite{Sel}, Maass \cite{Ma2} et al.

\vskip 0.2cm The aim of this article is to study arithmetic-geometric meaning of the
Minkowski domain ${\mathfrak R}_g$. First we introduce a notion of polarized real tori by
relating special real tori to polarized real abelian varieties. We realize that $\CP$ parametrizes
principally polarized real tori of dimension $g$ and also that ${\mathfrak R}_g$ may be regarded
as a moduli space of principally polarized real tori of dimension $g$. We also study smooth line bundles
over a polarized real torus by relating to holomorphic line bundles over the associated polarized abelian
variety. Those line bundles over a polarized real torus play an important role in
investigating some geometric properties of a polarized real torus.

\vskip 0.2cm We let
\begin{equation*}
G^M:=\,GL(g,\BR)\ltimes \BR^g
\end{equation*}
be the semidirect product of $GL(g,\BR)$ and $\BR^g$ with multiplication law
\begin{equation*}
(A,a)\cdot (B,b):=\,(AB,a\,{}^tB^{-1}+b),\qquad A,B\in GL(g,\BR),\ \ a,b\in \BR^g.
\end{equation*}
Then we have the {\it natural action} of $G^M$ on the Minkowski-Euclid space $\CP\times \BR^g$
defined by
\begin{equation}
(A,a)\cdot (Y,\zeta):=\,\big(AY\,{}^t\!A,\,(\zeta+a)\,{}^t\!A \big), \qquad (A,a)\in G^M,\ Y\in\CP,\ \zeta\in \BR^g.
\end{equation}
We let
\begin{equation*}
G^M(\BZ)\,=\,GL(g,\BZ)\ltimes \BZ^g
\end{equation*}
be the discrete subgroup of $G^M.$ Then $G^M(\BZ)$ acts on $\CP\times \BR^g$ properly discontinuously.
We show that by associating a principally polarized real torus of dimension $g$ to each equivalence class
in ${\mathfrak R}_g$,
the quotient space
\begin{equation*}
G^M(\BZ)\backslash \big(\CP\times \BR^g \big)
\end{equation*}
may be regarded as a family of principally polarized real tori of dimension $g$. To each equivalence class
$[Y]\in GL(g,\BZ)\ba \CP$ with $Y\in \CP$ we associate a principally polarized real torus $T_Y=\,\BR^g/\La_Y$, where
$\La_Y=\,Y\BZ^g$ is a lattice in $\BR^g$.

\vskip 0.2cm Let $Y_1$ and $Y_2$ be two elements in $\CP$ with $[Y_1]\neq [Y_2]$, that is,
$Y_2\neq A\,Y_1\,{}^t\!A$ for all $A\in GL(g,\BZ).$ We put $\La_i\,=\,Y_i\,\BZ^g$ for $i=1,2.$
Then a torus $T_1\,=\,\BR^g/\La_1$ is diffeomorphic to $T_2\,=\,\BR^g/\La_2$ as smooth manifolds but $T_1$
is not isomorphic to $T_2$ as polarized real tori.

\vskip 0.2cm The Siegel modular variety ${\Cal A}_g$ has three remarkable properties\,: (a) it is the
moduli space of principally polarized abelian varieties of dimension $g$, (b) it has the structure of
a quasi-projective complex algebraic variety which is defined over $\BQ$, and (c) it has a canonical
compactification, the so-called Satake-Baily-Borel compactification which is defined over $\BQ$.
Unfortunately the Minkowski domain ${\mathfrak R}_g$ does not admit the structure of a real algebraic
variety. Moreover ${\mathfrak R}_g$ does not admit a compactification which is defined over $\BQ$.
Silhol \cite{Si3} constructs the moduli space of real principally poarized abelian varieties and he
shows that it is a topological ramified covering of ${\mathfrak R}_g$. Furthermore Silhol constructs
a compactification of this moduli space analogous to the Satake-Baily-Borel compactification. However,
neither the moduli space nor this compactification has an algebraic structure. On the other hand, by
considering real abelian varieties with a suitable level structure Goresky and Tai \cite{GT1} shows that
the moduli space of real principally polarized abelian varieties with level $4m$ structure $(m\geq 1)$
coincides with the set of real points of a quasi-projective algebraic variety defined over $\BQ$ and
consists of finitely many copies of the quotient ${\mathfrak G}_g(4m)\ba \CP$ with a discrete
subgroup ${\mathfrak G}_g(4m)$ of $GL(g,\BZ)$, where
${\mathfrak G}_g(4m)=\,\{ \gamma\in GL(g,\BZ)\,|\ \gamma \equiv I_g \ (\textrm{mod}\,4m)\,\}$.

\vskip 0.2cm
This paper is organized as follows. In Section 2, we collect some basic properties about the symplectic group
$Sp(g,\BR)$ to be used frequently in the subsequent sections. In Section 3, we give basic definitions concerning real
abelian varieties and review some properties of real abelian varieties. In Section 4, we discuss a moduli space for
real abelian varieties and recall some basic properties of a moduli for real abelian varieties. In Section 5 we discuss
compactifications of the moduli space for real abelian varieties and review some results on this moduli space obtained by
Silhol\,\cite{Si3}, Goresky and Tai \cite{GT1}. In Section 6 we introduce a notion of polarized real tori and investigate some
properties of polarized real tori. We give several examples of polarized real tori. In Section 7 we study smooth line
bundles over a real torus, in particular a polarized real torus by relating those smooth line bundles to holomorphic line
bundles over the associated complex torus. To each smooth line bundle on a real torus we naturally attach a
holomorphic line bundle over the associated complex torus. Conversely to a holomorphic line bundle over a polarized
abelian variety we associate a smooth line bundle over the associated polarized real torus. Using these results on
line bundles, we embed a real torus in a complex projective space and hence in a real projective space smoothly.
We also review briefly holomorphic line bundles over a complex torus. In Section 8 we study the moduli space for
polarized real tori. We first review basic geometric properties on the Minkowski domain ${\mathfrak R}_g$.
We show that $\CP$ parameterizes principally polarized real tori of dimension $g$ and that ${\mathfrak R}_g$ can be regarded
as the moduli space of principally polarized real tori of dimension $g$. We show that the quotient space
$G^M(\BZ)\ba (\CP\times \BR^g)$ may be considered as a family of principally polarized tori of dimension $g$.
In Section 9 we discuss real semi-abelian
varieties corresponding to the boundary points of a compactification of a moduli space for real abelian
varieties. We recall that a semi-abelian variety is defined to be an extension of an abelian variety by
a group of multiplicative type. In Section 10 we discuss briefly real semi-tori corresponding to the
boundary points of a moduli space for polarized real tori.
In the final section we present some problems related to real polarized tori which should be investigared in the near future.
In the appendix we collect and review some results on non-abelian cohomology to be needed necessarily in this article.
We give some sketchy proofs for the convenience of the reader.

\vskip 0.3cm Finally I would like to mention that this work was motivated and initiated by the works of Silhol
\cite{Si3} and Goresky-Tai \cite{GT1}.

\vskip 0.10cm

\vskip 0.31cm \noindent {\bf Notations:} \ \ We denote by
$\BQ,\,\BR$ and $\BC$ the field of rational numbers, the field of
real numbers and the field of complex numbers respectively. We
denote by $\BZ$ and $\BZ^+$ the ring of integers and the set of
all positive integers respectively. The symbol ``:='' means that
the expression on the right is the definition of that on the left.
For two positive integers $k$ and $l$, $F^{(k,l)}$ denotes the set
of all $k\times l$ matrices with entries in a commutative ring
$F$. For a square matrix $A\in F^{(k,k)}$ of degree $k$,
$\sigma(A)$ denotes the trace of $A$. For any $M\in F^{(k,l)},\
^t\!M$ denotes the transpose matrix of $M$. $I_n$ denotes the
identity matrix of degree $n$. For a matrix $Z$, we denote by $\textrm{Re}\,Z$
(resp. $\textrm{Im}\,Z$) the real (resp. imaginary) part of $Z$.
For $A\in F^{(k,l)}$ and $B\in
F^{(k,k)}$, we set $B[A]=\,^tABA.$ For a complex matrix $A$,
${\overline A}$ denotes the complex {\it conjugate} of $A$. For
$A\in \BC^{(k,l)}$ and $B\in \BC^{(k,k)}$, we use the abbreviation
$B\{ A\}=\,^t{\overline A}BA.$ We denote $\BC^*_1\,=\,\{ \,\xi\in\BC\,|\
|\xi|=1\,\}.$
Let
\begin{equation*}
\G_g=\,\left\{\,\g\in \BZ^{(2g,2g)}\,|\ \,{}^t\g\, J_g\g=J_g\,\right\}
\end{equation*}
denote the Siegel modular group of degree $g$, where
\begin{equation*}
J_g=\begin{pmatrix} 0&I_g\\
                   -I_g&0\end{pmatrix}
\end{equation*}
is the symplectic matrix of degree $2g$. For a positive integer $N$, we let
$$\G_g(N)=\,\left\{\,\g\in \G_g\,|\ \g\equiv I_{2g}\,({\rm mod}\ N)\,\right\}$$
denote the principal congruence subgroup of $\G_g$ of level $N$ and for a positive integer $m$, we let
\begin{equation}
\G_g(2,2m)=\,\left\{  \begin{pmatrix} A& B\\
                   C & D\end{pmatrix}\in \G_g\ \Big|\ \ A,D\equiv I_g\ ({\rm mod}\ 2),\quad
                   B,C\equiv 0\ ({\rm mod}\ 2m)\ \right\}.
\end{equation}
Let ${\mathfrak G}_g:=GL(g,\BZ)$ and for a positive integer $N$ let
\begin{equation}
{\mathfrak G}_g (N) =\,\left\{\,\g\in GL(g,\BZ)\,|\ \g\equiv I_{g}\,({\rm mod}\ N)\,\right\}.
\end{equation}

\end{section}

\vskip 1cm

\begin{section}{{\large\bf The Symplectic Group}}
\setcounter{equation}{0}

For a given fixed positive integer $g$,
we let
let
$$Sp(g,\BR)=\big\{\, M\in \BR^{(2g,2g)}\ \vert \ ^t\!MJ_gM= J_g\ \big\}$$
be the symplectic group of degree $g$.

\vskip 0.2cm If $M=\begin{pmatrix} A&B\\
C&D\end{pmatrix}\in Sp(g,\BR)$ with $A,B,C,D\in \BR^{(g,g)}$, then it is easily seen that
\begin{equation}
A\,{}^tD-B\,^tC=I_g,\quad A\,{}^tB=B\,{}^tA,\quad C\,{}^tD=D\,{}^tC
\end{equation}
or
\begin{equation}
{}^tAD-{}^tC B=I_g,\quad {}^tAC=\,{}^tCA,\quad {}^tB D={}^tDB.
\end{equation}
The inverse of such a symplectic matrix $M$ is given by
\begin{equation*}
M^{-1}=M=\begin{pmatrix} {}^tD & -{}^tB\\
-{}^tC &  {}^tA\end{pmatrix}.
\end{equation*}
We identify $GL(g,\BR)\hookrightarrow Sp(g,\BR)$ with its image under the embedding
$$ A\longmapsto \begin{pmatrix} A & 0\\
                   0 & {}^tA^{-1}  \end{pmatrix},\quad A\in GL(g,\BR).$$
A Cartan involution $\theta$ of $Sp(g,\BR)$ is given by $\theta (x)= J_g \,x \,J_g^{-1},
\ x\in Sp(g,\BR),$ in other words,
\begin{equation}
\theta \begin{pmatrix} A&B\\
C&D\end{pmatrix}= \begin{pmatrix} \ D & -C\\
-B&\ A\end{pmatrix},\qquad \begin{pmatrix} A&B\\
C&D\end{pmatrix}\in Sp(g,\BR).
\end{equation}
The fixed point set $K$ of $\theta$ is given by
$$K=\left\{ \begin{pmatrix} \ A&B\\
-B&A\end{pmatrix}\in Sp(g,\BR)\ \right\}.$$
We may identify $K$ with the unitary group $U(g)$ of degree $g$ via
$$ K\ni \begin{pmatrix} \ A&B\\
-B&A\end{pmatrix} \longmapsto A+\,iB\in U(g).$$

Let
$${\mathbb H}_g=\,\big\{\,\Omega\in \BC^{(g,g)}\,|\ \Om=\,^t\Om,\ \ \ \text{Im}\,\Om>0\,\big\}$$
be the Siegel upper half plane of degree $g$. Then $Sp(g,\BR)$ acts on $\BH_g$ transitively by
\begin{equation}
M\cdot\Om=(A\Om+B)(C\Om+D)^{-1},
\end{equation} where $M=\begin{pmatrix} A&B\\
C&D\end{pmatrix}\in Sp(g,\BR)$ and $\Om\in \BH_g.$
The stabilizer at $iI_g$ is given by the compact subgroup $K\cong U(g)$ of $Sp(g,\BR).$
Thus $\BH_g$ is biholomorphic to the Hermitian symmetric space $Sp(g,\BR)/K$ via
$$ Sp(g,\BR)/K \longrightarrow \BH_g,\qquad xK \longmapsto x\cdot (iI_g),\ x\in Sp(g,\BR).$$
\noindent
We note that the Siegel modular group  $\G_g$ of
degree $g$ acts on $\BH_g$ properly discontinuously.

\vskip 0.25cm Now we let
\begin{equation}
I_*:= \begin{pmatrix} -I_g & 0\\
\ 0 & I_g \end{pmatrix}.
\end{equation}
We define the involution $\tau: Sp(g,\BR)\lrt Sp(g,\BR)$ by
\begin{equation}
\tau (x): =\,I_*\, x\,I_*,\qquad x\in Sp(g,\BR).
\end{equation}
Precisely $\tau$ is given by
\begin{equation}
\tau \begin{pmatrix} A & B\\
C & D \end{pmatrix}= \begin{pmatrix} \ A & -B\\
-C & \ D \end{pmatrix}, \qquad \begin{pmatrix} A & B\\
C & D \end{pmatrix}\in Sp(g,\BR).
\end{equation}

\begin{lemma} (1) $\tau(x)=x,\ x\in Sp(g,\BR)$ if and only if $x\in GL(g,\BR).$\\
(2) $\tau\theta =\theta\tau.$ So $\tau (K)=K.$\\
(3) If $A+\,iB\in U(g)$ with $A,B\in\BR^{(g,g)},$ then $\tau (A+iB)=A-iB.$
\end{lemma}
\noindent {\it Proof.} It is easy to prove the above lemma. We leave the proof to the reader.
\hfill $\square$

\vskip 0.2cm We note that $\tau: Sp(g,\BR)\lrt Sp(g,\BR)$ passes to an involution (which we denote by the same
letter) $\tau:\BH_g\lrt\BH_g$ such that
\begin{equation}
\tau(x\cdot\Omega)=\,\tau(x)\,\tau(\Omega)\qquad \textrm{for all}\ x\in Sp(g,\BR),\ \Omega\in\BH_g.
\end{equation}
In fact, we can see easily that the involution $\tau:\BH_g\lrt\BH_g$ is the antiholomorphic involution given by
\begin{equation}
\tau (\Omega)=-{\overline\Omega},\qquad \Omega\in\BH_g.
\end{equation}
\noindent
Its fixed point set is the orbit
$$i\,{\mathcal P}_g=\,GL(g,\BR)\cdot (iI_g)\subset \BC^{(g,g)}$$
of $GL(g,\BR)$, where
$$ {\mathcal P}_g=\left\{ \,Y\in\BR^{(g,g)}\,|\ Y=\,{}^tY > 0\ \right\}$$
\noindent
is the open convex cone of positive definite symmetric real matrices of degree $g$ in the Euclidean space
$\BR^{g(g+1)/2}.$

\vskip 0.2cm For $x\in Sp(g,\BR)$ and $\Omega\in\BH_g$, we define the set
\begin{equation}
{\mathbb H}_g^{\tau x}:=\,\left\{\, \Omega\in\BH_g\,\,|\ \,x\cdot\Omega=\,\tau(\Omega)=\,-{\overline \Omega}\ \right\}
\end{equation}
be the locus of $x$-real points. If $\Gamma \subset Sp(g,\BR)$ is an arithmetic subgroup of $Sp(g,\BR)$
such that $\tau (\Gamma)=\Gamma$, we define
\begin{equation}
{\mathbb H}_g^{\tau \Gamma}:=\,\bigcup_{\gamma\in\Gamma} {\mathbb H}_g^{\tau\gamma}.
\end{equation}

\begin{lemma}
Let $x\in Sp(g,\BR)$ and ${\mathbb H}_g^x$ be the set of points in $\BH_g$ which are fixed under the action of $x$.
Then the set ${\mathbb H}_g^x \cap i\,{\mathcal P}_g$ is a proper real algebraic variety of $i\,{\mathcal P}_g$
if $x\neq \pm I_g\in GL(g,\BR).$
\end{lemma}
\noindent {\it Proof.} It is easy to prove the above lemma. We omit the proof. \hfill $\square$

\end{section}

\vskip 1cm

\begin{section}{{\large\bf Real Abelian Varieties}}
\setcounter{equation}{0}

\vskip 0.3cm In this section we review basic notions and some results on real principally polarized abelian varieties
(cf.\,\cite{GT1, S-S, Si1, Si2, Si3}).

\begin{definition}
A pair $(\FA,S)$ is said to be a \textsf{real abelian variety} if $\FA$ is a complex abelian variety and $S$ is an
anti-holomorphic involution of $\FA$ leaving the origin of $\FA$ fixed. The set of all fixed points of $S$ is called the
{\it real point} of $(\FA,S)$ and denoted by $(\FA,S)(\BR)$ or simply $\FA(\BR)$. We call $S$ a
\textsf{real structure} on $\FA$.
\end{definition}

\begin{definition}
(1) A \textsf{polarization} on a complex abelian variety $\FA$ is defined to be the Chern class
$c_1(D)\in H^2(\FA,\BZ)$ of an ample divisor $D$ on $\FA$. We can identify $H^2(\FA,\BZ)$ with
$\bigwedge^2 H^1(\FA,\BZ)$. We write $\FA=V/L$, where $V$ is a finite dimensional complex vector space and $L$ is
a lattice in $V$. So a polarization on $\FA$ can be defined as an alternating form $E$ on $L\cong H_1(\FA,\BZ)$
satisfying the following conditions (E1) and (E2)\,:
\vskip 0.1cm \noindent
(E1) The Hermitian form $H:V\times V\lrt \BC$ defined by
\begin{equation}
H(u,v)=\,E(i\,u,v)\,+\,i\, E(u,v),\qquad u,v\in V
\end{equation}
\noindent is positive definite. Here $E$ can be extended $\BR$-linearly to an alternating form on $V$.
\vskip 0.1cm \noindent
(E2) $E(L\times L)\subset \BZ$, i.e., $E$ is integral valued on $L\times L.$
\vskip 0.2cm\noindent
(2) Let $(\FA,S)$ be a real abelian variety with a polarization $E$ of dimension $g$.
A polarization $E$ is said to be
\textsf{real} or $S$-\textsf{real} if
\begin{equation}
E(S_*(a),S_*(b))=\,-E(a,b),\qquad a,b\in H_1(\FA,\BZ).
\end{equation}
Here $S_*:H_1(\FA,\BZ) \lrt H_1(\FA,\BZ)$ is the map induced by a real structure $S$.
If a polarization $E$ is real, the triple $(\FA,E,S)$ is called a \textsf{real polarized abelian variety}.
A polarization $E$ on $\FA$ is said to be \textsf{principal} if for a suitable basis (i.e., a symplectic
basis) of $H_1(\FA,\BZ)\cong L$, it is represented by the symplectic matrix $J_g$ (cf.\,see Notations in the
introduction).
A real abelian variety $(\FA,S)$ with a principal polarization $E$ is called a
\textsf{real principally polarized abelian variety}.

\vskip 0.2cm\noindent
(3) Let $(\FA,E)$ be a principally polarized abelian variety of dimension $g$ and let
$\{ \alpha_i\,|\ 1\leq i\leq 2g\,\}$ be a symplectic basis of $H_1(\FA,\BZ)$. It is known that there is
a basis $\{ \omega_1,\cdots,\omega_g \}$ of the vector space $H^0(\FA,\Omega^1)$ of holomorphic 1-forms on $\FA$
such that
\begin{equation*}
\left(\, \int_{\alpha_j}\omega_i\,\right) =\,(\Omega,I_g)\qquad \textrm{for some}\ \Omega\in\BH_g.
\end{equation*}
The $g\times 2g$ matrix $(\Omega,I_g)$ or simply $\Omega$ is called a \textsf{period matrix} for $(\FA,E).$
\end{definition}

\vskip 0.3cm The definition of a {\it real polarized abelian variety} is motivated by the following theorem.
\begin{theorem}
Let $(\FA,S)$ be a real abelian variety and let $E$ be a polarization on $\FA$. Then there exists an ample
$S$-invariant (or $S$-real) divisor with Chern class $E$ if and only if $E$ satisfies the condition (3.2).
\end{theorem}
\noindent
{\it Proof.} The proof can be found in \cite[Theorem 3.4, pp.\,81-84]{Si2}. \hfill $\square$

\vskip 0.3cm
Now we consider a principally polarized abelian variety of dimension $g$ with a level structure.
Let $N$ be a positive integer. Let $(\FA=\,\BC^g/L,E)$ be a principally polarized abelian variety of dimension $g$.
From now on we write $\FA=\,\BC^g/L$, where $L$ is a lattice in $\BC^g$. A \textsf{level} $N$
\textsf{structure} on $\FA$ is a choice of a basis $\{ U_i,V_j\}\,(1\leq i,j\leq g)$ for a $N$-torsion points of $\FA$
which is symplectic, in the sense that there exists a symplectic basis $\{ u_i,v_j\}$ of $L$ such that
\begin{equation*}
U_i\equiv {{u_i}\over N}\ (\textrm{mod}\,L)\qquad \textrm{and}\qquad
V_j\equiv {{v_j}\over N}\ (\textrm{mod}\,L),\qquad 1\leq i,j\leq g.
\end{equation*}
For a given level $N$ structure, such a choice of a symplectic basis $\{ u_i,v_j\}$ of $L$ determines a
mapping
$$F:\BR^g\oplus \BR^g\lrt \BC^g$$
such that $F(\BZ^g\oplus \BZ^g)=\,L$ by $F(e_i)=u_i$ and $F(f_j)=v_j$, where $\{e_i,f_j\}\,(1\leq i,j\leq g)$
is the standard basis of $\BR^g\oplus\BR^g.$ The choice $\{ u_i,v_j\}$ (or equivalently, the mapping $F$) will
be referred to as a {\it lift} of the level $N$ structure.
Such a mapping $F$ is well defined modulo the principal
congruence subgroup $\G_g(N)$, that is, if $F'$ is another lift of the level structure, then
$F'\circ F^{-1}\in \G_g(N).$ A level $N$ structure $\{ U_i,V_j\}$ is said to be $ \textsf{compatible}$ with
a real structure $S$ on $(\FA,E)$ if, for some (and hence for any) lift $\{ u_i,v_j\}$ of the level structure,
\begin{equation*}
S\left( {{u_i}\over N}\right)\equiv -{{u_i}\over N}\ (\textrm{mod}\,L)\qquad \textrm{and}\qquad
S \left( {{v_j}\over N}\right) \equiv {{v_j}\over N}\ (\textrm{mod}\,L),\qquad 1\leq i,j\leq g.
\end{equation*}

\begin{definition}
A real principally polarized abelian variety of dimension $g$ with a level $N$ structure is a quadruple
${\mathcal A}=\,(\FA,E,S, \{ U_i,V_j\})$ with $\FA=\BC^g/L$, where $(\FA,E,S)$ is a real principally polarized abelian variety
and $\{ U_i,V_j\}$ is a level $N$ structure compatible with a real structure $S$. An isomorphism
$${\mathcal A}=\,(\FA,E,S, \{ U_i,V_j\})\,\cong\,(\FA',E',S', \{ U_i',V_j'\})={\mathcal A}'$$
is a complex linear mapping $\phi:\BC^g\lrt \BC^g$ such that
\begin{equation}
\phi(L)=\,L',
\end{equation}
\begin{equation}
\phi_*(E)=\,E',
\end{equation}
\begin{equation}
\phi_*(S)=\,S',\ that\ is,\ \ \phi\circ S\circ \phi^{-1}=S',
\end{equation}
\begin{equation}
\phi \left( {{u_i}\over N}\right)\equiv {{u_i'}\over N}\ (\textrm{mod}\,L')\qquad \textrm{and}\qquad
\phi \left( {{v_j}\over N}\right) \equiv {{v_j'}\over N}\ (\textrm{mod}\,L'),\qquad 1\leq i,j\leq g.
\end{equation}
\noindent for some lift $\{ u_i,v_j\}$ and $\{ u_i',v_j'\}$ of the level structures.
\end{definition}

Now we show that a given positive integer $N$ and a given $\Omega\in\BH_g$ determine naturally a
principally polarized abelian variety $(\FA_\Om,E_\Om)$ of dimension $g$ with a level $N$ structure.
Let $E_0$ be the standard alternating form on $\BR^g\oplus\BR^g$ with the symplectic matrix $J_g$
with respect to the standard basis of $\BR^g\oplus\BR^g$. Let $F_\Om:\BR^g\oplus\BR^g\lrt \BC^g$
be the real linear mapping with matrix $(\Om,I_g)$, that is,
\begin{equation}
F_\Om \begin{pmatrix} x \\ y \end{pmatrix}:=\Om\, x+y,\qquad x,y\in\BR^g.
\end{equation}
We define $E_\Om:=\,(F_\Om)_*(E_0)$ and $L_\Om:=\,F_\Om (\BZ^g\oplus\BZ^g).$ Then
$(\FA_\Om=\,\BC^g/L_\Om,E_\Om)$ is a principally polarized abelian variety. The Hermitian form
$H_\Om$ on $\BC^g$ corresponding to $E_\Om$ is given by
\begin{equation}
H_\Om (u,v)=\, {}^tu\,( \textrm{Im}\,\Om)^{-1}\,{\overline v},\quad E_\Om= \textrm{Im}\,H_\Om,
\qquad u,v\in\BC^g.
\end{equation}
\noindent
If $z_1,\cdots,z_g$ are the standard coordinates on $\BC^g$, then the holomorphic 1-forms $dz_1,\cdots,
dz_g$ have the period matrix $(\Om,I_g).$ If $\{e_i,f_j\}$ is the standard basis of $\BR^g\oplus\BR^g$,
then $\left\{ F_\Om(e_i/N),\,F_\Om(f_j/N)\right\}$\\
(mod\ $L_\Om$) is a level $N$ structure on
$(\FA_\Om,E_\Om)$, which we refer to as the {\it standard} $N$ {\it structure}. Assume that $\Om_1$ and
$\Om_2$ are two elements of $\BH_g$ such that
$$\psi: (\FA_{\Om_1}=\BC^g/L_{\Om_1},E_{\Om_1})\lrt (\FA_{\Om_2}=\BC^g/L_{\Om_2},E_{\Om_2})$$
is an isomorphism of the corresponding principally polarized abelian varieties, i.e.,
$\psi(L_{\Om_1})=L_{\Om_2}$ and $\psi_*(E_{\Om_1})=\,E_{\Om_2}.$ We set
$$ h=\,{}^t\big( F_{\Om_2}^{-1}\circ \psi\circ F_{\Om_1}\big)=\,
\begin{pmatrix} A & B\\
C & D \end{pmatrix}.$$
Then we see that $h\in \G_g$. And we have
\begin{equation}
\Om_1=\,h\cdot\Om_2 =\,(A\Om_2+B)(C\Om_2+D)^{-1}
\end{equation}
and
\begin{equation}
\psi(Z)=\,{}^t(C\Om_2+D)Z,\qquad Z\in\BC^g.
\end{equation}

\vskip 0.2cm
Let $\Om\in\BH_g$ such that $\gamma\cdot \Om=\,\tau(\Om)=\,-{\overline \Om}$ for some
$\gamma=\begin{pmatrix} A & B\\
C & D \end{pmatrix}\in \G_g.$ We define the mapping $S_{\gamma,\Om}:\BC^g\lrt\BC^g$ by
\begin{equation}
S_{\gamma,\Om}(Z):=\,{}^t(C\Om+D)\,{\overline Z},\qquad Z\in\BC^g.
\end{equation}
Then we can show that $S_{\gamma,\Om}$ is a real structure on $(\FA_\Om,E_\Om)$ which is compatible
with the polarization $E_\Om$ (that is, $E_\Om(S_{\gamma,\Om}(u),S_{\gamma,\Om}(v))=-E_\Om(u,v)$
for all $u,v\in \BC^g$). Indeed according to Comessatti's Theorem (see Theorem 3.1), $S_{\gamma,\Om}(Z)=\,
{\overline Z},$ i.e., $S_{\gamma,\Om}$ is a complex conjugation. Therefore we have
$$E_\Om(S_{\gamma,\Om}(u),S_{\gamma,\Om}(v))=\,E_\Om({\overline u},{\overline v})=-E_\Om(u,v)$$
for all $u,v\in \BC^g$. From now on we write simply $\s_\Om=\,S_{\gamma,\Om}.$

\begin{theorem} Let $(\FA,E,S)$ be a real principally polarized abelian variety of
dimension $g$. Then there exists $\Om=\,X +\, i\,Y \in \BH_g$ such that $2X\in \BZ^{(g,g)}$ and there
exists an isomorphism of real principally polarized abelian varieties
$$(\FA,E,S)\,\cong\,(\FA_\Om,E_\Om,\sigma_\Om),$$
where $\sigma_\Om$ is a real structure on $\FA_\Om$ induced by a complex conjugation $\sigma:\BC^g\lrt\BC^g$.
\end{theorem}

The above theorem is essentially due to Comessatti \cite{C}. We refer to \cite{Si1, Si2} for the proof of
Theorem 3.2.

\vskip 0.5cm Theorem 3.2 leads us to define the subset ${\mathscr H}_g$ of $\BH_g$ by
\begin{equation}
{\mathscr H}_g:=\,\left\{\,\Om\in\BH_g\,\,|\ \,2\,\textrm{Re}\,\Om \in \BZ^{(g,g)}\ \right\}.
\end{equation}
Assume $\Om=\,X+\,i\,Y\in {\mathscr H}_g.$ Then according to Theorem 3.2, $(\FA_\Om,E_\Om,\sigma_\Om)$ is a real
principally polarized abelian variety of dimension $g$. The matrix $M_\s$ for the action of a complex conjugation
$\s$ on the lattice $L_\Om=\Om\BZ^g+\BZ^g$ with respect to the basis given by the columns of $(\Om,I_g)$ is given by
\begin{equation}
M_\s=\,\begin{pmatrix} -I_g & 0 \\ 2\,X & I_g \end{pmatrix}.
\end{equation}
Since
$${}^tM_\s\, J_g\, M_\s=\,\begin{pmatrix} -I_g & 2\,X \\ 0 & I_g \end{pmatrix}\,J_g\,\begin{pmatrix} -I_g & 0 \\ 2\,X & I_g \end{pmatrix}
=-J_g,$$
the canonical polarization $J_g$ is $\s$-real.

\begin{theorem}
Let $\Om$ and $\Om_*$ be two elements in ${\mathscr H}_g$. Then $\Om$ and $\Om_*$ represent (real) isomorphic triples
$(\FA,E,\s)$ and $(\FA_*,E_*,\s_*)$ if and only if there exists an element $A\in GL(g,\BZ)$ such that
\begin{equation}
2\, \textrm{Re}\,\,\Om_*=\,2\, A\,(\textrm{Re}\,\,\Om)\,\,{}^tA\ \ (\textrm{mod}\,\,2)
\end{equation}
and
\begin{equation}
\textrm{Im}\,\,\Om_*=\,A\,(\textrm{Im}\,\,\Om)\,\,{}^tA.
\end{equation}
\end{theorem}
\noindent
{\it Proof.} Suppose $(\FA,E,\s)$ and $(\FA_*,E_*,\s_*)$ are real isomorphic. Then we can find an element
$\gamma=\begin{pmatrix} A & B\\
C & D \end{pmatrix}\in \G_g$ such that
$$\Om_*=\,(A\Om+B)(C\Om+D)^{-1}.$$
The map
$$\varphi: \BC^g/L_{\Om_*}=\,\FA_{\Om_*}\lrt \FA_\Om=\,\BC^g/L_\Om$$
induced by the map
$${\widetilde\varphi}:\BC^g\lrt\BC^g,\qquad Z\longmapsto \,{}^t(C\Om+D)\,Z$$
is a real isomorphism. Since ${\widetilde\varphi}\circ \s_* =\,\s\circ {\widetilde\varphi},$ i.e.,
${\widetilde\varphi}$ commutes with complex conjugation on $\BC^g$, we have $C=0.$ Therefore
$$\Om_*=\,(A\Om+B)\,{}^t\!A=\,(AX\,{}^t\!A+B\,{}^t\!A)+\,i\,AY\,{}^t\!A,$$
where $\Om=\,X+\,i\,Y.$ Hence we obtain the desired results (3.14) and (3.15).
\vskip 0.1cm Conversely we assume that there exists $A\in GL(g,\BZ)$ satisfying the conditions
(3.14) and (3.15). Then
$$\Om_*=\,\gamma\cdot \Om=\,(A\Om+B)\,{}^t\!A$$
for some $\gamma=\begin{pmatrix} A & B\\
0 & {}^tA^{-1} \end{pmatrix}\in \G_g$ with $B\in\BZ^{(g,g)}$ with $B\,{}^tA=\,A\,{}^tB.$ The map
$\psi:\,\FA_\Om\lrt \FA_{\Om_*}$ induced by the map
$${\widetilde\psi}:\BC^g\lrt \BC^g,\qquad Z\longmapsto A^{-1}Z$$
is a complex isomorphism commuting complex conjugation $\s$. Therefore $\psi$ is a real isomorphism
of $(\FA,E,\s)$ onto $(\FA_*,E_*,\s_*)$. \hfill $\square$

\vskip 0.3cm
According to Theorem 3.3, we are led to define the subgroup $\G_g^{\star}$ of $\G_g$ by
\begin{equation}
\G_g^{\star}:=\,\left\{\, \begin{pmatrix} A & B\\
0 & {}^tA^{-1} \end{pmatrix}\in \G_g\ \big|\ \ B\in \BZ^{(g,g)},\quad A\,{}^tB=\,B\,{}^t\!A\ \right\}.
\end{equation}
It is easily seen that $\G_g^{\star}$ acts on ${\mathscr H}_g$ properly discontinuously by
\begin{equation}
\gamma\cdot \Om=\,A\Om\,{}^t\!A\,+\,B\,{}^t\!A,
\end{equation}
where $\gamma=\begin{pmatrix} A & B\\
0 & {}^tA^{-1} \end{pmatrix}\in \G_g^{\star}$ and $\Om\in {\mathscr H}_g.$

\end{section}

\vskip 1cm

\begin{section}{{\large\bf Moduli Spaces for Real Abelian Varieties}}
\setcounter{equation}{0}

\vskip 0.3cm
In Section 3, we knew that $\G_g^{\star}$ acts on ${\mathscr H}_g$ properly discontinuously by
the formula (3.17). So the quotient space
$${\mathscr X}_\BR^g:=\,\G_g^{\star}\backslash {\mathscr H}_g$$
inherits a structure of stratified real analytic space from the real analytic structure on ${\mathscr H}_g$.
The stratified real analytic space ${\mathscr X}_\BR^g$ classifies, up to real isomorphism,
real principally polarized abelian varieties $(\FA,E,S)$ of dimension $g$. Thus ${\mathscr X}_\BR^g$ is called the (real)
moduli space of real principally polarized abelian varieties $(\FA,E,S)$ of dimension $g$.

\vskip 0.2cm
To study the structure of ${\mathscr X}_\BR^g$, we need the following result of A. A. Albert \cite{AL}.

\begin{lemma} Let $S_g(\BZ/2)$ be the set of all $g\times g$ symmetric matrices with coefficients in $\BZ/2$.
We note that $GL(g,\BZ/2)$ acts on $S_g(\BZ/2)$ by $N\longmapsto AN\,{}^t\!A$ with $A\in GL(g,\BZ/2)$ and
$N\in S_g(\BZ/2).$ We put
$$ \pi(N):=\,\prod_{k=1}^g (1-n_{kk})\qquad \textrm{for}\ N=(n_{ij})\in S_g(\BZ/2).$$
Then $N\in S_g(\BZ/2)$ is equivalent mod $GL(g,\BZ/2)$ to a matrix of the form :
\vskip 0.1cm
(I)\ \ \ $\begin{pmatrix} I_{\lambda} & 0 \\ 0 & 0 \end{pmatrix}$ \ \ if \ $\pi(N)=0$ and
$\textrm{rank}\,(N)=\lambda$
\vskip 0.1cm\noindent
or
\vskip 0.1cm
(II)\ \ \ $\begin{pmatrix} H_{\lambda} & 0 \\ 0 & 0 \end{pmatrix}$ with
$H_{\lambda}:=\begin{pmatrix} 0 & \cdots  & 1 \\
\vdots & \udots & \vdots \\ 1 & \cdots & 0
 \end{pmatrix}\in \BZ^{(\lambda,\lambda)}$
\ \ if \ $\pi(N)=1$ and
$\textrm{rank}\,(N)=\lambda$.
\end{lemma}

$N\in S_g(\BZ/2)$ is said to be {\it diasymmetric} in Case (I) and to be {\it orthosymmetric} in Case (II).

\vskip 0.5cm
\begin{theorem}
Let $(\FA,E)$ be a principally polarized abelian variety of dimension $g$. Then there exists a real structure $S$
on $\FA$ such that $E$ is $S$-real if and only if $(\FA,E)$ admits a period matrix of the following form
$$ \big( I_g, {\frac 12}\,M+\,i\,Y\big),\qquad Y\in \PG,$$
\noindent where $M$ is one of the forms (I) and (II) in Lemma 4.1.
\end{theorem}

The above theorem is essentially due to Comessatti \cite{C}.
We refer to \cite{Si1} or \cite[Theorem 2.3, pp.\,78--80 and Theorem 4.1, pp.\,86--88]{Si2} for the proof of the above theorem.

\begin{lemma} Let $\Om_1$ and $\Om_2$ be two elements of $\HG$ such that
$$\Om_i=\,{\frac 12}\,X_i +\,i\,Y_i,\quad M_i\in\BZ^{(g,g)}, \ Y_i\in {\mathcal P}_g,\ i=1,2.$$
Then $\Om_1$ and $\Om_2$ have images, under the natural projection
$\pi_g: \HG\lrt \XG,$ in the same connected component of $\XG$, if and only if
$ \textrm{rank}\,(M_1 \,mod\ 2)= \textrm{rank}\,(M_2\,mod\,2)$ and
$\pi (M_1 \,mod\ 2)=\,\pi (M_2 \,mod\ 2).$
\end{lemma}

\begin{theorem} $\XG$ is a real analytic manifold of dimension $g(g+1)/2$ and has $g+1+ \left[ {\frac g2}\right]$
connected components. Moreover $\XG$ is semi-algebraic, i.e., $\XG$ is defined by a finite number of polynomial
equalities and inequalities.
\end{theorem}
\noindent
{\it Proof.} The proof can be found in \cite[Theorem 6.1, p. 161]{S-S}. \hfill $\square$

\vskip 0.3cm
\begin{remark} Let $\Om=\,{\frac 12}M+\,i\,Y\in \HG$ with $M=\,^tM\in\BZ^{(g,g)}.$
If $\textrm{rank}\,(M \,mod\ 2)=\lambda$, then $\FA_\Om(\BR)$ has $2^{g-\lambda}$ connected components
(cf.\,\cite{S-S, Si1}). The other invariant $\pi (M_2 \,mod\ 2)$ is an invariant related to the polarization.
\end{remark}

Recall that by Lemma 4.1, the  connected components of $\XG$ correspond to the different possible values of
$(\lambda,i)=\big(\textrm{rank}\,(M\, \,mod\ 2),\pi (M\, \,mod\ 2)\big)$ on which we have the restriction\,:
\begin{equation}
0\leq \lambda \leq g,\quad i=0\ or \ 1,\quad \textrm{and}\ \ i=0\ \textrm{if}\ \lambda\ is\ odd,\quad i=1
\ if \ \lambda=0.
\end{equation}
We denote by $\XGI$ the connected components of $\XG$ corresponding to the invariants $(\lambda,i).$

\begin{definition}
Let $M\in \BZ^{(g,g)}$ be a $g\times g$ symmetric integral matrix. We say that $M$ is $\textsf{of the standard form}$
if $M$ is of one of the forms in Lemma 4.1 (we observe that for fixed $(\lambda,i)$ this form is unique).
\end{definition}

Now we can prove the following.

\begin{lemma} Let $M\in \BZ^{(g,g)}$ be a symmetric integral matrix which is of the standard form with invariants
$(\lambda,i)$. Let
$$\G_{(\lambda,i)}^g:=\,\left\{\, A\in GL(g,\BZ)\,|\ \,AM\,{}^t\!A \equiv M\ (mod\,\,2)\ \right\}.$$
Then
$$\XGI\,\cong\,\G_{(\lambda,i)}^g \ba {\mathcal P}_g.$$
\end{lemma}
\noindent
{\it Proof.} Let $[\Om]$ be a class in $\XGI$. By Lemma 4.1 and Lemma 4.2, there exist a symmetric integral
matrix $M\in\BZ^{(g,g)}$ with invariants $(\lambda,i)$ of the standard form and an element $Y\in {\mathcal P}_g$
such that ${\frac 12}M+\,i\,Y$ is a representative for the class $[\Om]$. If $Y_*\in\PG$ is such that
${\frac 12}M+\,i\,Y_*$ is also a representative for the class $[\Om]$, according to Theorem 3.2,
$$M\equiv AM\,{}^tA\ \ (mod\,\,2)\qquad \textrm{and}\qquad Y_*=AY\,{}^tA$$
for some $A\in GL(g,\BZ).$ \hfill $\square$

\begin{theorem}
$\XGI$ is a connected semi-algebraic set with a real analytic structure.
\end{theorem}
\noindent
{\it Proof.} The proof can be found in \cite[p.\,160]{S-S}. \hfill $\square$

\vskip 0.3cm
Let $(\FA,E,S)$ be a real polarized abelian variety and $-S$ be the real structure obtained by composing $S$ with the
involution $z\longmapsto -z$ of $\FA$. We see that $(\FA,E,-S)$ is also a real polarized abelian variety.
In general $(\FA,E,-S)$ is not real isomorphic to $(\FA,E,S)$. Therefore the following correspondence
\begin{equation}
\Sigma : \XG\lrt \XG,\qquad (\FA,E,S)\longmapsto (\FA,E,-S)
\end{equation}
defines a non-trivial involution of $\XG$.

\vskip 0.2cm Let $M\in \BZ^{(g,g)}$ be a symmetric integral matrix which is of the standard form with invariants
$(\lambda,i)$. It is easily checked that $M^3=M.$ We put
\begin{equation}
\Sigma_M:=\,\begin{pmatrix} -M & I_g \\ -(I_g+M^2) & M \end{pmatrix}.
\end{equation}
It is easy to see the following facts (4.4) and (4.5).
\begin{equation}
\Sigma_M \in \G_g   \qquad \textrm{and}\qquad \big(\Sigma_M\big)^{-1}=-\Sigma_M .
\end{equation}

\begin{equation}
\big(\,{}^t\Sigma_M\big)^{-1}\begin{pmatrix} -I_g & 0 \\ M & I_g \end{pmatrix}
{}^t\Sigma_M\,=\,\begin{pmatrix} I_g & 0 \\ -M & -I_g \end{pmatrix}.
\end{equation}
Now we assume that $\Om={\frac 12}M+\,i\,Y\in \HG$ represents $(\FA,E,S)$. By (3.13),
the matrices of $S$ and $-S$ are given by
\begin{equation}
M_S=\begin{pmatrix} -I_g & 0 \\ M & I_g \end{pmatrix}\qquad \textrm{and}\qquad
M_{-S}=\,\begin{pmatrix} I_g & 0 \\ -M & -I_g \end{pmatrix}
\end{equation}
respectively with respect to the $\BR$-basis given by the columns of $(\Om,I_g).$ By the formulas (4.5) and (4.6)
we see that $\Sigma_M (\Om)$ represents the real polarized abelian variety.

\begin{lemma}
Let $M\in \BZ^{(g,g)}$ be a symmetric integral matrix which is of the standard form with invariants
$(\lambda,i)$ and $Y\in \PG.$ Then we have
\begin{equation}
\Sigma_M \left({\frac 12}M+\,i\,Y\right)  =\, {\frac 12}\,M \,+\,i\,
\begin{pmatrix} {\frac 12}\,I_\lambda & 0 \\ 0 & I_{g-\lambda} \end{pmatrix}Y^{-1}
\begin{pmatrix} {\frac 12}\,I_\lambda & 0 \\ 0 & I_{g-\lambda} \end{pmatrix}^{-1}.
\end{equation}
\end{lemma}
\noindent
{\it Proof.} Using the fact that $M^3=M$, by a direct computation, we get
\begin{equation}
\Sigma_M \left({\frac 12}M+\,i\,Y\right)=\,M\big(I_g+M^2\big)^{-1}\,+\,i\,\big(I_g-{\frac 12}M^2\big)
\,Y^{-1}\big(I_g+M^2\big)^{-1}.
\end{equation}
It is easily checked that
\begin{equation}
\big(I_g+M^2\big)^{-1}=\,I_g-{\frac 12}M^2=\,
\begin{pmatrix} {\frac 12}\,I_\lambda & 0 \\ 0 & I_{g-\lambda} \end{pmatrix}.
\end{equation}
The formula (4.7) follows immediately from (4.8) and (4.9). \hfill $\square$

\begin{proposition}
The map $\Sigma :\XG\lrt \XG$ defined by
\begin{equation}
\Sigma\left( [(\FA,E,S)]\right):=\,[(\FA,E,-S)],\qquad [(\FA,E,S)]\in \XG
\end{equation}
is a real analytic involution of $\XG$. For each connected component $\XGI$, we have
$$\Sigma \left( \XGI\right) =\,\XGI.$$
Hence $\Sigma$ leaves the connected components of $\XG$ globally fixed.
\end{proposition}
\noindent {\it Proof.}
Let $M\in \BZ^{(g,g)}$ be a symmetric integral matrix which is of the standard form with invariants
$(\lambda,i)$. We denote by $\HG(M)$ the connected component of $\HG$ containing the matrices of
the form ${\frac 12}M+\,i\,Y\in \HG$ with $Y\in \PG.$ According to (4.5) and Lemma 4.4, we see that
$\Sigma_M$ defines an involution of $\HG(M)$. Since $\HG (M)$ is mapped onto $\XGI$, we obtain
the desired result. \hfill $\square$

\end{section}

\vskip 1cm

\vskip 0.3cm
\begin{section}{{\large\bf Compactifications of the Moduli Space ${\mathscr X}_\BR^g$}}
\setcounter{equation}{0}

\vskip 0.3cm
In this section we review the compactification ${\overline{\XG}}$ of $\XG$ obtained by R. Silhol \cite{Si3} and the
Baily-Borel compactification of $\G_g(4m)\ba \BH_g$ which is related to the moduli space of real abelian varieties with level
$4m$ structure.

\vskip 0.2cm First of all we recall the Satake compactification of the Siegel modular variety
${\mathcal A}_g:=\G_g\ba \BH_g.$ Let
\begin{equation}
\BD_g:=\,\left\{\, W\in\BC^{(g,g)}\,\,|\ \, W=\,{}^tW,\ \, I_g-{\overline W}W>0\ \right\}
\end{equation}
be the generalized unit disk of degree $g$ which is a bounded realization of $\BH_g.$
In fact, the Cayley transform $\Phi_g:\BD_g\lrt \BH_g$ defined by
\begin{equation}
\Phi_g(W):=\,i\,(I_g+W)(I_g-W)^{-1},\qquad W\in\BD_g
\end{equation}
is a biholomorphic mapping of $\BD_g$ onto $\BH_g$ which gives the bounded realization of $\BH_g$ by $\BD_g$
\cite[pp.\,281-283]{Sieg1}. The inverse $\Psi_g$ of $\Phi_g$ is given by
\begin{equation}
\Psi_g(\Om)=\,(\Om-i\,I_g)(\Om+\,i\,I_g)^{-1},\qquad \Om\in\BH_g.
\end{equation}

We let
\begin{equation*}
T={1\over {\sqrt{2}} }\,
\begin{pmatrix} \ I_g&\ I_g\\
                   iI_g&-iI_g\end{pmatrix}
\end{equation*}
be the $2g\times 2g$ matrix represented by $\Phi_g.$ Then
\begin{equation*}
T^{-1}Sp(g,\BR)\,T=\left\{ \begin{pmatrix} P & Q \\ \OQ & \OP
\end{pmatrix}\in \BC^{(2g,2g)}\,\Big|\ ^tP\OP-\,{}^t\OQ Q=I_g,\ {}^tP\OQ=\,{}^t\OQ
P\,\right\}.
\end{equation*}
Indeed, if $M=\begin{pmatrix} A&B\\
C&D\end{pmatrix}\in Sp(g,\BR)$, then
\begin{equation*}
T^{-1}MT=\begin{pmatrix} P & Q \\ \OQ & \OP
\end{pmatrix},
\end{equation*}
where
\begin{equation}
P= {\frac 12}\,\left\{ (A+D)+\,i\,(B-C)\right\}
\end{equation}
and
\begin{equation}
 Q={\frac
12}\,\left\{ (A-D)-\,i\,(B+C)\right\}.
\end{equation}

For brevity, we set
\begin{equation*}
G_*= T^{-1}Sp(g,\BR)T.
\end{equation*}
Then $G_*$ is a subgroup of $SU(g,g),$ where
$$SU(g,g)=\left\{\,h\in\BC^{(2g,2g)}\,|\ {}^th I_{g,g}{\overline
h}=I_{g,g},\ \det h=1\,\right\},\quad I_{g,g}=\begin{pmatrix} \ I_g&\ 0\\
0&-I_g\end{pmatrix}.$$ In the case $g=1$, we observe that
$$T^{-1}Sp(1,\BR)T=T^{-1}SL_2(\BR)T=SU(1,1).$$
If $g>1,$ then $G_*$ is a {\it proper} subgroup of $SU(g,g).$ In
fact, since ${}^tTJ_gT=-\,i\,J_g$, we get
$$G_*=\left\{\,h\in SU(g,g)\,|\ {}^thJ_gh=J_g\,\right\}.$$

Let
\begin{equation*}
P^+=\left\{\begin{pmatrix} I_g & Z\\ 0 & I_g
\end{pmatrix}\,\Big|\ Z=\,{}^tZ\in\BC^{(g,g)}\,\right\}
\end{equation*}
be the $P^+$-part of the complexification of $G_*\subset SU(g,g).$

Since the Harish-Chandra decomposition of an element
$\begin{pmatrix} P & Q\\ {\overline Q} & {\overline P}
\end{pmatrix}$ in $G_*^J$ is
\begin{equation*}
\begin{pmatrix} P & Q\\ \OQ & \OP
\end{pmatrix}=\begin{pmatrix} I_g & Q\OP^{-1}\\ 0 & I_g
\end{pmatrix} \begin{pmatrix} P-Q\OP^{-1}\OQ & 0\\ 0 & \OP
\end{pmatrix} \begin{pmatrix} I_g & 0\\ \OP^{-1}\OQ & I_g
\end{pmatrix},
\end{equation*}
the $P^+$-component of the following element
$$\begin{pmatrix} P & Q\\ \OQ & \OP
\end{pmatrix}   \cdot\begin{pmatrix} I_g & W\\ 0 & I_g
\end{pmatrix},\quad W\in \BD_g$$ of the complexification of $G_*^J$ is
given by
\begin{equation*}
\left( \begin{pmatrix} I_g & (PW+Q)(\OQ W+\OP)^{-1}
\\ 0 & I_g
\end{pmatrix}\right).
\end{equation*}

We note that $Q\OP^{-1}\in\Dg.$ We get the Harish-Chandra
embedding of $\Dg$ into $P^+$\,(cf.\,\cite[p.\,155]{Kna} or
\cite[pp.\,58-59]{Sata}). Therefore we see that $G_*$ acts on $\Dg$
transitively by
\begin{equation}
\begin{pmatrix} P & Q \\ \OQ & \OP
\end{pmatrix}\cdot W=(PW+Q)(\OQ W+\OP)^{-1},\quad \begin{pmatrix} P & Q \\ \OQ & \OP
\end{pmatrix}\in G_*,\ W\in \Dg.
\end{equation}
The isotropy subgroup at the origin $o$ is given by
$$K=\left\{\,\begin{pmatrix} P & 0 \\ 0 & {\overline
P}\end{pmatrix}\,\Big|\ P\in U(g)\ \right\}.$$ Thus $G_*/K$ is
biholomorphic to $\Dg$. The action (2.4) is compatible with the
action (5.6) via the Cayley transform (5.2).

\vskip 0.2cm
In summary, $Sp(g,\BR)$ acts on $\BD_g$ transitively by
\begin{equation}
\begin{pmatrix} A & B \\ C & D
\end{pmatrix}\cdot W=(PW+Q)(\OQ W+\OP)^{-1},\quad \begin{pmatrix} A & B \\ C & D
\end{pmatrix}\in Sp(g,\BR),\ W\in \Dg,
\end{equation}
where $P$ and $Q$ are given by (5.4) and (5.5). This action extends to the closure ${\overline{\BD}}_g$ of
$\BD_g$ in $\BC^{g(g+1)/2}.$

\vskip 0.2cm
For an integer $s$ with $0\leq s\leq g$, we let
\begin{equation}
{\mathscr F}_s:=\,\left\{\,W=
\begin{pmatrix} W_1 & 0 \\ 0 & I_{g-s}
\end{pmatrix}\,\Big|\ \,W_1\in\BD_s\ \right\}\,\subset\ {\overline{\BD}}_g.
\end{equation}
We say that ${\mathscr F}_s$ is the {\it standard boundary component} of degree $s$. If there exists an element
$\gamma\in Sp(g,\BQ)$ (equivalently $\gamma\in \G_g$) with ${\mathscr F}=\,\g\, ({\mathscr F}_s)\subset \OBD,$ then
${\mathscr F}$ is said to be a {\it rational boundary component} of degree $s$. The Siegel upper half plane $\BH_s$
is attached to $\BH_g$ as a limit of matrices in $\BC^{(g,g)}$ by
\begin{equation*}
\Om_1\longmapsto \lim_{Y\lrt\infty} \begin{pmatrix} \Om_1 & 0 \\
0 & i\,Y \end{pmatrix},\qquad \Om_1\in \BH_s,\ Y\in {\mathcal P}_{g-s},
\end{equation*}
meaning that all the eigenvalues of $Y$ converge to $\infty.$

\vskip 0.2cm
For a rational boundary component ${\mathscr F}\subset \OBD,$ we let
\begin{equation*}
P({\mathscr F})=\,\left\{\,\alpha\in Sp(g,\BQ)\,|\ \alpha({\mathscr F})=\,{\mathscr F}\,\right\}
\end{equation*}
be the normalizer in $Sp(g,\BQ)$ of ${\mathscr F}$ (or the parabolic subgroup of $Sp(g,\BQ)$ associated to ${\mathscr F}$)
and let
\begin{equation*}
Z({\mathscr F})= \,\left\{\,\alpha\in Sp(g,\BQ)\,|\ \alpha(W)=W \    \,\textrm{for all}\ W\in{\mathscr F}\,\right\}
\end{equation*}
be the centralizer of ${\mathscr F}$. We put
\begin{equation*}
G({\mathscr F}):= \,P({\mathscr F})/ Z({\mathscr F})\,\cong\, Sp(s,\BQ).
\end{equation*}
Obviously $G({\mathscr F})$ acts on ${\mathscr F}$. We choose the standard boundary component
${\mathscr F}={\mathscr F}_s.$ An element $\g$ of $P({\mathscr F})$ is of the form
\begin{equation}
\g=\,\begin{pmatrix} A_1 & 0 & B_1 & *\\ * & u & * & * \\
C_1 & 0 & D_1 & * \\ 0 & 0 & 0 & {}^tu^{-1}
\end{pmatrix}\in Sp(g,\BQ),
\end{equation}
where
\begin{equation*}
\g_1= \begin{pmatrix} A_1 & B_1 \\ C_1 & D_1
\end{pmatrix}\in Sp(s,\BQ) \qquad \textrm{and} \qquad u\in GL(g-s,\BQ).
\end{equation*}

The unipotent radical $U( {\mathscr F})$ of $P({\mathscr F})$ is given by
\begin{equation}
U( {\mathscr F})=\,\left\{\,\begin{pmatrix} I_s & 0 & 0 & {}^t\mu\\ \lambda & I_{g-s} & \mu & \kappa \\
0 & 0 & I_s & -\,{}^t\lambda \\ 0 & 0 & 0 & I_{g-s}\end{pmatrix}
\ \Biggl| \ \, \lambda,\mu\in \BQ^{(g-s,s)},\ \kappa\in \BQ^{(g-s,g-s)}\,\right\}
\end{equation}
and the centralizer $Z_U({\mathscr F})$ of $U( {\mathscr F})$ is given by
\begin{equation}
Z_U({\mathscr F})=
\,\left\{\,\begin{pmatrix} I_s & 0 & 0 & 0 \\ 0 & I_{g-s} & 0 & \kappa \\
0 & 0 & I_s & 0 \\ 0 & 0 & 0 & I_{g-s}\end{pmatrix}
\ \Biggl| \ \, \kappa\in \BQ^{(g-s,g-s)}\,\right\}.
\end{equation}
We have inclusions of normal subgroups
$$ Z_U({\mathscr F}) \subset U( {\mathscr F}) \subset P({\mathscr F}).$$
The Levi factor $ L({\mathscr F})$ of $ P({\mathscr F})$ is given by
\begin{equation}
L({\mathscr F})=\, G_h({\mathscr F})\,G_l({\mathscr F})
\end{equation}
with
\begin{equation}
G_h({\mathscr F})=
\,\left\{\,\begin{pmatrix} A_1 & 0 & B_1 & 0 \\ 0 & I_{g-s} & 0 & 0 \\
C_1 & 0 & D_1 & 0 \\ 0 & 0 & 0 & I_{g-s}\end{pmatrix}
\in P({\mathscr F}) \Biggl| \ \,\begin{pmatrix} A_1 & B_1 \\ C_1 & D_1
\end{pmatrix}\in Sp(s,\BQ) \,\right\}.
\end{equation}
and
\begin{equation}
G_l({\mathscr F})=
\,\left\{\,\begin{pmatrix} I_s & 0 & 0 & 0 \\ 0 & S & 0 & 0 \\
0 & 0 & I_s & 0 \\ 0 & 0 & 0 & {}^tS^{-1} \end{pmatrix}
\in P({\mathscr F}) \Biggl| \ \,S\in GL(g-s,\BQ) \,\right\}.
\end{equation}
The subgroup $U( {\mathscr F})G_h({\mathscr F})$ is normal in $P({\mathscr F})$.
The map $P({\mathscr F})\lrt Sp(s,\BQ),\ \g\mapsto \g_1$ is surjective and induces the isomorphism
$G_h({\mathscr F}_s)\cong Sp(s,\BQ).$ We note that the map $f:P({\mathscr F}_s)\cap Sp(g,\BZ)\lrt Sp(s,\BZ),\
\g\mapsto \g_1$ is obtained via
$\begin{pmatrix} W & 0 \\ 0 & I_{g-s} \end{pmatrix}\longmapsto W$, in the sense that if $\g\in P({\mathscr F}_s),$
then
$$\gamma\cdot\begin{pmatrix} W & 0 \\ 0 & I_{g-s} \end{pmatrix}=\,
\begin{pmatrix} \g_1(W) & 0 \\ 0 & I_{g-s} \end{pmatrix}.$$
We define
\begin{equation*}
\BD_g^{st}:=\,\coprod_{0\leq s\leq g} {\mathscr F}_s
\end{equation*}
and
\begin{equation*}
\BD_g^*:=\,\coprod_{{\mathscr F}:\,rational} {\mathscr F},
\end{equation*}
where ${\mathscr F}$ runs over all rational boundary components. Via the Cayley transform $\Phi_g$ (cf.\,(5.2)),
we identify
$$\BD_g^{st}=\,\BH_g^{st}=\,\coprod_{0\leq s\leq g} \BH_s.$$

\begin{definition}
Let $u>1.$ We denote by ${\mathfrak W}_g(u)$ the set of all matrices $\Om=X+\,i\,Y$ in $\BH_g$ with
$X=(x_{ij})\in\BR^{(g,g)}$ satisfying the conditions $(\Om 1)$ and $(\Om 2)$\,:
\vskip 0.2cm\noindent
$(\Om 1)\ |x_{ij}|< u\,;$
\vskip 0.1cm\noindent
$(\Om 2)$\ if $Y=\,{}^tWDW$ is the Jacobi decomposition of $Y$ with $W=(w_{ij})$ strictly upper triangular and $D=\textrm{diag}(d_1,\cdots,d_g)$ diagonal, then we have
$$|w_{ij}|<u,\quad 1 < u\, d_1,\quad d_i < u\, d_{i+1},\quad i=1,\cdots,g-1.$$
\end{definition}

It is well known that for sufficiently large $u>0$, the set ${\mathfrak W}_g(u)$ is a fundamental set for the action
of $\G_g$ on $\BH_g$, that is, $\G_g\cdot {\mathfrak W}_g(u)=\BH_g$, and
$$\left\{\,\g\in \G_g\,|\ \g\cdot {\mathfrak W}_g(u) \cap {\mathfrak W}_g(u)\neq \emptyset\,\right\}$$
is a finite set. We observe that if $\Om=\begin{pmatrix} \Om_1 & \Om_3 \\ {}^t\Om_3 & \Om_2 \end{pmatrix}
\in {\mathfrak W}_g(u)$ with $\Om_1\in \BC^{(s,s)},$ then $\Om_1\in {\mathfrak W}_s(u).$

\begin{definition}
We can choose a sufficiently large $u_0>0$ such that for all $0\leq s\leq g,\
{\mathfrak W}_s(u_0)$ is a fundamental set for the action
of $\G_s$ on $\BH_s$. In this case we simply write ${\mathfrak W}_s={\mathfrak W}_s(u_0)$ with
$0\leq s\leq g$. We define
\begin{equation*}
{\mathfrak W}_g^*:=\,\coprod_{0\leq s\leq g} {\mathfrak W}_s.
\end{equation*}
For $\Om_*\in {\mathfrak W}_{g-r}$, we let $U$ be a neighborhood of $\Om_*$ in ${\mathfrak W}_{g-r}$ and $v$
a positive real number. For $0\leq s\leq r$, we let $W_s(U,v)$ be the set of all
$$\Om=\,\begin{pmatrix} \Om_1 & \Om_3 \\ {}^t\Om_3 & \Om_2 \end{pmatrix}
\in {\mathfrak W}_{g-s}\quad {\rm with}\quad \Om_1\in \BC^{(g-r,g-r)}$$ satisfying the conditions
$({\mathfrak W}1)$ and $({\mathfrak W}1)$\,:
\vskip 0.2cm\noindent
$({\mathfrak W}1)$ \ \ $\Om_1\in U\,;$
\vskip 0.2cm\noindent
$({\mathfrak W}2)$ \ if $Y=\,{}^tWDW$ is the Jacobi decomposition of $Y$ with $W$ strictly upper triangular and $D=\textrm{diag}(d_1,\cdots,d_g)$ diagonal, then we have
$d_{g-r+1} > v.$
\end{definition}

A fundamental set of neighborhoods of $\Om_*\in {\mathfrak W}_{g-r}$ for the Satake topology on
${\mathfrak W}_g^*$ is given by the collection $\left\{\, \bigcup_{0\leq s\leq r} W_s(U,v)\,\right\}$'s,
where $U$ runs through neighborhoods of $\Om_*$ in ${\mathfrak W}_{g-r}$ and $v$ ranges in $\BR^+$. We regard
$${\mathfrak W}_g^* \subset \,\BH_g^{st}\cong\,\BD_g^{st}$$
as a subset of $\BD_g^*.$

\vskip 0.2cm The Satake topology on $\BD_g^*$ is characterized as the unique topology ${\mathscr T}$ extending
the ordinary matrix topology on $\BD_g$ and satisfying the following properties (ST1)--(ST4)\,:
\vskip 0.2cm\noindent
(ST1)\ \ ${\mathscr T}$ induces on ${\mathfrak W}_g^*$ the topology defined in Definition 5.2\;
\vskip 0.2cm\noindent
(ST2)\ \ $Sp(g,\BQ)$ acts continuously on $\BD_g^*$\,;

\vskip 0.2cm\noindent
(ST3)\ \ ${\mathcal A}_g^*=\,\G_g\ba {\mathbb D}^*_g$ is a compact Hausdorff space\,;

\vskip 0.2cm\noindent
(ST4)\ \ For any $\Om\in \BD_g^*,$ there exists a fundamental set of neighborhoods $\{ U\}$ of $\Om$
such that $\g\cdot U=U$ if $\g\in \G_g(\Om):=\,\{ \g\in\G_g\,|\ \g\cdot \Om=\Om\,\},$ and
$\g\cdot U \cap U=\emptyset$ if $\g\notin \G_g(\Om).$

\vskip 0.2cm For a proof of these above facts we refer to \cite{BB}.

\vskip 0.3cm Now we are ready to investigate the compactification of the moduli space $\XG$ of real
principally polarized abelian varieties of dimension $g$ obtained by R. Silhol.

\begin{definition}
Let $u>1.$ We let $F_g(u)$ be the set of all $\Om=X+\,i\,Y\in\HG$ with $X= \textrm{Re}\,\Om=(x_{ij})$
satisfying the following conditions (a) and (b)\,:
\vskip 0.2cm\noindent
(a) $x_{ij}=0$ or ${\frac 12}$\,;
\vskip 0.2cm\noindent
(b) if $Y=\,{}^tWDW$ is the Jacobi decomposition of $Y$ with $W$ strictly upper triangular and $D=\textrm{diag}(d_1,\cdots,d_g)$ diagonal, then we have
$$ |w_{ij}|\leq u\qquad {\rm and} \qquad 0 <d_i \leq u\,d_{i+1}.$$
We define $F'_g(u)$ to be the set of matrices in $\HG$ satisfying the condition $|x_{ij}|\leq {\frac 12}$
and the above condition (b). Let $u_0>$ be as in Definition 5.2. We put
$$F_g:=\,F_g(u_0).$$
It is well known that $F_g$ is a fundamental set for the action of $\G_g^{\star}$ on $\HG$.
\end{definition}

For two nonnegative integers $s$ and $t$, we define two subsets ${\mathscr F}_{s,t}$ and $F_{s,t}$ of $\BD_g^*$ as follows.
\begin{equation}
{\mathscr F}_{s,t}:=\,\left\{\,
\begin{pmatrix} -I_s & 0 & 0 \\ 0 & W & 0 \\ 0 & 0 & I_t \end{pmatrix}\in\BD_g^* \ \Big|\ \
W\in \BD_{g-(s+t)}\ \right\}
\end{equation}
and
\begin{equation}
F_{s,t}:=\,\left\{\,
\begin{pmatrix} -I_s & 0 & 0 \\ 0 & W & 0 \\ 0 & 0 & I_t \end{pmatrix}\in {\mathscr F}_{s,t}
\ \Big|\ \
W\in F_{g-(s+t)}\ \right\}.
\end{equation}

For $M\in \BZ^{(g,g)},$ we set
\begin{equation*}
F_M:=\,\left\{\,\Om\in F_g\,|\ 2\,\textrm{Re}\,\Om=\,M\ \right\}.
\end{equation*}
In particular, $F_{\bf 0}=\,\left\{\,\Om\in F_g\,|\ \textrm{Re}\,\Om=\,0\ \right\}$, where ${\bf 0}$ denotes the
$g\times g$ zero matrix. We let
$${\mathcal M}:=\,\left\{\,M=(m_{ij})\in \BZ^{(g,g)}\,|\ M=\,{}^tM,\ m_{ij}=0\ or\ 1\ \right\}.$$

For any $M\in {\mathcal M},$ we set
\begin{equation*}
B_M:=\,\begin{pmatrix} I_g & {\frac 12}\,M  \\ 0 & I_g  \end{pmatrix}\in Sp(g,\BQ).
\end{equation*}
By the definition we have
\begin{equation*}
F_g=\,\bigcup_{M\in {\mathcal M}}B_M(F_{\bf 0}) \qquad \textrm{and} \qquad
{\overline F}_g=\,\bigcup_{M\in {\mathcal M}}B_M({\overline F}_{\bf 0} ).
\end{equation*}
We can show that $\OHG=\,\G_g^{\star}\cdot {\overline F}_g. $

\vskip 0.3cm
Now we embed $\HG$ into $\BD_g^*$ via the Cayley transform (5.3). We let $\OHG$ be the closure of $\HG$
in $\BD_g^*$. Then the action of $\G_g^{\star}$ extends to an action of $\G_g^{\star}$ on $\OHG$
(see (3.16),\,(3.17),\,(ST2)). R. Silhol proved that the quotient space
$\G_g^{\star}\ba \OHG$ is a connected, compact Hausdorff space (cf.\,\cite[pp.\,173-177]{Si3}).
Let $\pi:\OHG\lrt \G_g^{\star}\ba \OHG$ be the canonical projection. For $M\in {\mathcal M}$, we define
$${\mathscr H}_M=\,\left\{\, {\frac 12}\,M+\,i\,Y\in \HG\,\right\}.$$
We let ${\overline{\mathscr H}}_M$ be the closure of ${\mathscr H}_M$ in $\OHG.$
Then without difficulty we can see that
\begin{equation}
\G_g^{\star}\ba \OHG=\,\bigcup_{0\leq s+t\leq g}\bigcup_{M\in {\mathcal M}}
\left( \pi\big( B_M( {\mathscr F}_{s,t})\cup {\overline{\mathscr H}}_{\bf 0}\big) \right),
\end{equation}

\vskip 0.3cm
Let $\{ {\mathscr X}_i\,|\ 1\leq i\leq N\,\}$ with $N=g+1+\left[ {\frac g2}\right]$ be the connected components of
$\XG\subset \G_g^{\star}\ba \OHG$ and let $\Sigma_i$ be the restriction to ${\mathscr X}_i$ of
the fundamental involution $\Sigma$ (cf.\, Proposition 4.1). We note that $\Sigma$ does not extend to a global
involution of $ \G_g^{\star}\ba \OHG$. But $\Sigma_i$ extends to an involution of the closure
${\overline{\mathscr X}}_i$ of ${\mathscr X}_i$ in $ \G_g^{\star}\ba \OHG$.
We observe that for each $1\leq i\leq N$, we have $ {\overline{\mathscr X}}_i
=\,\G_g^{\star}(M_i)\ba {\overline{\mathscr H}}_{M_i}$
for some $M_i\in {\mathcal M}.$ Here $\G_g^{\star}(M_i)=\,\left\{\, \g\in \G_g^{\star}\,|\ \g\, ({\mathscr H}_{M_i})
={\mathscr H}_{M_i}\,\right\}.$

\begin{definition}
Let $z_1\in {\overline{\mathscr X}}_i$ and  $z_2\in {\overline{\mathscr X}}_j$. We say that $z_1$ and $z_2$
are $\Sigma$-equivalent and write $z_1\sim z_2$ if $\Sigma_i(z_1)= \Sigma_j(z_2).$
\end{definition}

Silhol \cite[p.\,185]{Si3} showed that $\sim$ defines an equivalence relation in $\G_g^{\star}\ba \OHG$.

\vskip 0.2cm
By a direct computation, we obtain
\begin{equation*}
P({\mathscr F}_{s,t})=\,\left\{
\begin{pmatrix} v_1 & 0 & 0 & 0 & 0 & v_{21}  \\
* & A_1 & 0 & 0 & B_1 & * \\
* & * & u_2 & -u_{21} & * & * \\
* & * & -u_{12} & u_1 & * & * \\
* & C_1 & 0 & 0 & D_1 & *  \\
v_{12} & 0 & 0 & 0 & 0 & v_2  \\
\end{pmatrix}\in Sp(g,\BQ)\ \right\},
\end{equation*}
where
\begin{equation*}
\g_1=\,\begin{pmatrix} A_1 & B_1  \\ C_1 & D_1  \end{pmatrix}\in Sp(g-r,\BQ)
\quad \textrm{with}\ \ r=s+t
\end{equation*}
and
\begin{equation*}
U=\,\begin{pmatrix} u_1 & u_{12}  \\ u_{21} & u_2  \end{pmatrix}\in GL(r,\BQ),\qquad
\quad V=\,\begin{pmatrix} v_1 & v_{21}  \\ v_{12} & v_2  \end{pmatrix}=\,{}^tU^{-1}.
\end{equation*}

Now we define
\begin{equation}
\XG (s,t):=\,(\G_g^{\star}\cap P({\mathscr F}_{s,t})) \ba ({\mathscr F}_{s,t}\cap \OHG).
\end{equation}
It is easily checked that
\begin{equation*}
\G_g^{\star}\cap P({\mathscr F}_{s,t})\,\cong\,\G_{g-(s+t)}^{\star}\quad \textrm{and}\quad
{\mathscr F}_{s,t}\cap \OHG\,\cong\,{\overline{\mathscr H}}_{g-(s+t)}.
\end{equation*}
\noindent
We define
\begin{equation}
\OXG:=\,\G_g^{\star}\ba \OHG /\sim.
\end{equation}

Silhol \cite[Theorem 8.17]{Si3} proved the following theorem.

\begin{theorem}
$\OXG$ is a connected compact Hausdorff space containing $\XG$ as a dense open subset.
As a set,
$$\OXG=\,\coprod_{0\leq s+t\leq g} \XG (s,t).$$
\end{theorem}

\vskip 0.3cm
We recall that $\BH_g^*$ denotes the Satake partial compactification of $\BH_g$ that is obtained by attaching all
rational boundary components with the Satake topology. We know that $Sp(g,\BQ)$ acts on $\BH_g^*$, the involution
$\tau:\BH_g\lrt \BH_g$ (cf.\,(2.9)) extends to $\BH_g^*$ and $\tau(\alpha\cdot x)=\,\tau(\alpha)\,\tau(x)$ for all
$\alpha\in Sp(g,\BQ)$ and $x\in \BH_g^*.$

\vskip 0.2cm
Let $N=4m$ with $m$ a positive integer. We write
$$X(N):=\,\G_g(N)\ba \BH_g \qquad {\rm and}\qquad V(N):=\,\G_g(N)\ba \BH_g^*.$$
We let
\begin{equation}
\pi_{BB}: \BH_g^*\lrt V(N)=\,\G_g(N)\ba \BH_g^*
\end{equation}
be the canonical projection of $\BH_g^*$ to the Baily-Borel compactification of $X(N).$ The involution $\tau$
passes to complex conjugation $\tau: V(N)\lrt V(N)$, whose fixed points we denote by $V(N)_\BR$. Obviously the
$\tau$-fixed set
$$X(N)_\BR:=\,\left\{\, x\in X(N)\,|\ \tau(x)=x\,\right\}$$
is a subset of $V(N)_\BR.$ We let $\overline{X(N)}_\BR$ denote the closure of $X(N)_\BR$ in $V(N)_\BR.$

\begin{theorem}
There exists a natural rational structure on $V(N)$ which is compatible with the real structure defined by $\tau$.
\end{theorem}
\noindent
{\it Proof.} It follows from Shimura's result \cite{Shi1} that the $\G_g(N)$-automorphic forms on $\BH_g$ are generated by those
automorphic forms with rational Fourier coefficients. \hfill $\square$

\vskip 0.2cm
If $\g\in \G_g(N)$ and ${\mathscr F}$ is a rational boundary component of $\BH_g^*$ such that $\tau( {\mathscr F})={\mathscr F},$
we define the set of $\g$-real points of ${\mathscr F}$ to be
\begin{equation}
{\mathscr F}^{\tau\g}:=\,\left\{\, x\in {\mathscr F}\,|\ \tau (x)=\,\g\cdot x\,\right\}.
\end{equation}
Then $\pi_{BB} \big({\mathscr F}^{\tau\g}\big) \subset V(N)_\BR.$

\begin{definition}
Let $N=4m$. A $\G_g(N)$-$\textsf{real boundary pair}$ $({\mathscr F},\g)$ of degree $s$ consists of a rational boundary component
${\mathscr F}$ of degree $s$ and an element $\g\in \G_g(N)$ such that ${\mathscr F}^{\tau\g}\neq \emptyset.$ We say that two $\G_g(N)$-real boundary components $({\mathscr F},\g)$ and $({\mathscr F}_*,\g_*)$ are $\textsf{equivalent}$
if the resulting loci of real points
$\pi_{BB}({\mathscr F}^{\tau\g})=\,\pi_{BB}({\mathscr F}_*^{\tau\g})$ coincide.
\end{definition}

We observe that if $({\mathscr F},\g)$ is a $\G_g(N)$-real boundary pair and if $\alpha\in \G_g(N)$, we see that
$\tau ({\mathscr F})=\,\g ({\mathscr F})$ and $(\alpha( {\mathscr F}),\tau(\alpha)\,\g \,\alpha^{-1})$ is an equivalent
$\G_g(N)$-real boundary pair.

\vskip 0.2cm
Fix a positive integer $s$ with $1\leq s\leq g.$ We define the map $\Phi:\BH_s\lrt \BH_g^*$ by
\begin{equation}
\Phi(\Om_1)=\,\lim_{Y\lrt\infty} \begin{pmatrix} \Om_1 & 0  \\ 0 & i\,Y  \end{pmatrix},
\qquad \Om_1\in \BH_s,\ Y\in {\mathcal P}_{g-s}.
\end{equation}
Obviously $\Phi(\BH_s)=\,{\mathscr F}_s$ is the standard boundary component of degree $s$ (cf.\,(5.8)).
\vskip 0.2cm
Let
$$\nu_s:P({\mathscr F}_s)\lrt G_h({\mathscr F}_s)$$
be the projection to the quotient. It is easily seen that $\nu_s$ commutes with $\tau$. Therefore
${\mathscr F}_s$ is preserved by $\tau$. The set
\begin{equation*}
{\mathscr F}^{\tau}_s=\,\left\{\,\Phi(i\,Y)\,|\ Y\in {\mathcal P}_s\,\right\}
\end{equation*}
is the set of $\tau$-fixed points in ${\mathscr F}_s$ and may be canonically identified with ${\mathcal P}_s$.
We denote by $i\,I_s$ its canonical base point. Then ${\mathscr F}_s$ is attached to $\BH_g$ so that
the cone $\Phi (i\,{\mathcal P}_s)$ is contained in the closure of the cone $i\,{\mathcal P}_g.$

\begin{proposition}
Let $({\mathscr F},\g)$ be a $\G_g(N)$-real boundary pair of degree $s$. Then there exists $\g_*\in \G_g$
such that $\g_* ({\mathscr F}_s)={\mathscr F}$ and
\begin{equation*}
\tau(\g_*)^{-1}\g\,\g_*=\,
\begin{pmatrix} A & B  \\ 0 & {}^tA^{-1}  \end{pmatrix}\in {\rm ker}\,(\nu_s).
\end{equation*}
Moreover, we may take $B=0,$ i.e., there exist $\g'\in \G_g(4m)$ and $\g_0\in\G_g$ so that
${\mathscr F}^{\tau\g'}={\mathscr F}^{\tau\g},\ \g_0 ({\mathscr F}_s)={\mathscr F},$ and so that
\begin{equation*}
\tau(\g_0)^{-1}\g'\,\g_0=\,
\begin{pmatrix} A & 0  \\ 0 & {}^tA^{-1}  \end{pmatrix}\in {\rm ker}\,(\nu_s).
\end{equation*}
\end{proposition}
\noindent
{\it Proof.} The proof can be found in \cite[pp.\,19-21]{GT1}. \hfill$\square$
\vskip 0.2cm
As an application of Proposition 5.1, we get the following theorem.

\begin{theorem}
Let $m\geq 1$ be a positive integer. Let ${\mathscr F}$ be a proper rational boundary component of $\BH_g$ of degree $g-1$.
Let $\g\in \G_g(4m)$ such that
$$ {\mathscr F}^{\tau\g}=\,\{\,x\in {\mathscr F}\,|\ \tau(x)=\g\cdot x\,\}\neq \emptyset.$$
Then $ {\mathscr F}^{\tau\g}$ is contained in the closure of $\BH_g^{\tau \G_g(4m)}$ in $\BH_g^*,$ where
$$\BH_g^{\tau \G_g(4m)}=\,\left\{\,\Om\in \BH_g\,|\ \,\tau(\Om)=-{\overline{\Om}}=\,\g\cdot\Om\quad \textrm{for some}\
\g\in \G_g(4m)\,\right\}$$
denotes the set of $\G_g(4m)$-real points of $\BH_g.$
\end{theorem}
\noindent
{\it Proof.} The proof can be found in \cite[pp.\,23]{GT1}. \hfill$\square$

\begin{theorem}
Let $k$ be a positive integer with $k\geq 2.$ Let
$({\mathscr F},\g)$ be a $\G_g(2^k)$-real boundary pair. Then there exists $\g_1\in \G_g(2^k)$ such that
$ {\mathscr F}^{\tau\g}= {\mathscr F}^{\tau\g_1}$ and $ {\mathscr F}^{\tau\g}$ is contained in the closure
$\overline{\BH_g^{\tau\g_1}}$ of $\BH_g^{\tau\g_1}$ in $\BH_g^*.$
\end{theorem}
\noindent
{\it Proof.} The proof can be found in \cite[pp.\,23-26]{GT1}. \hfill$\square$

\vskip 0.3cm
We may summarize the above results as follows. The Baily-Borel compactification $V(N)=\G_g(N)\ba\BH_g^*$ with $N=4m$
is stratified by finitely many strata of the form $\pi_{BB}( {\mathscr F}),$ where ${\mathscr F}$ is a rational boundary component.
Each such strata is isomorphic to the standard rational boundary component ${\mathscr F}_s\cong \BH_s$. The stratum
$\pi_{BB}( {\mathscr F})$ is called a {\it boundary} stratum of degree $s$. Let $V(N)^r$ denote the union of all boundary
strata of rank $g-r$. We define
$$V(N)_\BR^r:=\,V(N)^r\cap V(N)_\BR.$$
According to Theorem 5.4, we have
\begin{equation*}
V(N)_\BR^0 \cup V(N)_\BR^1 \,\subset\, \overline{X(N)_\BR} \,\subset\, V(N)_\BR,
\end{equation*}
where
$\overline{X(N)_\BR} $ denotes the closure of $X(N)_\BR$ in $V(N)$.

\end{section}

\vskip 1cm

\begin{section}{{\large\bf Polarized Real Tori }}
\setcounter{equation}{0}

\vskip 0.3cm
In this section we introduce the notion of polarized real tori.
\vskip 0.2cm First we review the properties of real tori briefly. We fix a positive integer $g$ in this section. Let
$T=\BR^g/\Lambda$ be a real torus of dimension $g$, where $\Lambda$ is a lattice in $\BR^g$.
$T$ has a unique structure of a smooth (or real analytic) manifold such that the canonical projection $p:\BR^g\lrt T$
is smooth (or real analytic).
We fix the standard basis $\{ e_1,\cdots,e_g\}$ for $\BR^g$. We see that $\La=\Pi \BZ^g$ for some $\Pi\in GL(g,\BR).$
A matrix $\Pi$ is called a {\it period matrix} for $T$.
Let $\BC_1^*=\{ z\in\BC\,|\
|z|=1\,\}$ be a circle. Since $T$ is homeomorphic to $\BC_1^*\times \cdots \times \BC_1^*\ (g$-times), the fundamental group is
$$\pi_1(T)\cong \pi_1(\BC_1^*)\times \cdots \times \pi_1(\BC_1^*)\cong \BZ^g.$$
We see that
$$H_k(T,\BZ)\,\cong\,\BZ^{{}_gC_k}\,\cong\,H^k(T,\BZ),\quad k=0,1,\cdots,g$$
and
$$H^*(T,\BZ)\,\cong \,\bigwedge H^1(T,\BZ)\,\cong
\bigwedge \BZ^g.$$
Thus the Euler characteristic of $T$ is zero. The mapping class group $MCG(T)$ is
$$MCG(T)=\,\textrm{Aut}(\pi_1(T))\,=\,\textrm{Aut}(\BZ^g)\,=\,GL(g,\BZ).$$
It is known that any connected compact real manifold can be embedded into the Euclidean space $\BR^d$ with large
$d$. Thus a torus $T$ can be embedded in a real projective space ${\mathbb P}^d(\BR).$
Any connected compact abelian real Lie group is a real torus. Any two real tori of dimension $g$ are
isomorphic as real Lie groups. We easily see that if $S$ is a connected closed subgroup of a real torus $T$,
then $S$ and $T/S$ are real tori and $T\cong S\times T/S.$

\vskip 0.3cm
Let $T=V/\La$ and $T'=V'/\La'$ be two real tori. A {\it homomorphism} $\phi:T\lrt T'$ is a real
analytic map compatible with the group structures. It is easily seen that a homomorphism $\phi:T\lrt T'$
can be lifted to a uniquely determined $\BR$-linear map $\Phi:V\lrt V'$. This yields an injective homomorphism
of abelian groups
\begin{equation*}
\tau_a:\textrm{Hom}(T,T')\lrt \textrm{Hom}_\BR (V,V'),\qquad \phi\longmapsto \Phi,
\end{equation*}
where $\textrm{Hom}(T,T')$ is the abelian group of all homomorphisms of $T$ into $T'$ and
$\textrm{Hom}_\BR (V,V')$ is the abelian group of all $\BR$-linear maps of $V$ into $V'$. The above
$\tau_a$ is called a real analytic representation of $\textrm{Hom}(T,T')$. The restriction $\Phi_\La$ of
$\Phi$ to $\La$ is $\BZ$-linear. $\Phi_\La$ determines $\Phi$ and $\phi$ completely. Thus we get an
injective homomorphism
\begin{equation*}
\tau_r:\textrm{Hom}(T,T')\lrt \textrm{Hom}_\BZ (\La,\La'),\qquad \phi\longmapsto \Phi_\La,
\end{equation*}
called the rational representation of $\textrm{Hom}(T,T')$.

\begin{lemma}
Let $\phi:T\lrt T'$ be a homomorphism of real tori. Then
\vskip 0.2cm (1) the image $ \textrm{Im}\,\phi$ is a real subtorus of $T'$\,;
\vskip 0.2cm (2) the kernel $\textrm{ker}\,\phi$ of $\phi$ is a closed subgroup of $T$ and the identity
component $(\textrm{ker}\,\phi)_0$ of $\textrm{ker}\,\phi$ is a real subtorus of $T$ of finite index
in $\textrm{ker}\,\phi$.
\end{lemma}
\noindent {\it Proof.} It follows from the fact that a connected compact abelian real Lie group is
a real torus. Since $\textrm{ker}\,\phi$ is compact, $\textrm{ker}\,\phi$ has only a finite number of
connected components.
\hfill $\square$

\vskip 0.3cm A surjective homomorphism $\phi: T\lrt T'$ of real tori with finite kernel is called
a {\it real isogeny} or simply an {\it isogeny}. The {\it exponent} $e(\phi)$ of an isogeny $\phi$
is defined to be the exponent of the finite group $\textrm{ker}\,\phi$, that is, the smallest positive
integer $e$ such that $e\cdot x=0$ for all $x\in \textrm{ker}\,\phi$. Two real tori are said to be
{\it isogenous} if there is an isogeny between them. It is clear that a homomorphism $\phi: T\lrt T'$
is an isogeny if and only if it is surjective and $\dim T=\dim T'$. We can see that if $\G\subset T$
is a finite subgroup, the quotient space $T/\G$ is a real torus and the natural projection
$p_\G:T\lrt T/\G$ is an isogeny.

\vskip 0.2cm
For a homomorphism $\phi:T\lrt T'$ of real tori, we define the {\it degree} of $\phi$ to be
$$\textrm{deg}\,\phi:=\,\begin{cases} \textrm{ord\,(ker}\,\phi)\quad & if \ \textrm{ker}\,\phi\ is
\ finite\,;\\ \ \ 0 & otherwise.\end{cases}$$
Let $T=V/\La$ be a real torus of dimension $g$. For any nonzero integer $n\in\BZ$, we define the isogeny
$n_T:T\lrt T$ by $n_T(x):=\,n\!\cdot\! x$ for all $x\in T$. The kernel $T(n)$ of $n_T$ is called the group
of $n$-{\it division points} of $T$. It is easily seen that $T(n)\cong \big( \BZ/n\BZ\big)^g$
because $\textrm{ker}\,n_T\,=\,{\frac 1n}\La/\La\cong \La/n\La\cong \big( \BZ/n\BZ\big)^g$. So
$\textrm{deg}\,n_T=n^g.$

\vskip 0.2cm
We put
$$\textrm{Hom}_\BQ(T,T'):=\,\textrm{Hom}(T,T')\otimes_\BZ \BQ$$
and
$$ \textrm{End}(T):=\,\textrm{Hom}(T,T),\qquad  \textrm{End}_\BQ(T):=\, \textrm{End}(T)\otimes_\BZ\BQ.$$
For any $\alpha\in\BQ$ and $\phi\in \textrm{Hom}(T,T')$, we define the degree of
$\alpha\,\phi\in \textrm{Hom}_\BQ(T,T')$ by
$$ \textrm{deg}\,(\alpha\,\phi):=\,\alpha^g\,\textrm{deg}\,\phi.$$

\begin{lemma}
For any isogeny $\phi:T\lrt T'$ of real tori with exponent $e$, there exists an isogeny $\psi:T'\lrt T$,
unique up to isomorphisms, such that $\psi\circ \phi=e_T$ and $\phi\circ \psi=e_{T'}.$
\end{lemma}
\noindent {\it Proof.} Since $\textrm{ker}\,\phi\subseteq \textrm{ker}\,e_T$, there exists a unique
map $\psi:T'\lrt T$ such that $\psi\circ \phi=e_T$. It is easy to see that $\psi$ is also an isogeny and
that $\textrm{ker}\,\psi\subseteq \textrm{ker}\,e_{T'}$. Therefore there is a unique isogeny
$\phi':T'\lrt T$ such that $\phi'\circ \psi=e_{T'}.$ Since
$$\phi'\circ e_T=\phi'\circ \psi\circ \phi =\,e_{T'}\circ \phi=\,\phi\circ e_T$$
and $e_T$ is surjective, we have $\phi'=\phi.$ Hence we obtain $\psi\circ \phi=e_T$ and $\phi\circ \psi=e_{T'}.$
\hfill $\square$

\vskip 0.2cm
According to Lemma 6.2, we see that isogenies define an equivalence relation on the set of real tori,
and that an element in $\textrm{End}(T)$ is an isogeny if and only if it is invertible in
$\textrm{End}_\BQ(T)$.

\vskip 0.3cm
For a real torus $T=V/\La$ of dimension $g$, we put $V^*:=\,\textrm{Hom}_\BR(V,\BR)$. Then the
following canonical $\BR$-bilinear form
$$\langle \ ,\ \rangle_T : V^*\times V\lrt \BR,\qquad \langle \ell,v\rangle_T:=\,\ell (v),\
\ell\in V^*,\ v\in V$$
is non-degenerate. Thus the set
$$\widehat{\La}:=\,\left\{\, \ell\in V^*\,|\ \langle \ell,\La\rangle_T\subseteq \BZ\,\right\}$$
is a lattice in $V^*$. The quotient
$$ \widehat{T}:=\,V^*/{\widehat\La}$$
is a real torus of dimension $g$ which is called the $\textsf{dual real torus}$ of $T$.
Identifying $V$ with the space of $\BR$-linear forms $V^*\lrt \BR$ by double duality,
the non-degeneracy of $\langle \ ,\ \rangle_T$ implies that $\La$ is the lattice in $V$ dual to
$\widehat\La$. Therefore we get
$$\widehat{\widehat T}=\,T.$$

Let $\phi:T_1\lrt T_2$ be a homomorphism of real tori with $T_i=V_i/\La_i\,(i=1,2)$ and with
real analytic representation $\Phi:V_1\lrt V_2$. Since the dual map $\Phi^*:V_2^*\lrt V_1^*$
satisfies the condition $\Phi^*\big( {\widehat \La}_2\big)\,\subseteq \,{\widehat \La}_1,$
$\Phi^*$ induces a homomorphism, called the dual map
\begin{equation*}
{\widehat \phi}:{\widehat T}_2\lrt {\widehat T}_1.
\end{equation*}
If $\psi:T_2\lrt T_3$ is another homomorphism of real tori, then we get
$$\widehat{\psi\circ \phi}\,=\,{\widehat\phi}\circ {\widehat\psi}.$$
If $\phi:T_1\lrt T_2$ is an isogeny of real tori, then dual map
${\widehat \phi}:{\widehat T}_2\lrt {\widehat T}_1$
is also an isogeny.

\vskip 0.53cm
\begin{definition}
A real torus $T=\BR^g/\Lambda$ with a lattice $\La$ in $\BR^g$ is said to be $\textsf{polarized}$
if the the associated complex torus $\FA=\BC^g/L$ is a polarized real abelian variety, where
$L=\,\BZ^g+\,i\,\La$ is a lattice in $\BC^g.$ Moreover if $\FA$ is a principally polarized real abelian
variety, $T$ is said to be \textsf{principally polarized}.
Let $\Phi:T\lrt\FA$ be the smooth embedding of $T$ into $\FA$ defined by
\begin{equation}
\Phi(v+\La):=\,i\,v\,+\,L, \qquad v\in\BR^g.
\end{equation}
Let $\FL$ be a polarization of $\FA$, that is, an ample line bundle over $\FA$. The pullback $\Phi^*\FL$ is
called a \textsf{polarization} of $T$. We say that a pair $(T,\Phi^*\FL)$ is a \textsf{polarized real torus}.
\end{definition}

\noindent{\bf Example 6.1.} Let $Y\in \CP$ be a $g\times g$ positive definite symmetric real matrix. Then
$\La_Y=\,Y\BZ^g$ is a lattice in $\BR^g$. Then the $g$-dimensional torus $T_Y=\BR^g/\La_Y$ is a
principally polarized real torus. Indeed,
\begin{equation*}
\FA_Y\,=\,\BC^g/L_Y, \qquad  L_Y\,=\BZ^g+\,i\,\La_Y
\end{equation*}
is a princially polarized real abelian variety.
Its corresponding hermitian form $H_Y$ is given by
\begin{equation*}
H_Y(x,y)\,=\,E_Y(i\,x,y)\,+\,i\,E_Y(x,y)\,=\,{}^tx\,Y^{-1}\,{\overline y},\qquad x,y\in\BC^g,
\end{equation*}
where
$E_Y$ denotes the imaginary part of $H_Y.$ It is easily checked that $H_Y$ is positive definite and
$E_Y(L_Y\times L_Y)\subset \BZ$\,(cf.\,\cite[pp.\,29--30]{Mf2}).
The real structure $\s_Y$ on $\FA_Y$ is a complex conjugation.

\vskip 0.3cm
\noindent{\bf Example 6.2.} Let $Q=\,\begin{pmatrix} \sqrt{2} & \ \sqrt{3} \\ \sqrt{3} & -\sqrt{5}
\end{pmatrix}$ be a $2\times 2$ symmetric real matrix of signature $(1,1)$.
Then $\La_Q=\,Q\BZ^2$ is a lattice in $\BR^2$. Then the real torus $T_Q=\,\BR^2/\La_Q$ is not
polarized because the associated complex torus $\FA_Q=\,\BC^2/L_Q$ is not an abelian variety, where
$L_Q=\BZ^2+\,i\,\La_Q$ is a lattice in $\BC^2$.

\begin{definition}
Two polarized tori $T_1\,=\,\BR^g/\La_1$ and $T_2\,=\,\BR^g/\La_2$ are said to be isomorphic if the associated
polarized real abelian varieties $\FA_1\,=\,\BC^g/L_1$ and $\FA_2\,=\,\BC^g/L_2$ are isomorphic, where
$L_i\,=\,\BZ^g\,+\,i\,\La_i\ (i=1,2),$ more precisely, if there exists a linear isomorphism $\varphi:\BC^g
\lrt \BC^g$ such that
\begin{eqnarray}
\varphi (L_1)&=&L_2,\\
\varphi_* (E_1)&=&E_2,\\
\varphi_* (\s_1) &=& \varphi\circ \s_1\circ \varphi^{-1}\,=\,\s_2,
\end{eqnarray}
where $E_1$ and $E_2$ are polarizations of $\FA_1$ and $\FA_2$ respectively, and $\s_1$ and $\s_2$ denotes
the real structures (in fact complex conjugations) on $\FA_1$ and $\FA_2$ respectively.
\end{definition}

\noindent{\bf Example 6.3.} Let $Y_1$ and $Y_2$ be two $g\times g$ positive definite symmetric real matrices. Then
$\La_i:=\,Y_i\,\BZ^g$ is a lattice in $\BR^g$ $(i=1,2)$. We let
$$T_i:=\,\BR^g/\La_i,\qquad i=1,2$$
be real tori of dimension $g$. Then according to Example 6.1, $T_1$ and $T_2$ are principally polarized real tori.
We see that $T_1$ is isomorphic to $T_2$ as polarized real tori if and only if there is an element $A\in GL(g,\BZ)$
such that $Y_2\,=\,A\,Y_1\,{}^tA.$

\vskip 0.3cm
\noindent{\bf Example 6.4.} Let $Y=\,\begin{pmatrix} \sqrt{2} & \ \sqrt{3} \\ \sqrt{3} & \sqrt{5}
\end{pmatrix}$. Let $T_Y\,=\,\BR^2/\La_Y$ be a two dimensional principally polarized torus, where $\La_Y=\,Y\BZ^2$ is a lattice
in $\BR^2.$ Let $T_Q$ be the torus in Example 6.2. Then $T_Y$ is diffeomorphic to $T_Q$. But $T_Q$ is not polarized.
$T_Y$ admits a differentiable embedding into a complex projective space but $T_Q$ does not.

\vskip 0.5cm Let $Y\in \CP$ be a $g\times g$ positive definite symmetric real matrix. Then
$\La_Y=\,Y\BZ^g$ is a lattice in $\BR^g$. We already showed that the $g$-dimensional torus $T_Y=\BR^g/\La_Y$ is a
principally polarized real torus (cf.\,Example 6.1). We know that the following complex torus
\begin{equation*}
\FA_Y\,=\,\BC^g/L_Y, \qquad  L_Y\,=\BZ^g+\,i\,\La_Y
\end{equation*}
is a princially polarized real abelian variety. We define a map $\Phi_Y: T_Y\lrt \FA_Y$ by
\begin{equation*}
\Phi_Y(a+ \La_Y):=\, i\,a\,+\,L_Y,\qquad a\in \BR^g.
\end{equation*}
Then $\Phi_Y$ is well defined and is an injective smooth map. Therefore $T_Y$ is smoothly embedded into
a complex projective space and hence into a real projective space because $\FA_Y$ can be holomorphically embedded into a complex projective space \,(cf.\,\cite[pp.\,29--30]{Mf2}).

\vskip 0.3cm
Let $\FA=\BC^g/L$ and $\FA'=\BC^{g'}/L'$ be two abelian complex tori of dimension $g$ and dimension $g'$ respectively, where
$L$ (resp. $L'$) is a lattice in $\BC^g$ (resp. $\BC^{g'}$). A homomorphism $f:\FA\lrt \FA'$ lifts to a uniquely determined
$\BC$-linear map $F:\BC^g\lrt \BC^{g'}$. This yields an injective homomorphism
\begin{equation*}
\rho_a : {\rm Hom}(\FA,\FA')\lrt {\rm Hom}_\BC(\BC^{g},\BC^{g'})=\BC^{(g'\!,g)},\qquad f\longmapsto F=\rho_a(f).
\end{equation*}
Its restriction $F|_L$ to the lattice $L$ is $\BZ$-linear and determines $F$ and $f$ completely. Therefore we get an injective
homomorphism
\begin{equation*}
\rho_r : {\rm Hom}(\FA,\FA')\lrt {\rm Hom}_\BZ(L,L'),\qquad f\longmapsto F|_L.
\end{equation*}
Let ${\widetilde \Pi}\in \BC^{(g,2g)}$ and ${\widetilde \Pi'}\in \BC^{(g'\!,2g')}$ be period matrices for $\FA$ and $\FA'$
respectively. With respect to the chosen bases, $\rho_a(f)$ (resp. $\rho_r(f)$) can be considered as a matrix in $\BC^{(g'\!,g)}$
$\big($resp. $\BZ^{(2g',2g)}\big)$. We have the following diagram\,:
\begin{equation*}
\begin{array}{ccccccccc}
\BZ^{2g} & \stackrel{{\widetilde\Pi}}{-\!\!\!-\!\!\!-\!\!\!-\!\!\!\longrightarrow} & \BC^{g}\\
\ \ \ \ \ \Big\downarrow {\scriptsize \rho_r(f)} {}  &     {}       & \ \ \ \ \ \ \Big\downarrow {\scriptsize\rho_a(f)}\\
\BZ^{2g'} & \stackrel{{\widetilde\Pi'}}{-\!\!\!-\!\!\!-\!\!\!-\!\!\!\longrightarrow} & \BC^{g'}\ ,\\
\end{array}
\end{equation*}
that is, by the equation
$$\rho_a (f)\,{\widetilde\Pi}\,=\,{\widetilde\Pi'}\,\rho_r(f).$$
Conversely any two matrices $A\in\BC^{(g'\!,g)}$ and $R\in\BZ^{(2g',2g)}$ satisfying the equation $A\,{\widetilde\Pi}={\widetilde\Pi'}R$
define a homomorphism $\FA\lrt\FA'$.

\vskip 0.3cm
For two real tori $T_1$ and $T_2$ of dimension $g_1$ and dimension $g_2$ respectively,
we let $\textrm{Ext}(T_2,T_1)$ be the set of all isomorphism classes of extensions of
$T_2$ by $T_1$ up to real analytic isomorphism. Since any two real tori of dimension $g_1+g_2$ are
isomorphic as real analytic real Lie groups, $\textrm{Ext}(T_2,T_1)$ is trivial. This leads us to
consider polarized real tori $T_1$ and $T_2$ with $T_i=\BR^{g_i}/\La_i\, (i=1,2).$ Here $\La_i$ is a lattice
in $\BR^{g_i}$ for $i=1,2.$ Let $\FA_1$ and $\FA_2$ be the polarized real abelian varieties associated to $T_1$
and $T_2$ respectively, that is,
$$\FA_i\,=\,\BC^{g_i}/L_i,\qquad L_i=\BZ^{g_i}\,+\,\La_i\BZ^{g_i},\quad i=1,2.$$
Let $\textrm{Ext}(T_2,T_1)_{\rm pt}$ be the set of all isomorphism classes of extensions of $\FA_2$ by $\FA_1$.
We can show that a homomorphism $\phi:\FA_2'\lrt \FA_2$ such that $\FA_2'$ is the real abelian variety associated to
a polarized real torus $T_2'$ induces a map
\begin{equation}
\phi^*: \textrm{Ext}(T_2,T_1)_{\rm pt}\lrt \textrm{Ext}(T_2',T_1)_{\rm pt}
\end{equation}
and that a homomorphism $\psi:\FA_1\lrt \FA_1'$ such that $\FA_1'$ is the real abelian variety associated to
a polarized real torus $T_1'$ induces a map
\begin{equation}
\psi_*: \textrm{Ext}(T_2,T_1)_{\rm pt}\lrt \textrm{Ext}(T_2,T_1')_{\rm pt}.
\end{equation}
Indeed, if
\begin{equation}
e\ :\quad 0 \lrt \FA_1 \stackrel{\iota}\longrightarrow \FA \stackrel{p}\longrightarrow \FA_2 \lrt 0
\end{equation}
is an extension in $\textrm{Ext}(T_2,T_1)_{\rm pt}$, the image $\phi^*(e)$ is defined to be the identity component of the kernel
of the homomorphism $C_{p,\phi}:\FA \times \FA_2'\lrt \FA_2$ defined by
\begin{equation*}
C_{p,\phi}(x,y):=\,p(x)-\phi(y),\qquad x\in \FA,\ y\in \FA_2'.
\end{equation*}
The dualization of the exact sequence (6.7) gives an element
${\hat e}\in \textrm{Ext}({\widehat \FA}_1,{\widehat \FA}_2).$ We define
\begin{equation}
\psi_*(e):=\, {\widehat{{\widehat \psi}^* ({\hat e})} } \quad\in \textrm{Ext}(T_2,T_1')_{\rm pt}=\textrm{Ext}(\FA_2,\FA_1').
\end{equation}
Therefore $\textrm{Ext}(\ ,\ )_{\rm pt}$ is a functor which is contravariant in the first and covariant in the second argument.

\vskip 0.2cm
We can equip the set $\textrm{Ext}(T_2,T_1)_{\rm pt}$ with the canonical group structure as follows\,: Let $e$ and
$e_\diamond$ be the extensions in $\textrm{Ext}(T_2,T_1)_{\rm pt}$ which are represented by the exact sequence (6.7) and
the following exact sequence
$$e_\diamond\ :\quad 0 \lrt \FA_1 \longrightarrow \FA_\diamond \longrightarrow \FA_2 \lrt 0.$$
The product $e\times e_\diamond$ is
represented by the exact sequence
$$e\times e_\diamond\ : \quad 0\lrt \FA_1\times \FA_1 \lrt \FA\times \FA_\diamond \lrt \FA_2\times \FA_2 \lrt 0.$$
If $\Delta:\FA_2\lrt \FA_2\times \FA_2$ is the diagonal map, $x\longmapsto (x,x),\ x\in \FA_2$ and
$\mu:T_1\times T_1\lrt T_1$ is the addition map, $(s,t)\longmapsto s+t,\ s,t\in \FA_1$, the
sum $e+e_\diamond$ is defined to be the image of $e\times e_\diamond$ under the composition
$$ \textrm{Ext}(T_2\times T_2,T_1\times T_1)_{\rm pt}\stackrel{\Delta^*}\lrt
\textrm{Ext}(T_2,T_1\times T_1)_{\rm pt} \stackrel{\mu_*}\lrt \textrm{Ext}(T_2,T_1)_{\rm pt},$$
that is,
\begin{equation}
e+e_\diamond:=\,\mu_*\Delta^* (e\times e_\diamond).
\end{equation}
We can show that $\textrm{Ext}(T_2,T_1)_{\rm pt}$ is an abelian group with respect to the addition (6.9)\,
(cf.\,\cite{BL}).

\vskip 0.3cm
Now we describe the group $\textrm{Ext}(T_2,T_1)_{\rm pt}$ in terms of period matrices. First we fix period
matrices $\Pi_1$ and $\Pi_2$ for $T_1$ and $T_2$ respectively, that is, $\La_i=\Pi_i\BZ^{g_i}$ for $i=1,2.$
 We know that $\Pi_i\in GL(g_i,\BR)$ for
$i=1,2.$ To each extension
$$e\ :\quad 0 \lrt \FA_1 \longrightarrow \FA \longrightarrow \FA_2 \lrt 0$$
in $\textrm{Ext}(T_2,T_1)_{\rm pt}$, there is associated a period matrix for $\FA$ of the form
\begin{equation}
\begin{pmatrix} {\widetilde\Pi}_1 & \sigma \\ 0 & {\widetilde\Pi}_2 \end{pmatrix},\qquad
{\widetilde \Pi}_i=(I_{g_i},\Pi_i) \ {\rm for}\ i=1,2,\ \ \s\in \BC^{(g_1,2g_2)}.
\end{equation}
Conversely it is obvious that for any $\s\in \BC^{(g_1,2g_2)}$, the matrix of the form (6.10) is a period
matrix defining an extension of $\FA_2$ by $\FA_1$ in $\textrm{Ext}(T_2,T_1)_{\rm pt}$.

\begin{lemma}
Let $\sigma$ and $\sigma'$ be elements in $\BC^{(g_1,2g_2)}.$ Suppose that $\Pi_1$ and $\Pi_2$ are period
matrices for polarized real tori $T_1$ and $T_2$ respectively. Then the period matrices
\begin{equation*}
{\widetilde\Pi}_\s\,=\,\begin{pmatrix} {\widetilde\Pi}_1 & \sigma \\ 0 & {\widetilde\Pi}_2 \end{pmatrix}
\qquad \textrm{and} \qquad
{\widetilde\Pi}_{\s'}\,=\,\begin{pmatrix} {\widetilde\Pi}_1 & \sigma' \\ 0 & {\widetilde\Pi}_2 \end{pmatrix},
\quad {\widetilde \Pi}_i=(I_{g_i},\Pi_i) \ {\rm for}\ i=1,2
\end{equation*}
define isomorphic extensions of $\FA_2$ by $\FA_1$ in $\textrm{Ext}(T_2,T_1)_{\rm pt}$ if and only if
\begin{equation}
\s'=\,\s\,+\,{\widetilde\Pi}_1 M\,+\,A\,{\widetilde\Pi}_2
\end{equation}
with some $M\in\BZ^{(2g_1,2g_2)}$ and $A\in \BC^{(g_1,g_2)}.$
\end{lemma}
\noindent {\it Proof.} Let ${\widetilde\Pi}_\s$ and ${\widetilde\Pi}_{\s'}$ define isomorphic extensions $e$ and $e'$ of $\FA_2$
by $\FA_1$\,:
\begin{equation*}
\begin{array}{ccccccccc}
e\ :\quad 0 & \longrightarrow &  \FA_1      & \longrightarrow & \FA
 & \longrightarrow & \FA_2 &  \longrightarrow & 0\\
{} & {}             &\big{|}\big{|} & {}              & \ \big\downarrow f & {}
&\big{|}\big{|} & {} & {}\\
e'\ :\quad 0 & \longrightarrow & \FA_1 & \longrightarrow &
\FA' & \longrightarrow & \FA_2 &  \longrightarrow &
0
\end{array}
\end{equation*}

\vskip 0.2cm\noindent
Then we have the following commutative diagram\,:
\begin{equation*}
\begin{array}{ccccccccc}
\BZ^{2g_1+2g_2} & \stackrel{{\widetilde\Pi}_\s}{-\!\!\!-\!\!\!-\!\!\!-\!\!\!\longrightarrow} & \BC^{g_1+g_2}\\
\ \ \Big\downarrow {\scriptsize \rho_r(f)} {}  &     {}       & \ \ \Big\downarrow {\scriptsize\rho_a(f)}\\
\BZ^{2g_1+2g_2} & \stackrel{{\widetilde\Pi}_{\s'}}{-\!\!\!-\!\!\!-\!\!\!-\!\!\!\longrightarrow} & \BC^{g_1+g_2}\\
\end{array}
\end{equation*}

\vskip 0.2cm
\noindent Therefore there are $A\in \BC^{(g_1,g_2)}$ and $M\in\BZ^{(2g_1,2g_2)}$ that satisfy the following
equation
\begin{equation}
\begin{pmatrix} I_{g_1} & A \\ 0 & I_{g_2} \end{pmatrix}{\widetilde\Pi}_\s\,=\,
{\widetilde\Pi}_{\s'}\,\begin{pmatrix} I_{g_1} & -M \\ 0 & \ I_{g_2} \end{pmatrix}.
\end{equation}
We obtain the equation (6.11) from the equation (6.12).
\vskip 0.2cm\noindent
Conversely if we have $\s$ and $\s'$ in $\BC^{(g_1,g_2)}$ satisfying the equation (6.11),
then we see easily that ${\widetilde\Pi}_\s$ and ${\widetilde\Pi}_{\s'}$ define isomorphic
extensions of $\FA_2$ by $\FA_1.$
\hfill $\square$

\vskip 0.3cm\noindent
\begin{proposition}
Let $\sigma$ and $\sigma'$ be elements in $\BC^{(g_1,2g_2)}.$ Suppose that $\Pi_1$ and $\Pi_2$ are period
matrices for real tori $T_1$ and $T_2$ respectively. Assume that the following period matrices
\begin{equation*}
{\widetilde\Pi}_\s\,=\,\begin{pmatrix} {\widetilde\Pi}_1 & \sigma \\ 0 & {\widetilde\Pi}_2 \end{pmatrix}
\qquad \textrm{and} \qquad
{\widetilde\Pi}_{\s'}\,=\,\begin{pmatrix} {\widetilde\Pi}_1 & \sigma' \\ 0 & {\widetilde\Pi}_2 \end{pmatrix},
\quad {\widetilde \Pi}_i=(I_{g_i},\Pi_i) \ {\rm for}\ i=1,2
\end{equation*}
define extensions $e$ and $e'$ of $\FA_2$ by $\FA_1$ in $\textrm{Ext}(T_2,T_1)_{\rm pt}$. Then the period matrix
\begin{equation*}
{\widetilde\Pi}_{\s+\s'}\,=\,\begin{pmatrix} {\widetilde\Pi}_1 & \sigma +\sigma'\\ 0 & {\widetilde\Pi}_2 \end{pmatrix}
\end{equation*}
defines the extension $e+e'$ in $\textrm{Ext}(T_2,T_1)_{\rm pt}$.
\end{proposition}
\noindent {\it Proof.} We denote
$$\FA\,=\,\BC^{g_1+g_2}/{\widetilde\Pi}_\s\BZ^{2g_1+2g_2}\qquad \textrm{and}\qquad
\FA'\,=\,\BC^{g_1+g_2}/{\widetilde\Pi}_{\s'}\BZ^{2g_1+2g_2}.$$
Then we have the extensions
$$ e\ :\quad 0\lrt \FA_1 \lrt \FA \lrt \FA_2 \lrt 0$$
and
$$ e'\ :\quad 0\lrt \FA_1 \lrt \FA' \lrt \FA_2 \lrt 0$$
in $\textrm{Ext}(T_2,T_1)_{\rm pt}.$ The complex torus $\FA\times \FA'$ defined by the extension $e\times e'$ in
$\textrm{Ext}(\FA_2\times \FA_2,\FA_1\times \FA_1)$
is given by
the period matrix
\begin{equation*}
{\large\Box_1}\,=\,\begin{pmatrix} {\widetilde\Pi}_1 & 0 & \sigma  & 0 \\ 0 &  {\widetilde\Pi}_1 & 0 & \s' \\
0 & 0 &  {\widetilde\Pi}_2 & 0 \\ 0 & 0 & 0 &  {\widetilde\Pi}_2 \end{pmatrix}.
\end{equation*}
Let $\Delta:\FA_2\lrt \FA_2\times \FA_2$ be the diagonal map. Then we have the induced map $\Delta^*:
\textrm{Ext}(\FA_2\times \FA_2,\FA_1\times \FA_1)\lrt \textrm{Ext}(\FA_2,\FA_1\times \FA_1)$. If
$$\Delta^*(e\times e')\,:\ 0\lrt \FA_1\times \FA_1\lrt S \lrt \FA_2\lrt 0$$
is given, the complex torus $S$ is given by a period matrix of the form
\begin{equation*}
{\large\Box_2}\,=\,\begin{pmatrix}  {\widetilde\Pi}_1 & 0 & \alpha \\ 0 &  {\widetilde\Pi}_1 & \beta  \\
0 & 0 &  {\widetilde\Pi}_2  \end{pmatrix}
\end{equation*}
with $\alpha\in \BC^{(g_1,2g_2)}$ and $\beta\in \BC^{(g_1,2g_2)}$. The homomorphism
$$\textrm{Ext}(\FA_2\times \FA_2,\FA_1\times \FA_1)\lrt \textrm{Ext}(\FA_2,\FA_1\times \FA_1),\qquad
e\times e'\longmapsto \Delta^*(e\times e')$$
corresponds to a homomorphism $S\longmapsto \FA\times \FA'$ of real tori given by the equation
\begin{equation*}
{\large\Box_1}\cdot\begin{pmatrix} I_{2g_1} & 0 & M_1 \\ 0 & I_{2g_1} & M_2 \\
0 & 0 & I_{2g_2} \\ 0 & 0 & I_{2g_2} \end{pmatrix}
\,=\,
\begin{pmatrix} I_{g_1} & 0 & A_1 \\ 0 & I_{g_1} & A_2 \\
0 & 0 & I_{g_2} \\ 0 & 0 & I_{g_2} \end{pmatrix}\cdot {\Large {\Box_2}}.
\end{equation*}
Thus we have the equations
$$\alpha\,=\,\s\,+\,{\widetilde\Pi}_1M_1\,-\,A_1{\widetilde\Pi}_2\qquad \quad\textrm{and}\qquad\quad
\beta\,=\,\s'\,+\,{\widetilde\Pi}_1M_2\,-\,A_2{\widetilde\Pi}_2.$$
According to Lemma 6.3,
\begin{equation*}
\begin{pmatrix} {\widetilde\Pi}_1 & 0 & \s \\ 0 & {\widetilde\Pi}_1 & \s'  \\
0 & 0 & {\widetilde\Pi}_2  \end{pmatrix}
\end{equation*}
is also a period matrix for $S$ respectively $\Delta^*(e\times e').$
We denote
$$e+e'\ :\quad 0 \lrt \FA_1 \lrt {\mathfrak B} \lrt \FA_2 \lrt 0.$$
A period matrix for ${\mathfrak B}$ is of the form
$$\Pi_{\mathfrak B} :=\, \begin{pmatrix} {\widetilde\Pi}_1 & \tau \\ 0 & {\widetilde\Pi}_2  \end{pmatrix},
\qquad \tau\in \BR^{(g_1,g_2)}.$$
The homomorphism $\mu_*:\Delta^*(e\times e')\longmapsto e+e'$ defines a homomorphism $S\longmapsto {\mathfrak B}$ which is
given by the equation
\begin{equation}
\begin{pmatrix} I_{g_1} & I_{g_1} & A \\ 0 & 0 & I_{g_2}  \end{pmatrix}
\begin{pmatrix} {\widetilde\Pi}_1 & 0 & \s \\ 0 & {\widetilde\Pi}_1 & \s'  \\
0 & 0 & {\widetilde\Pi}_2  \end{pmatrix}\,=\,
\begin{pmatrix} {\widetilde\Pi}_1 & \tau \\ 0 & {\widetilde\Pi}_2  \end{pmatrix}
\begin{pmatrix} I_{2g_1} & I_{2g_1} & M \\ 0 & 0 & I_{2g_2}  \end{pmatrix}
\end{equation}
with $A\in \BC^{(g_1,g_2)}$ and $M\in \BZ^{(2g_1,2g_2)}$. Comparing both sides in the equation (6.13), we obtain
$$\tau\,=\,\s+\s'\,-\,{\widetilde\Pi}_1M\,+\,A{\widetilde\Pi}_2.$$
According to Lemma 6.3, we see that
$${\widetilde\Pi}_{\s+\s'}\,=\, \begin{pmatrix} {\widetilde\Pi}_1 & \s+\s' \\ 0 & {\widetilde\Pi}_2  \end{pmatrix}$$
is a period matrix for ${\mathfrak B}$, respectively $e+e'$. \hfill $\square$

\vskip 0.3cm
Let $T_1,T_2,\Pi_1,\Pi_2,{\widetilde \Pi}_1,{\widetilde \Pi}_2, \s, \s', {\widetilde \Pi}_\s$ and ${\widetilde \Pi}_{\s'}$ be as above
in Proposition 6.1. We note that the assignment
\begin{equation*}
\s\longmapsto \BC^{g_1+g_2}/{\widetilde\Pi}_\s\BZ^{2g_1+2g_2},\qquad \s\in\BC^{(g_1,2g_2)}
\end{equation*}
induces a surjective homomorphism of abelian groups
\begin{equation}
\Phi_{\Pi_1,\Pi_2}:\BC^{(g_1,2g_2)}\lrt {\rm Ext}(T_2,T_1)_{\rm pt}.
\end{equation}
According to Lemma 6.3, we see that the kernel of $\Phi_{\Pi_1,\Pi_2}$ is given by
$$ {\rm ker}\, \Phi_{\Pi_1,\Pi_2}\,=\,{\widetilde\Pi}_1\BZ^{(2g_1,2g_2)}\,+\,\BC^{(g_1,g_2)}{\widetilde\Pi}_2.$$
Obviously the homomorphism $\Phi_{\Pi_1,\Pi_2}$ depends on the choice of the period matrices $\Pi_1$ and $\Pi_2$.

\vskip 0.3cm\noindent
\begin{proposition}
Let $T_1$ and $T_1'$ be polarized real tori of dimension $g_1$ and dimension $g_1'$ with period
matrices $\Pi_1$ and $\Pi_1'$ respectively. Let $T_2$ and $T_2'$ be polarized real tori of dimension $g_2$ and dimension $g_2'$
with period matrices $\Pi_2$ and $\Pi_2'$ respectively. Then
\vskip 0.2cm\noindent
(a) for a homomorphism $f:\FA_2'\lrt \FA_2$ such that $\FA_2'$ is the polarized real abelian variety associated to a polarized
real torus $T_2'$, the following diagram
\begin{equation*}
\begin{array}{ccccccccc}
\BC^{(g_1,2g_2)} & \stackrel{{\Phi}_{\Pi_1,\Pi_2}}{-\!\!\!-\!\!\!-\!\!\!-\!\!\!\longrightarrow} & {\rm Ext}(T_2,T_1)_{\rm pt}
\\
\ \ \Big\downarrow {\scriptsize \cdot\rho_r(f)}  {}  &     {}       & \ \ \Big\downarrow {\scriptsize f^*}\\
\BC^{(g_1,2g_2')} & \stackrel{{\Phi}_{\Pi_1,\Pi_2'}}{-\!\!\!-\!\!\!-\!\!\!-\!\!\!\longrightarrow} & {\rm Ext}(T_2',T_1)_{\rm pt}\\
\end{array}
\end{equation*}
commutes and
\vskip 0.2cm\noindent
(b) for a homomorphism $h:\FA_1\lrt \FA_1'$ such that $\FA_1'$ is the polarized real abelian variety associated to a polarized
real torus $T_1'$, the following diagram
\begin{equation*}
\begin{array}{ccccccccc}
\BC^{(g_1,2g_2)} & \stackrel{{\Phi}_{\Pi_1,\Pi_2}}{-\!\!\!-\!\!\!-\!\!\!-\!\!\!\longrightarrow} & {\rm Ext}(T_2,T_1)_{\rm pt}
\\
{\scriptsize \rho_a(h)} \cdot \Big\downarrow\ \ \ \
 {}  &     {}       & \ \ \Big\downarrow {\scriptsize h_*}\\
\BC^{(g_1',2g_2)} & \stackrel{{\Phi}_{\Pi_1',\Pi_2}}{-\!\!\!-\!\!\!-\!\!\!-\!\!\!\longrightarrow} & {\rm Ext}(T_2,T_1')_{\rm pt}\\
\end{array}
\end{equation*}
commutes.
\end{proposition}

\noindent {\it Proof.} (a) For an extension $e\in {\rm Ext}(T_2,T_1)_{\rm pt}$ we choose $\s\in \BC^{(g_1,2g_2)}$ with
$\Phi_{\Pi_1,\Pi_2}(\s)=e$ and $\s'\in \BC^{(g_1,2g_2')}$ with
$\Phi_{\Pi_1,\Pi_2'}(\s')=\,f^*(e).$ We see that the following diagram with exact rows
\begin{equation*}
\begin{array}{ccccccccc}
f^*(e)\ :\quad 0 & \longrightarrow &  \FA_1      & \longrightarrow & \FA'
 & \longrightarrow & \FA_2' &  \longrightarrow & 0\\
{} & {}             & \big{|}\big{|} & {}              & \ \big\downarrow f^* & {}
&\big\downarrow f& {} & {}\\
e\ :\quad 0 & \longrightarrow & \FA_1 & \longrightarrow &
\FA & \longrightarrow & \FA_2 &  \longrightarrow & 0
\end{array}
\end{equation*}
commutes. Thus $\s$ and $\s'$ are related by the equation
\begin{equation}
\begin{pmatrix} I_{g_1} & A \\ 0 & \rho_a(f) \end{pmatrix}
\begin{pmatrix} {\widetilde \Pi}_1 & \s' \\ 0 & {\widetilde \Pi}_2' \end{pmatrix} \,=\,
\begin{pmatrix} {\widetilde \Pi}_1 & \s \\ 0 & {\widetilde \Pi}_2 \end{pmatrix}
\begin{pmatrix} I_{2g_1} & M \\ 0 & \rho_r(f) \end{pmatrix}
\end{equation}
with $A\in \BC^{(g_1,g_2')}$ and $M\in \BZ^{(2g_1,2g_2')}.$ Comparing both sides in the equation (6.15). we get
$$\s'=\,\s \cdot\rho_r(f)\,+\,{\widetilde \Pi}_1 M- A\, {\widetilde {\Pi}}_2'.$$
According to Lemma 6.3, we have
$$\Phi_{\Pi_1,\Pi_2'}(\s')\,=\,\Phi_{\Pi_1,\Pi_2'}(\s\cdot \rho_r(f))\,=\,f^*(e).$$
This completes the proof of (a).

\vskip 0.2cm\noindent
(b) For an extension $e_\diamond\in {\rm Ext}(T_2,T_1)_{\rm pt}$ we choose $\s_\diamond\in \BC^{(g_1,2g_2)}$ with
$\Phi_{\Pi_1,\Pi_2}(\s_\diamond)=e_\diamond$ and $\s'_\diamond\in \BC^{(g_1',2g_2)}$ with
$\Phi_{\Pi_1',\Pi_2}(\s'_\diamond)=\,h_*(e_\diamond).$ We see that the following diagram with exact rows
\begin{equation*}
\begin{array}{ccccccccc}
e_\diamond\ :\quad 0 & \longrightarrow &  \FA_1      & \longrightarrow & \FA_\diamond
 & \longrightarrow & \FA_2 &  \longrightarrow & 0\\
{} & {}             & \ \ \big\downarrow h\ & {}              & \ \ \big\downarrow h_* & {}
& \!\! \big{|}\big{|}  & {} & {}\\
e\ :\quad 0 & \longrightarrow & \FA_1' & \longrightarrow &
\FA_\diamond' & \longrightarrow & \FA_2 &  \longrightarrow & 0
\end{array}
\end{equation*}
commutes. Thus $\s_\diamond$ and $\s'_\diamond$ are related by the equation
\begin{equation}
\begin{pmatrix} {\widetilde\Pi}_1' & \s_\diamond' \\ 0 &  {\widetilde\Pi}_2\end{pmatrix}
\begin{pmatrix} \rho_r(h) & M_\diamond \\ 0 & I_{2g_2} \end{pmatrix} \,=\,
\begin{pmatrix} \rho_a(h) & A_\diamond \\ 0 & I_{g_2} \end{pmatrix}
\begin{pmatrix} {\widetilde\Pi}_1 & \s_\diamond \\ 0 & {\widetilde \Pi}_2 \end{pmatrix}.
\end{equation}
with $A_\diamond\in \BC^{(g_1',g_2')}$ and $M_\diamond\in \BZ^{(2g_1',2g_2')}.$ Comparing both sides in the equation (6.16). we get
$$\s_\diamond'=\,\rho_a(h)\cdot \s_\diamond\,+\,A_\diamond {\widetilde \Pi}_2\,-\, {\widetilde {\Pi}}_1'M_\diamond.$$
According to Lemma 6.3, we get
$$h_*(e_\diamond)\,=\,h_*(\Phi_{\Pi_1,\Pi_2}(\s_\diamond))\,=\,\Phi_{\Pi_1',\Pi_2}( \rho_a(h)\cdot \s_\diamond).$$
This completes the proof of (b).
\hfill $\square$

\begin{corollary}
For $e\in {\rm Ext}(T_2,T_1)_{\rm pt}$ and $n\in\BZ$, we have
$$n_{\FA_2}^*(e)\,=\,n\cdot e\,=\,(n_{\FA_1})_*(e).$$
\end{corollary}
\vskip 0.2cm\noindent
{\it Proof.} We consider the following commutative diagram\,:
\begin{equation*}
\begin{array}{ccccccccc}
\BC^{(g_1,2g_2)} & \stackrel{{\Phi}_{\Pi_1,\Pi_2}}{-\!\!\!-\!\!\!-\!\!\!-\!\!\!-\!\!\!-\!\!\!-\!\!\!\longrightarrow} & {\rm Ext}(T_2,T_1)_{\rm pt}
\\
{\scriptsize \cdot\rho_r(n_{\FA_2})}  \Big\downarrow\ \ \ \
 {}  &     {}       & \ \ \Big\downarrow {\scriptsize (n_{\FA_2})^*}\\
\BC^{(g_1,2g_2)} & \stackrel{{\Phi}_{\Pi_1,\Pi_2}}{-\!\!\!-\!\!\!-\!\!\!-\!\!\!-\!\!\!-\!\!\!-\!\!\!\longrightarrow} & {\rm Ext}(T_2,T_1)_{\rm pt}\\
\end{array}
\end{equation*}

\noindent
Since $\rho_r(n_{\FA_2}) =\,n\,I_{2g_2},$ we get
$$(n_{\FA_2})^*(e)\,=\,{\Phi}_{\Pi_1,\Pi_2}(n\s)\,=\,n\cdot {\Phi}_{\Pi_1,\Pi_2}(\s)\,=\,n\cdot e.$$
By a similar argument, we get
$$ (n_{\FA_1})_*(e)\,=\,n\cdot e.$$
\hfill $\square$

\vskip 0.3cm
\begin{proposition} We have an isomorphism of abelian groups
$$\BC^{(g_1,g_2)}/(I_{g_1},\Pi_1)\BZ^{(2g_1,2g_2)}\begin{pmatrix} \Pi_2 \\ I_{g_2} \end{pmatrix}
\,\cong\, {\rm Ext}(T_2,T_1)_{\rm pt}.$$
\end{proposition}
\vskip 0.2cm\noindent
{\it Proof.} Let $\s=(\s_1,\s_2)\in \BC^{(g_1,2g_2)}$ with $\s_1,\s_2\in \BC^{(g_1,g_2)}$ corresponding to the extension
$e=\,{\Phi}_{\Pi_1,\Pi_2}(\s)\in {\rm Ext}(T_2,T_1)_{\rm pt}.$ By Lemma 6.3, the matrix
$$\s\,-\,\s_1 {\widetilde\Pi}_2\,=\,(\s_1,\s_2)-\,\s_1 (I_{g_2},\Pi_2)\,=\,(0,\s_2-\s_1\Pi_2)$$
corresponds to the same extension $e$. This shows that every extension in ${\rm Ext}(T_2,T_1)_{\rm pt}$ can be represented by
a matrix $\s=(0,\alpha)$ with $\alpha\in \BC^{(g_1,g_2)}.$ Hence we get a surjective homomorphism of abelian groups
$$ \BC^{(g_1,g_2)} \lrt {\rm Ext}(T_2,T_1)_{\rm pt}.$$
According to Lemma 6.3, the matrices $\alpha$ and $\alpha'\in \BC^{(g_1,g_2)}$ define the same extension if and only if
\begin{equation}
(0,\alpha-\alpha')\,=\,{\widetilde \Pi}_1
\begin{pmatrix} M_1 & M_2 \\ M_3 & M_4 \end{pmatrix}\,+\, A\, {\widetilde\Pi}_2
\end{equation}
with $\begin{pmatrix} M_1 & M_2 \\ M_3 & M_4 \end{pmatrix}\in \BZ^{(2g_1,2g_2)}$ and $A\in \BC^{(g_1,g_2)}.$
From the equation (6.17) we get
$$ A\,=\,-M_1\,-\,\Pi_1 M_3.$$
Thus we have
\begin{eqnarray*}
\alpha-\alpha'&=&\Pi_1M_4\,-\,\Pi_1M_3\Pi_2\,+\,M_2\,-\,M_1\Pi_2\\
&=& (I_{g_1},\Pi_1) \begin{pmatrix} -M_1 & M_2 \\ -M_3 & M_4 \end{pmatrix}  \begin{pmatrix} \Pi_2 \\ I_{g_2} \end{pmatrix}.
\end{eqnarray*}
This completes the proof of the above proposition.
\hfill $\square$

\end{section}

\vskip 1cm

\begin{section}{{\large\bf Line Bundles over a Polarized Real Torus}}
\setcounter{equation}{0}
\vskip 0.3cm Before we investigate complex line bundles over a real torus, we need a knowledge of holomorphic
line bundles on a complex torus. We briefly review some results on holomorphic line bundles on a complex torus (cf.\ \cite{LB},\,\cite{Mf2}).
\vskip 0.2cm
Let $X=\BC^g/L$ be a complex torus, where $L$ is a lattice in $\BC^g$. The exponential sequence
$0\lrt \BZ\lrt {\mathcal O}_X\lrt {\mathcal O}_X^* \lrt 1$ induces the long exact sequence
$$\cdots \lrt H^1(X,\BZ)\lrt H^1(X,{\mathcal O}_X)\lrt H^1(X,{\mathcal O}_X^*)
\stackrel{c_1}{\lrt} H^2(X,\BZ)\lrt\cdots $$
We recall that the N{\'e}ron-Severi group $NS(X)$ (resp. $Pic^0(X)$) is defined to be the image of $c_1$ (resp.
the kernel of $c_1$). For a hermitian form $H$ on $\BC^g$ whose imaginary part $E_H:= \textrm{Im}\,(H)$ is integral
on $L\times L$, a semi-character for $H$ is defined to be a map $\alpha:L\lrt \BC^*_1$ is defined to be a map
such that
$$\alpha(\ell_1+\ell_2)\,=\,\alpha(\ell_1)\,\alpha(\ell_2)\,e^{\pi\,i\,E_H(\ell_1,\ell_2)},\qquad
\ell_1,\ell_2\in L.$$
We let $ \textrm{Her}(L)$ be the set of all hermitian forms on $\BC^g$ whose imaginary parts are integral on
$L\times L$. For any $H\in \textrm{Her}(L)$, we denote by $ \textrm{SC}(H)$ the set of all semi-characters for $H$.
To each pair $(H,\alpha)$ with $H\in \textrm{Her}(L)$ and $\alpha\in \textrm{SC}(H)$, we associate
the automorphic factor $J_{H,\alpha}:L\times \BC^g\lrt\BC^*$ defined by
\begin{equation}
J_{H,\alpha}(\ell,z):=\,\alpha(\ell)\,e^{{\frac {\pi}2}\,H(\ell,\ell)\,+\,\,\pi\,H(z,\ell)},
\qquad \ell\in L,\ z\in \BC^g.
\end{equation}
A lattice $L$ acts on the trivial line bundle $\BC^g\times \BC$ on $\BC^g$ freely by
\begin{equation}
\ell\cdot (z,\xi)\,=\,(z+\ell, J_{H,\alpha}(\ell,z)\xi),\qquad \ell\in L,\ z\in \BC^g,\ \xi\in\BC.
\end{equation}
The quotient
\begin{equation}
\FL (H,\alpha):=\,(\BC^g\times \BC)/L
\end{equation}
obtained by the action (7.2) of $L$ has a natural structure of a holomorphic line bundle over $X$.
We note that for each such pairs $(H_1,\alpha_1)$ and $(H_2,\alpha_2)$, we have
$$J_{H_1,\alpha_1}\cdot J_{H_2,\alpha_2}\,=\,J_{H_1+H_2,\alpha_1\alpha_2}\quad \textrm{and}\quad
\FL (H_1,\alpha_1)\otimes \FL (H_2,\alpha_2)\,=\,\FL (H_1+H_2,\alpha_1\,\alpha_2).$$
Let ${\mathfrak B}(L)$ be the set of all pairs $(H,\alpha)$ with $H\in \textrm{Her}(L)$ and
$\alpha\in \textrm{SC}(H)$. Then ${\mathfrak B}(L)$ has a group structure equipped with multiplication law
$$(H_1,\alpha_1)\cdot (H_2,\alpha_2)=(H_1+H_2,\alpha_1\,\alpha_2),\quad H_i\in \textrm{Her}(L),
\ \alpha_i\in \textrm{SC}(H_i),\ i=1,2.$$
\noindent
Appell-Humbert Theorem says that we have the following canonical isomorphism of exact sequences
\begin{equation*}
\begin{array}{ccccccccc}
0 & \longrightarrow &  \textrm{Hom}(L,{\BC_1}^*)      & \longrightarrow & {\mathfrak B}(L)
 & \longrightarrow & NS(X) &  \longrightarrow & 0\\
{} & {}             &\Big\downarrow\, c_L& {}              & \Big\downarrow \,\beta_L & {}
&\Big\downarrow\,\textrm{id} & {} & {}\\
0 & \longrightarrow & Pic^0(X) & \longrightarrow &
Pic(X) & \longrightarrow & NS(X) &  \longrightarrow &
0
\end{array}
\end{equation*}

\noindent
Here $\beta_L: {\mathfrak B}(L)\lrt Pic(X)=H^1(X,{\mathcal O}_X^*)$ is the group isomorphism defined by
$$\beta_L((H,\alpha)):=\,\FL (H,\alpha),\qquad (H,\alpha)\in {\mathfrak B}(L)$$
and $c_L$ is the isomorphism induced by $\beta_L$. It is known that $NS(X)$ is a free abelian group of
rank $\rho(X)\leq g^2$, where $\rho(X)$ is the Picard number of $X$. By Appell-Humbert Theorem,
$NS(X)$ is realized in several ways as follows\,:
\begin{eqnarray*}
NS(X) &=& \,Pic(X)/Pic^0(X)\,=\, c_1\big(H^1(X,{\mathcal O}_X^*) \big)\\
&=& \,\left\{\, H:\BC^g\times \BC^g\lrt \BC\ \textrm{hermitian},\ \textrm{Im}\,(H)(L\times L)\subseteq
\BZ\,\right\} \\
&=& \,\left\{\, E:\BC^g\times \BC^g\lrt \BR\ \textrm{alternating},\ E(L\times L)\subseteq
\BZ,\ E(i\,\cdot,\,\cdot\,)\  \textrm{symmetric}\,\right\}.
\end{eqnarray*}
Let ${\widehat X}=Pic^0(X)$ be the dual complex torus of $X$. There exists the holomorphic line bundle
${\mathscr P}$ over $X\times \widehat X$ uniquely determined up to isomorphism, the so-called
{\it Poincar{\'e} bundle} satisfying the following properties (PB1) and (PB2)\,:
\vskip 0.2cm
(PB1) ${\mathscr P}|_{X\times L}\,\cong\,L$ for all $L\in {\widehat X}$, and
\vskip 0.2cm
(PB2) ${\mathscr P}|_{\{0\}\times \widehat X}$ is trivial on $\widehat X.$
\vskip 0.2cm\noindent We can see that $H^g(X,{\mathscr P})=\BC$ and $H^q(X,{\mathscr P})=0$ for all $q\neq g.$

\vskip 0.53cm
Let $T_\La\,=\,V/\La$ be a real torus of dimension $g$, where $V\cong\BR^g$ is a real vector space of dimension $g$
and $\La$ is a lattice in $V$. Let $\rho:\La\lrt \BC^*$ be a character of $\La$. Let $B:V\times V \lrt \BR$ be a
real valued symmetric bilinear form on $V$. We define the map $I_{B,\rho}:\La\times V\lrt \BC^*$ by
\begin{equation}
I_{B,\rho}(\la,v)\,=\,\rho(\la)\,e^{\pi\,B(\la,\la)\,+\, 2\,\pi\,B(v,\la)},\qquad \la\in\La,\ v\in V,\ \eta\in \BC.
\end{equation}
It is easily checked that $I_{B,\rho}$ satisfies the following equation
$$I_{B,\rho}(\la_1+\la_2,v)\,=\,I_{B,\rho}(\la_1,\la_2+v)I_{B,\rho}(\la_2,v),\qquad \la_1,\la_2\in \La,\ v\in V.$$
Then $\La$ acts on the trivial line bundle $V\times \BC$ over $V$ freely by
\begin{equation}
\lambda\cdot (v,\eta)\,=\,\left( v+\la, \,I_{B,\rho}(\la,v)\eta\right),\qquad \la\in\La,\ v\in V,\ \eta\in \BC.
\end{equation}
Thus the quotient space
\begin{equation}
L(B,\rho)\,=\,(V\times \BC)/\La
\end{equation}
has a natural structure of a smooth (or real analytic) line bundle over a real torus $T_\La$.

\begin{lemma}
Suppose $B:V\times V\lrt \BR$ is a positive definite bilinear form on $V$. We define the function
$\theta_{B,\rho}:V\lrt \BC$ by
\begin{equation}
\theta_{B,\rho}(v)\,=\,\sum_{\la\in\La} \rho(\la)^{-1}\,e^{-\pi\,B(\la,\la)\,-\,2\,\pi\, B(v,\la)},\qquad v\in V.
\end{equation}
Then map $\Theta_{B,\rho}:V\lrt V\times \BC$ defined by
\begin{equation}
\Theta_{B,\rho}(v)\,=\,(v,\theta_{B,\rho}(v)),
\qquad v\in V
\end{equation}
defines a smooth (or real analytic) global section of the line bundle $L(B,\rho)$.
\end{lemma}
\noindent
{\it Proof.} For any $\la\in \La$ and $v\in V$, we have
\begin{eqnarray*}
\theta_{B,\rho}(\la+v)\,&=&\,\sum_{\mu\in\La} \rho(\mu)^{-1}
\,e^{-\pi\,B(\mu,\mu)\,-\,2\,\pi\, B(\la+v,\mu)}\\
&=&\, \rho(\la)e^{\pi\,
B(\la,\la)\,+\,2\,\pi\,B(v,\la)}
\sum_{\mu\in \La}\rho(\la+\mu)^{-1}\,e^{-\pi\,B(\la+\mu,\la+\mu)\,-\,2\,\pi\,B(v,\la+\mu)}\\
&=&\,I_{B,\rho}(\la,v)\,\theta_{B,\rho}(v).
\end{eqnarray*}
Therefore $\Theta_{B,\rho}$ is a smooth global section of $L(B,\rho)$.
\hfill $\square$

\begin{lemma} Suppose
$B:V\times V\lrt \BR$ is a positive definite bilinear form on $V$.
Assume $B$ is integral on $\La\times \La$, that is, $B(\La\times \La)\subset \BZ.$
Then for any character $\rho:\La\lrt\BC,$ the function $f_{B,\rho}:V\lrt \BC$ defined by
\begin{equation}
f_{B,\rho}(v)\,=\,\sum_{\la\in\La} \rho(\la)\,e^{-\pi\,B(\la,\la)\,+\,2\,\pi\,i\, B(v,\la)},\qquad v\in V
\end{equation}
is invariant under the action of $\La$. Therefore $f_{B,\rho}$ may be regarded as a function on $T_\La.$
\end{lemma}
\noindent {\it Proof.} It follows immediately from the definition. \hfill $\square$

\vskip 0.52cm We see that
$$L_\La=\,\BZ^g\,+\,i\,\La \subset \BC^g$$
is a lattice in $\BC^g$. We consider the complex torus
$${\mathfrak T}_\La=\,\BC^g/L_\La.$$
We define the $\BR$-linear map $S_B:\BC^g\times \BC^g\lrt \BR$ and $E_B:\BC^g\times \BC^g\lrt \BR$ by
\begin{equation}
S_B(x,y)=\,B(x_1,y_1)\,+\,B(x_2,y_2)
\end{equation}
and
\begin{equation}
E_B(x,y)=\,B(x_2,y_1)\,-\,B(x_1,y_2),
\end{equation}
where $x=x_1\,+\,i\,x_2\in \BC^g$ and $y=y_1\,+\,i\,y_2\in \BC^g$ with $x_1,x_2,y_1,y_2\in\BR^g.$
It is easily seen that $S_B$ is symmetric and $E_B$ is alternating. We note that
$S_B(x,y)=E_B(i\,x,y)$ for all $x,y\in\BC^g.$
We define the hermitian form
$H_B:\BC^g\times \BC^g\lrt \BC$ by
\begin{equation}
H_B(x,y):=\,S_B(x,y)\,+\,i\,E_B(x,y),\qquad x,y\in\BC^g.
\end{equation}
Moreover we assume that $E_B$ is integral on $L_\La\times L_\La.$
Let $\alpha:L_\La\lrt \BC_1^*$ be a semi-character of $L_\La$ for $H_B$ such that
$$\alpha(\ell_1+\ell_2)\,=\,\alpha(\ell_1)\,\alpha(\ell_2)\,e^{\pi\,i\,E_B(\ell_1,\ell_2)},
\qquad \ell_1,\,\ell_2\in L_\La.$$
Then the mapping $J_{B,\alpha}:L_\La\times\BC^g\lrt \BC^*$ defined by
\begin{equation}
J_{B,\alpha}(\ell,z)\,=\,\alpha(\ell)\,e^{{\frac {\pi}2}\,H_B(\ell,\ell)\,+\,\pi\,H_B(z,\ell)},\qquad
\ell\in L_\La,\ z\in \BC^g
\end{equation}
is an automorphic factor for $L_\La$ on $\BC^g.$ Clearly $L_\La$ acts on the trivial line bundle
$\BC^g\times \BC$ over $\BC^g$ freely by
\begin{equation}
\ell\cdot (z,\xi)\,=\,(\ell+z, J_{B,\alpha}(\ell,z)\xi\,),\qquad
\ell\in L_\La,\ z\in \BC^g,\ \xi\in \BC.
\end{equation}
The quotient
$$ {\mathfrak L}(B,\alpha):=\,(\BC^g\times \BC)/L_\La$$
of $\BC^g\times \BC$ by $L_\La$ has a natural structure of a holomorphic line bundle over a complex torus
${\mathfrak T}_\La$.

\vskip 0.3cm In summary, to each pair $(B,\alpha)$ with s symmetric $\BR$-bilinear form $B$ on $V$ such that
$E_B$ is integral on $L_\La\times L_\La$ and a semi-character $\alpha$ for $H_B$ there is associated the holomorphic line bundle
${\mathfrak L}(B,\alpha)$ over ${\mathfrak T}_\La$.

\vskip 0.3cm
We assume that $B$ is non-degenerate of signature $(r,s)$ with $r+s=g.$ Then the hermitian form $H_B$ is
also non-degenerate of signature $(r,s).$ Moreover we assume that $E_B$ is integral on $L_\La\times L_\La$.
Under these assumptions, Matsushima \cite{Mat} proved that the cohomology group
$H^q(\FT_\La,{\mathfrak L}(B,\alpha))= 0$ for all $q\neq s$ and that $H^s(\FT_\La,{\mathfrak L}(B,\alpha))$
is identified with the complex vector space of all $C^\infty$ functions $f$ on $\BC^g$ satisfying the
following conditions\,:
\vskip 0.2cm
(a) $f$ is a differentiable theta functions for the automorphic factor $J_{B,\alpha}$\,; namely we have
$$f(\ell +z)\,=\,J_{B,\alpha}(\ell,z)\,f(z),\qquad \ell\in L_\La,\ z\in \BC^g,$$

(b) ${{\partial f}\over {\partial {\overline z}_i}}=0$\ \ for all $i\in \{1,2,\cdots,r \}$
and
$$ {{\partial f}\over {\partial { z}_i}}\,+\,\pi\,{\overline z}_i\,f\,=\,0 \qquad
\textrm{for all}\ i\in \{r+1,\cdots,g \},$$
where $(z_1,\cdots,z_g)$ is the coordinate of $\BC^g$ determined by a privileged basis of $\BC^g$
for the hermitian form $H_B.$ We can show that the cohomology group
$H^s\big(\FT_\La,{\mathfrak L}(B,\alpha)^{\otimes 3}\big)$ defines a smooth embedding of
$\FT_\La$ into the projective space ${\mathbb P}^d(\BC)$ with
$d+1=\dim H^s\big(\FT_\La,{\mathfrak L}(B,\alpha)^{\otimes 3}\big)$ which is holomorphic in
$z_1,\cdots,z_r$ and anti-holomorphic in $z_{r+1},\cdots,z_g$ (cf.\,\cite{Mat} and \cite{Mf2}).

\vskip 0.2cm We consider the {\it canonical semi-character} $\g_{\La,B}$ of $L_\La$ defined by
\begin{equation}
\g_{\La,B}(\kappa+\,i\,\la):=\,e^{\pi\,i\,E_B(\kappa,\,i\,\la)},\qquad \kappa\in\BZ^g,\ \la\in\La.
\end{equation}
Then $\g_{\La,B}$ defines the holomorphic line bundle ${\mathfrak L}(B,\g_{\La,B})$ over a complex torus
$\FT_\La$. For any $z\in \FT_\La$ we denote by $T_z$ the translation of $\FT_\La$ by $z$. Let $\pi_\La:\BC^g\lrt
\FT_\La$ be the natural projection. Then there exists an element $c_{\La,B,\alpha}$ of $\BC^g$ such that
\begin{equation}
{\mathfrak L}(B,\alpha)\,=\,T^{\ast}_{\pi_\La (c_{\La,B,\alpha})}{\mathfrak L}(B,\g_{\La,B}).
\end{equation}
$c_{\La,B,\alpha}$ is called a {\it characteristic} of the holomorphic line bundle ${\mathfrak L}(B,\alpha).$
We refer to \cite{LB} for detail.

\vskip 0.53cm
Now we let $T_\La=V/\La$ be a polarized real torus of dimension $g$. Its associated polarized real abelian variety
\begin{equation*}
\FA_\La\,=\,\BC^g/L_\La,\qquad L_\La=\BZ^g+\,i\,\La
\end{equation*}
admits a positive definite hermitian form $H_\La$ on $\BC^g$ whose imaginary
part $\textrm{Im}\,(H_\La)$ is integral on $\La\times \La$ (cf.\,\cite[p.\,35]{Mf2}).
We write
\begin{equation*}
H_\La(x,y)\,=\,S_\La (x,y)\,+\,i\,E_\La (x,y),\qquad x,y\in\BC^g,
\end{equation*}
where $S_\La$ and $E_\La$ are the real part (resp. imaginary part) of $H_\La$ respectively. We know that
$S_\La$ is a real valued symmetric bilinear form on $V$ and $E_\La$ is a real valued alternating
bilinear form on $V$. Let $\alpha_{\La}:L_\La\lrt \BC_1^*$ be a canonical semi-character of $L_\La$ defined by
\begin{equation}
\alpha_\La(\kappa+\,i\,\la):=\,e^{\pi\,i\,E_\La(\kappa,\,i\,\la)},\qquad \kappa\in\BZ^g,\ \la\in\La.
\end{equation}

\noindent
We let $J_{H_\La,\alpha_\La}: L_\La \times \BC^g\lrt \BC^*$ be the automorphic factor for $\La$ on $V$ that
is canonically given by
\begin{equation}
J_{H_\La,\alpha_\La} (\ell,z)\,=\,\alpha_\La(\ell)\,e^{{\frac {\pi}{2}}\,H_\La(\ell,\ell)\,+\,\pi\,H_\La (z,\ell)},
\qquad \ell\in L_\La,\ z\in\BC^g.
\end{equation}
Obviously $L_\La$ acts on $\BC^g\times \BC$ freely by
\begin{equation*}
\ell\cdot (z,\xi)\,=\,\big( \ell +z, \,J_{H_\La,\alpha_\La} (\ell,z)\,\xi\, \big),
\qquad \ell\in L_\La,\ z\in\BC^g,\ \xi\in\BC.
\end{equation*}
So the quotient space
\begin{equation}
{\mathfrak L}(H_\La,\alpha_\La):=\,\big( \BC^g\times \BC\big)/L_\La
\end{equation}
has a natural structure of a holomorphic line bundle over an abelian variety $\FA_\La$.

\vskip 0.2cm
Now we define the map $\Phi_\La:T_\La \lrt \FA_\La$ by
\begin{equation}
\Phi_\La (v+ \La):=\,i\,v\,+\, L_\La,\qquad v\in \BR^g.
\end{equation}
$\Phi_\La$ is a well defined injective mapping. It is well known that
$H^q(\FA_\La, \FL (H_\La,\alpha_\La))=0$ for all $q\neq 0$ and that
the space of global holomorphic
sections of $\FL (H_\La,\alpha_\La)^{\otimes n}$ for any positive integer $n\geq 3$ give a holomorphic
embedding of $\FA_\La$ as a closed complex manifold in a projective complex manifold ${\mathbb P}^d(\BC)$
\,(cf.\,\cite[pp.\,29--33]{Mf2}). Therefore we have a differentiable embedding of $T_\La$ into a
complex projective space ${\mathbb P}^d(\BC)$ and hence into a real projective space ${\mathbb P}^N(\BR)$
with large enough $N>0.$

\vskip 0.2cm We will characterize the pullback $L(\alpha_\La):=\Phi_\La^*{\mathfrak L}(H_\La,\alpha_\La)$.
We first define the automorphic factor $I_{\alpha_\La}:\La\times \BR^g\lrt \BC^*$ by
\begin{equation}
I_{\alpha_\La}(\la,v):=\,\alpha_\La (i\,\la)\,e^{{\frac {\pi}2}\, H_\La (\la,\la)\,+\,\pi H_\La (v,\la)},
\qquad \la\in\La,\ v\in\BR^g.
\end{equation}
This automorphic factor $I_{\alpha_\La}$ yields the smooth (or real analytic) line bundle over $T_\La$ which is
nothing but the pullback $L(\alpha_\La).$ We observe that if $\theta$ is a holomorphic theta function for ${\mathfrak L}(H_\La,\alpha_\La)$,
then the function $f_\theta :\BR^g\lrt \BC$ defined by $f_\theta (v):=\,\theta(i\,v),\ v\in \BR^g$ defines
a global smooth (or real analytic) section of $L(\alpha_\La).$

\vskip 0.2cm Now we will show that a holomorphic line bundle ${\mathfrak L}(H_\La,\alpha_\La)$ over $\FA_\La$
naturally yields a smooth line bundle over a polarized torus $T_\La.$
Let $B_\La$ be the restriction of $S_\La$ to $\BR^g\times \BR^g.$
First we define the automorphic factor
$I_{B_\La,\alpha_\La}:\La\times \BR^g\lrt \BC^*$ by
\begin{equation}
I_{B_\La,\alpha_\La}(\la,v):=\,\alpha_\La (2\,i\,\la)\,e^{\pi B_\La (\la,\la)\,+\,2\,\pi B_\La (v,\la)},
\qquad \la\in\La,\ v\in\BR^g.
\end{equation}
This automorphic factor $I_{B_\La,\alpha_\La}(\la,v)$ yields a smooth line bundle
\begin{equation}
L(B_\La,\alpha_\La):=\,(\BR^g\times\BC)/\La
\end{equation}
over a polarized real torus $T_\La.$ Since $B_\La$ is positive definite, according to Lemma 4.1, the space
$\Gamma(T_\La,L(B_\La,\alpha_\La))$ of smooth (or real analytic) global sections of $L(B_\La,\alpha_\La)$
is not zero. If $B_\La$ is integral on $\La\times \La$, according to Lemma 4.2, we see that
the function $f_{\La,\alpha_\La}:\BR^g\lrt \BC$ defined by
$$f_{\La,\alpha_\La}(v)\,=\,\sum_{\la\in\La} \alpha_\La (2\,i\,\la)\,e^{-\pi\,B_\La (\la,\la)\,+\,
2\,\pi\,i\,B_\La (v,\la)},\qquad v\in \BR^g$$
is a function on $T_\La.$

\vskip 0.3cm So far we have proved the following.
\begin{theorem}
Let $T_\La=V/\La$ be a polarized real torus of dimension $g$. Then there is a smooth line bundle
$L(B_\La,\alpha_\La)$ over $T_\La$ which is constructed canonically by (7.23).
\end{theorem}

\vskip 0.3cm \noindent
{\bf Example 7.1.} Let $Y\in \CP$ be a $g\times g$ positive definite symmetric real matrix. Then
$\La_Y=\,Y\BZ^g$ is a lattice in $\BR^g$. Then the $g$-dimensional torus $T_Y=\BR^g/\La_Y$ is a
principally polarized real torus. Indeed,
\begin{equation*}
\FA_Y\,=\,\BC^g/L_Y, \qquad  L_Y\,=\BZ^g+\,i\,\La_Y
\end{equation*}
is a princially polarized real abelian variety\,(cf.\,Example 6.1).
Its corresponding hermitian form $H_Y$ is given by
\begin{equation*}
H_Y(x,y)\,=\,S_Y(x,y)\,+\,i\,E_Y(x,y)\,=\,{}^tx\,Y^{-1}\,{\overline y},\qquad x,y\in\BC^g,
\end{equation*}
where
$S_Y$ and $E_Y$ denote the real part and the imaginary part of $H_Y$ respectively. Let
$\alpha:L_Y\lrt \BC_1^*$ be a semi-character of $L_Y$. To a pair $(H_Y,\alpha)$
the canonical automorphic factor $J_{Y,\alpha}:L_Y\times \BC^g\lrt \BC$ is associated by
$$ J_{Y,\alpha}(\ell,z)\,=\,\alpha(\ell)\,e^{{\frac 12}\,\pi\,i\,\,{}^t{\ell}\,\,Y^{-1}\,{\overline\ell}
\,\,+\,\,
\pi\,i\,{}^tz Y^{-1}\,{\overline\ell} },\qquad \ell\in L_Y,\ z\in \BC^g.$$
The associated automorphic factor $I_{Y,\alpha}:\La_Y\times\BR^g\lrt \BC^*$ is given by
$$I_{Y,\alpha}(\la,v)\,=\,\alpha(2\,i\,\la)\,e^{\pi\,{}^t\la\,Y^{-1}\la\,+\,2\,\pi\,{}^tv\,Y^{-1}
\la},\qquad \la\in \La_Y,\ v\in\BR^g.$$
We get the associated line bundle
$$L(B_Y,\alpha)\,=\,(\BR^g\times \BC)/\La_Y$$
given by $I_{Y,\alpha},$ where $B_Y$ is the restriction of $S_Y$ to $\BR^g\times \BR^g$.
Then the function $\theta_{Y,\alpha}:\BR^g\lrt \BC$ defined by
$$\theta_{Y,\alpha}(v)\,=\,\sum_{\la\in \La_Y} \alpha(2\,i\,\la)\,
e^{-\pi\,{}^t\la\,Y^{-1}\la\,-\,2\,\pi\,{}^tv\,Y^{-1}
\la},\qquad v\in\BR^g$$
yields a smooth global section of $L(B_Y,\alpha)$ over a real torus $T_Y$.
The canonical semi-character $\alpha_Y$ of $L_Y$ is given by
$$\alpha_Y(\kappa+\,i\,\la)\,=\,e^{-\pi\,i\,{}^t\kappa Y^{-1}\la},\qquad \kappa\in \BZ^g,\ \la\in \La_Y.$$

\end{section}

\vskip 1cm

\begin{section}{{\large\bf Moduli Space for Principally Polarized Real Tori}}
\setcounter{equation}{0}

We have
the natural action of $GL(g,\BR)$ on ${\mathcal P}_g$ given by
\begin{equation}
A*Y=\,AY\,{}^t\!A,\qquad A\in GL(g,\BR),\ Y\in {\mathcal P}_g.
\end{equation}

We put ${\mathfrak G}_g=GL(g,\BZ)$ (see Notations in the introduction).
The fundamental domain ${\mathfrak R}_g$ for ${\mathfrak G}_g\ba \CP$ which was
found by H. Minkowski \cite{Mink} is defined as a subset of $\CP$
consisting of $Y=(y_{ij})\in \CP$ satisfying the following
conditions (M.1)--(M.2)\ (cf. \cite[p.\,191]{Igu} or
\cite[p.\,123]{Ma2}): \vskip 0.1cm (M.1)\ \ \ $aY\,^ta\geq
y_{kk}$\ \ for every $a=(a_i)\in\BZ^g$ in which $a_k,\cdots,a_g$
are relatively prime for $k=1,2,\cdots,g$. \vskip 0.1cm (M.2)\ \ \
\ $y_{k,k+1}\geq 0$ \ for $k=1,\cdots,g-1.$ \vskip 0.1cm We say
that a point of $\Rg$ is {\it Minkowski reduced} or simply {\it
M}-{\it reduced}. $\Rg$ has the following properties (R1)-(R6):
\vskip 0.1cm (R1) \ For any $Y\in\CP,$ there exist a matrix $A\in
GL(g,\BZ)$ and $R\in\Rg$ such that $Y=R[A]$\,(cf.
\cite[p.\,191]{Igu} or \cite[p.\,139]{Ma2}). That is,
$$GL(g,\BZ)\circ \Rg=\CP.$$
\indent (R2)\ \ $\Rg$ is a convex cone through the origin bounded
by a finite number of hyperplanes. $\Rg$ is closed in $\CP$
(cf.\,\cite[p.\,139]{Ma2}).

\vskip 0.1cm (R3) If $Y$ and $Y[A]$ lie in $\Rg$ for $A\in
GL(g,\BZ)$ with $A\neq \pm I_g,$ then $Y$ lies on the boundary
$\partial \Rg$ of $\Rg$. Moreover $\Rg\cap (\Rg [A])\neq
\emptyset$ for only finitely many $A\in GL(g,\BZ)$
(cf.\,\cite[p.\,139]{Ma2}). \vskip 0.1cm (R4) If $Y=(y_{ij})$ is
an element of $\Rg$, then
$$y_{11}\leq y_{22}\leq \cdots \leq y_{gg}\quad \text{and}\quad
|y_{ij}|<{\frac 12}y_{ii}\quad \text{for}\ 1\leq i< j\leq g.$$
\indent We refer to \cite[p.\,192]{Igu} or
\cite[pp.\,123-124]{Ma2}.

\vskip 0.3cm
For $Y=(y_{ij})\in \CP,$ we put
\begin{equation*}
dY=(dy_{ij})\qquad\text{and}\qquad \PY\,=\,\left(\, {
{1+\delta_{ij}}\over 2}\, { {\partial}\over {\partial y_{ij} } }
\,\right).
\end{equation*}

\vskip 0.2cm For a fixed element $A\in GL(g,\BR)$, we put
$$Y_*=A\star Y=AY\,^t\!A,\quad Y\in \CP.$$
Then
\begin{equation}
dY_*=A\,dY\,^t\!A \quad \textrm{and}\quad {{\partial}\over {\partial
Y_*}}=\,^t\!A^{-1} \Yd\, A^{-1}.
\end{equation}

We consider the following differential operators
\begin{equation}
D_k= \s\left( \left( Y\Yd \right)^k\right),\quad
k=1,2,\cdots,g,
\end{equation}

\noindent where $ \s(M)$ denotes the trace of a square
matrix $M$. By Formula (8.2), we get
\begin{equation*}
\left( Y_* {{\partial}\over {\partial Y_*}}\right)^i=\,A\,\left(
Y\Yd\right)^i A^{-1}
\end{equation*}

\noindent for any $A\in GL(g,\BR)$. So each $D_i$ is invariant
under the action (8.1) of $GL(g,\BR)$.

\vskip 0.2cm Selberg \cite{Sel} proved the following.

\begin{theorem}
The algebra ${\mathbb D}(\CP)$ of all differential operators on
$\CP$ invariant under the action (8.1) of $GL(g,\BR)$ is generated
by $D_1,D_2,\cdots,D_g.$ Furthermore $D_1,D_2,\cdots,D_g$ are algebraically independent and ${\mathbb
D}(\CP)$ is isomorphic to the commutative ring $\BC
[x_1,x_2,\cdots,x_g]$ with $g$ indeterminates
$x_1,x_2,\cdots,x_g.$
\end{theorem}

\begin{proof} The proof can be found in \cite[pp.\,64-66]{Ma2}.\end{proof}
We can see easily that
\begin{equation*}  ds^2=\,\s( (Y^{-1}dY)^2)  \end{equation*}
is a $GL(g,\BR)$-invariant Riemannian metric on $\CP$ and its
Laplacian is given by
\begin{equation*}
\Delta=   \s\left( \left( Y\PY\right)^2\right).
\end{equation*}
We also can see that
\begin{equation*}
d\mu_g(Y)=(\det Y)^{-{ {g+1}\over2 } }\prod_{i\leq j}dy_{ij}
\end{equation*}
is a $GL(g,\BR)$-invariant volume element on $\CP$. The metric
$ds^2$ on $\CP$ induces the metric $ds_{\mathcal R}^2$ on $\Rg.$
Minkowski\,\cite{Mink} calculated the volume of $\Rg$ explicitly.


\vskip 0.3cm
$\CP$ parameterizes principally polarized real tori of dimension $g$. The Minkowski domain
$\Rg$ is the moduli space for isomorphism classes of principally polarized real tori of dimension $g$.
According to (R2) we see that $\Rg$ is a semi-algebraic set with real analytic structure.
Unfortunately $\Rg$ does not admit the structure of a real algebraic variety and does not admit
a compactification which is defined over the rational number field $\BQ$.
We see that
$\Rg$ is real analytically isomorphic to the semi-algebraic subset ${\mathscr X}_{(0,1)}^g$ of $\XG$.
We define the embedding $\Phi_g: \CP\lrt \HG$ by
\begin{equation}
\Phi_g(Y)\,=\,i\,Y,\qquad Y\in \CP.
\end{equation}
We have the following inclusions
$$\CP\stackrel{\Phi_g}{\lrt} i\,\CP\,\,\hookrightarrow \,\HG\,\, \hookrightarrow \,\, \BH_g \subset\,\, \BH_g^*.$$
${\mathfrak G}_g$ acts on $\CP$ and $i\,\CP$, $\G_g^{\star}$ acts on $\HG$, and $\G_g$ acts on $\BH_g$ and $\BH_g^*$.
It might be interesting to characterize the boundary points of the closure of $i\,\CP$ (or $\CP$) in $\BH_g^*$ explicitly.
In Section 5 we reviewed
Silhol's compactification $\OXG$ of $\XG$ which is analogous to the Satake-Baily-Borel compactification.
The theory of automorphic forms on $\Rg$ has been developed by Selberg \cite{Sel}, Maass \cite{Ma2} et al
past a half century. According to Theorem 5.1, $\OXG$ is a connected compact Hausdorff space containing $\XG$ as an open dense subset
of $\OXG$. But $\OXG$ does not admit an algebraic structure.

\vskip 0.5cm
For any positive integer $h\in \BZ^+$, we let
\begin{equation}
GL_{g,h}:=\,GL(g,\BR)\ltimes \BR^{(h,g)}
\end{equation}
be the semi-direct product of $GL(g,\BR)$ and $\BR^{(h,g)}$ with the multiplication law
\begin{equation}
(A,a)\cdot (B,b)\,=\,(AB,a\,{}^tB^{-1}+b),\qquad A,B\in GL(g,\BR),\ \ a,b\in \BR^{(h,g)}.
\end{equation}
Then we have the {\it natural action} of $GL_{g,h}$ on the Minkowski-Euclid space $\CP\times \BR^{(h,g)}$
defined by
\begin{equation}
(A,a)\cdot (Y,\zeta)\,=\,\big(AY\,{}^t\!A,\,(\zeta+a)\,{}^t\!A \big), \qquad (A,a)\in GL_{g,h},\ Y\in\CP,\ \zeta\in \BR^{(h,g)}.
\end{equation}

\vskip 0.3cm
For a variable $(Y,V)\in \CP \times \BR^{(h,g)}$ with $Y\in \CP$ and $V\in \Rmn$, we
put $$Y=(y_{\mu\nu})\ \text{with}\ y_{\mu\nu}=y_{\nu\mu},\ \
V=(v_{kl}),$$   $$ dY=(dy_{\mu\nu}),\ \ dV=(dv_{kl}),$$
$$[dY]=\prod_{\mu\leq\nu}dy_{\mu\nu},\ \ \ \ \
[dV]=\prod_{k,l}dv_{kl},$$ and
$$\Yd\,=\left( { {1+\delta_{\mu\nu}}\over 2 } {
{\partial}\over {\partial y_{\mu\nu}} }\right),\ \ \ \Vd\,=\left({
{\partial}\over {\partial v_{kl}} } \right),$$ where $1\leq
\mu,\nu,\,l\leq g$ and $1\leq k\leq h.$

\vskip 0.1cm \noindent
\begin{lemma}
For all two positive real numbers $A$ and $B$,
the following metric $ds^2_{g,h;A,B}$ on $\PR$ defined by
\begin{equation}
ds^2_{g,h;A,B}=\,A\,\s(Y^{-1}dY\,Y^{-1}dY)\,+\,B\,\s(Y^{-1}\,^t(dV)\,dV)
\end{equation}
is a Riemannian metric on $\PR$ which is invariant under the action (8.7) of $\Gnm$. The
Laplacian $\Delta_{g,h;A,B}$ of $(\PR,\,ds^2_{g,h;A,B})$ is given by
$$\Delta_{g,h;A,B}=\,{\frac 1A}\,\s \left(\left( Y {{\partial}\over {\partial
Y}}\right)^2\right) -{\frac h{2\,A}}\,\s\left( Y {{\partial}\over
{\partial Y}}\right)\,+ \,{\frac 1B}\,\sum_{k\leq p} \left( \left({{\partial}\over
{\partial V}}\right) Y
{}^{{}^{{}^{{}^\text{\scriptsize $t$}}}}\!\!\!
\left({{\partial}\over {\partial
V}}\right)\right)_{kp}.$$ Moreover $\Delta_{g,h;A,B}$ is a
differential operator of order 2 which is invariant under the
action (8.7) of $\Gnm.$\end{lemma}

\newcommand\ya{Y^{\ast}}
\newcommand\va{V^{\ast}}
\begin{proof} For a fixed element $(A,a)\in \Gnm$, we set
$$(Y^{\ast},V^{\ast})=(A,a)\cdot (Y,V).$$
Then
$$Y^*=\,A\,Y\,^t\!A,\qquad V^*=(V+a)\,^t\!A.$$
The first statement follows immediately from the fact that
$$dY^{\ast}=\,A\,dY\,^t\!A\ \ \ \ \text{and}\ \ \ dV^{\ast}=dV\,^t\!A.$$
Using the formula (13) in \cite[p.\,245]{Hel}, we can compute the
Laplacian $\Delta_{g,h;A,B}$ of $(\PR,\,ds^2_{g,h;A,B})$. The last
statement follows from the fact that
$${{\partial}\over {\partial
\ya}}=\,^t\!A^{-1}\Yd\,A^{-1},\ \ \ {{\partial}\over {\partial
\va}}=\Vd\cdot A^{-1}.$$
\end{proof}

\begin{lemma} The following volume element $dv_{g,h}(Y,V)$ on $\PR$ defined by
\begin{equation}
dv_{g,h}(Y,V)=(\det\,Y)^{-{ {g+h+1}\over 2} }[dY][dV]
\end{equation} is invariant under the action (8.7)
of $\Gnm$.
\end{lemma}

\begin{proof}
For a fixed element $(A,a)\in \Gnm,$ we set
$$(Y^{\ast},V^{\ast})=\,(A,a)\cdot
(Y,V)=\,(\,AY\,^tA,(V+a)\,^tA).$$ Let ${{\partial (\Ys,\Vs)}\over
{\partial (Y,V)}}$ be the Jacobian determinant of the action (8.7)
of $\Gnm$ on $\PR$. It is known that the Jacobian determinant of
the action $Y\longmapsto \Ys$ is given by $${{\partial (\Ys)}\over
{\partial (Y)}}=\,(\,\det A\,)^{g+1}.$$ Take the diagonal
matrix $g=(d_1,\cdots,d_g)$ with distinct real numbers $d_i$.
Obviously if $a=(a_{kl}),\,V=(v_{kl})$ and $\Vs=(v_{kl}^{\ast}),$
then $v_{kl}^{\ast}=(v_{kl}+a_{kl})d_l$ for all $k,l.$ Thus we
have
\begin{equation}
{{\partial (\Vs)}\over {\partial (V)}}=\,(d_1\,\cdots \,d_g)^h
=\,(\,\text{det}\,A\,)^h.\end{equation} Since the set of all
$g\times g$ real matrices whose eigenvalues are all distinct is
everywhere dense in $GL(g,\BR)$ , and ${{\partial (\Vs)}\over
{\partial (V)}}$ is a rational function, the relation (8.10) holds
for any $A\in GL(g,\BR).$ It is easy to see that
$${{\partial (\Ys,\Vs)}\over {\partial (Y,V)}}=
{{\partial(\Ys)}\over {\partial(Y)}}\cdot {{\partial(\Vs)}\over
{\partial (V)}}.$$ Thus we obtain
$$[d\Ys][d\Vs]=\,|\,\textrm{det}\,A\,|^{g+h+1}[dY][dV].$$ Since
$\text{det}\,\Ys\,=\,(\,\text{det}\,A\,)^2\,\text{det}\,Y,$ we
have $$(\,\text{det}\,\Ys\,)^{-{{g+h+1}\over 2}}[d\Ys][d\Vs]=\,
(\,\text{det}\,Y\,)^{-{{g+h+1}\over 2}}[dY][dV].$$ Hence the
volume element (8.9) is invariant under the action (8.7).
\end{proof}

\vskip 0.3cm
It is known that
$$ d\mu_g(Y):=\,(\det Y)^{-{{g+1}\over 2}} \,[dY]$$
is a volume element on ${\mathcal P}_g$ invariant under the action (8.1) of $GL(g,\BR)$
\,(cf.\,\cite[p.\,23]{Ma2}).
Let $r$ be a positive integer with $0< r< g.$ We define a bijective transformation
\begin{equation*}
{\mathcal P}_g\lrt {\mathcal P}_r \times {\mathcal P}_s \times \BR^{(s,r)},\ \ r+s=g,\quad
Y\longmapsto (F,G,H)
\end{equation*}
by
\begin{equation}
Y=\,\begin{pmatrix} F & 0 \\ 0 & G \end{pmatrix}
\left[ \begin{pmatrix} I_r & 0 \\ H & I_s \end{pmatrix} \right],\quad
Y\in {\mathcal P}_n,\ F\in {\mathcal P}_r,\ G\in {\mathcal P}_s,\ H\in \BR^{(s,r)}.
\end{equation}

According to \cite[pp.\,24-26]{Ma2}, we obtain
\begin{equation}
[dY]=\,(\det G)^r\,[dF][dH][dG],
\end{equation}
equivalently
\begin{equation}
d\mu_g(Y)=\,(\det F)^{-{\frac s2}}\,(\det G)^{\frac r2}\,d\mu_r(F)\,d\mu_s(G)\,[dH].
\end{equation}

\noindent Therefore we get
\begin{equation}
dv_{g,h}(Y,V)=\,(\det F)^{-{\frac {h+s}2}}\,(\det G)^{\frac {r-h}2}
\,d\mu_r(F)\,d\mu_s(G)\,[dH]\,[dV].
\end{equation}

\vskip 0.3cm
Similarly if $Y\in {\mathcal P}_g,\ g=r+s$ with $0<r<g$, we write
\begin{equation}
Y=\,\begin{pmatrix} P & 0 \\ 0 & Q \end{pmatrix}
\left[ \begin{pmatrix} I_r & R \\ 0 & I_s \end{pmatrix} \right],\quad
Y\in {\mathcal P}_g,\ P\in {\mathcal P}_r,\ Q\in {\mathcal P}_s,\ R\in \BR^{(r,s)}.
\end{equation}

\noindent According to \cite[pp.\,27]{Ma2}, we obtain
\begin{equation}
[dY]=\,(\det P)^s\,[dP][dQ][dR],
\end{equation}
equivalently
\begin{equation}
d\mu_g(Y)=\,(\det P)^{{\frac s2}}\,(\det G)^{-{\frac r2}}\,d\mu_r(P)\,d\mu_s(Q)\,[dR].
\end{equation}

\noindent Therefore we get
\begin{equation}
dv_{g,h}(Y,V)=\,(\det P)^{{\frac {s-h}2}}\,(\det G)^{-{\frac {r+h}2}}
\,d\mu_r(P)\,d\mu_s(Q)\,[dR]\,[dV].
\end{equation}

\noindent The coordinates $(F,G,H)$ or $(P,Q,R)$ are called the partial Iwasawa coordinates on ${\mathcal P}_g.$

\vskip 0.5cm
\begin{theorem} Any geodesic through the origin $(I_g,0)$ is of
the form
$$\g(t)=\left(\,\la(2t)[k],\,Z\left(\int_0^t
\la (t-s)ds\right)[k]\right),$$ where $k$ is a fixed element of
$O(g),\ Z$ is a fixed $h\times g$ real matrix, $t$ is a real
variable, $\la_1,\la_2,\cdots,\la_g$ are fixed real numbers but
not all zero and
$$\la(t):=\text{diag}\,(e^{\la_1t},\cdots,e^{\la_g t}).$$
Furthermore, the tangent vector $\g'(0)$ of the geodesic $\g(t)$
at $(I_g,0)$ is $(D[k],Z),$ where
$D=\text{diag}\,(2\la_1,\cdots,2\la_g).$
\end{theorem}

\begin{proof} Let $W=(X,Z)$ be an element of $\fp$ with $X\neq 0.$
Then the curve
$$\al(t)=\,\text{exp}\,tW\,=\left(\,e^{tX},\,Z\left(\int_0^t
e^{-sX}ds \right)\right),\ \ \ t\in \BR$$ is a geodesic in $\Gnm$
with $\al'(0)=W$ passing through the identity of $\Gnm.$ Thus the
curve $$\gamma(t)=\,\al(t)\cdot (I_g,0)=\,
\left(\,e^{2tX},\,Z\left(\int_0^te^{-sX}ds\right)\,e^{tX}\,\right)$$
is a geodesic in $\PR$ passing through the origin $(I_g,0).$ Since
$X$ is a symmetric real matrix, there is a diagonal matrix
$\Lambda=\,\text{diag}\,(\la_1,\cdots,\la_g)$ with
$\la_1,\cdots,\la_g\in\BR$
such that
$$X=\,^tk\Lambda k\ \ \ \ \text{for\ some}\ k\in O(g),$$ where
$\la_1,\cdots,\la_n$ are real numbers and not all zero. Thus we
may write $$\gamma(t)=\,\left(\,(\delta_{kl}e^{2\la_k t})[k],\,
Z\left(\int_0^t e^{(t-s)\Lambda}ds\right)[k]\,\right).$$ Hence
this completes the proof.
\end{proof}

\vskip 0.3cm
\begin{theorem}
Let $(Y_0,V_0)$ and $(Y_1,V_1)$ be two points in $\PR.$ Let $g$ be an element in $GL(g,\BR)$
such that $Y_0[\,{}^tg]=\,I_g$ and $Y_1[\,{}^tg]$ is diagonal.
Then the length $s\big((Y_0,V_0),(Y_1,V_1) \big)$ of the geodesic
joining $(Y_0,V_0)$ and $(Y_1,V_1)$ for the $GL_{g,h}$-invariant Riemannian metric
$ds_{g,h;A,B}^2$ is given by
\begin{equation}
s\big((Y_0,V_0),(Y_1,V_1)\big)=\,A\,\left\{ \sum_{j=1}^g (\ln t_j)^2 \right\}^{1/2}\,+\,B\,
\int_0^1 \left( \sum_{j=1}^g \Delta_j\,e^{-(\ln t_j)\,t}\right)^{1/2}\,dt,
\end{equation}
where $\Delta_j=\,\sum_{k=1}^h{\widetilde v}_{kj}^2 \ (1\leq j\leq g)$ with
$(V_1-V_0)\,{}^tg=\,({\widetilde v}_{kj})$ and
$t_1,\cdots,t_g$ denotes the zeros of $\det (t\,Y_0-Y_1)$.
\end{theorem}

\begin{proof}
Without loss of generality we may assume that $(Y_0,V_0)=\,(I_g,0)$ and
$(Y_1,V_1)=\,(T,{\widetilde V})$ with $T=\, \textrm{diag}(t_1,\cdots,t_n)$\ diagoanl
because the element $(g,-V_0)\in GL_{g,h}$ can be regarded as an isometry of $\PR$
for the Riemannian metric $ds_{g,h;A,B}^2$ (cf.\,Lemma 8.1).
Let$\gamma(t)=\,\big( \alpha(t),\beta(t)\big)$ with $0\leq t\leq 1$ be the geodesic in
$\PR$ joining two points $\gamma(0)=(Y_0,V_0)$ and $\gamma(1)=(Y_1,V_1)$, where
$\alpha(t)$ is the uniquely determined curve in $\CP$ and $\beta(t)$ is the uniquely
determined curve in $\BR^{(h,g)}.$
\vskip 0.2cm We now use the partial Iwasawa coordinates in $\CP$. Then if $Y\in\CP$, we write for any
positive integer $r$ with $0<r<g,\ r+s=g,$
$$Y=\,\begin{pmatrix} F & 0 \\ 0 & G \end{pmatrix}
\left[ \begin{pmatrix} I_r & 0 \\ H & I_s \end{pmatrix}\right],\qquad F\in {\mathcal P}_r,\
G\in {\mathcal P}_s,\ H\in \BR^{(h,g)}.$$
For $V\in \BR^{(h,g)}$, we write
$$V=\,(R,S),\qquad R\in \BR^{(h,r)},\ S\in \BR^{(h,s)}.$$

Now we express $ds_{g,h;A,B}^2$ in terms of $F,G,H,R$ and $S.$

\begin{lemma}
\begin{eqnarray*}
ds_{g,h;A,B}^2 &=&\, A\cdot\left\{ \sigma \big( (F^{-1}dF)^2\big)\,+\,
\sigma \big( (G^{-1}dG)^2\big)\,+\,2\,\sigma \big( F^{-1}
{}^t\!(dH)\,G\,dH\big)\right\}\\
& &\ +\,B\cdot\left\{ \sigma \big( F^{-1}\,{}^t\!(dR)\,dR\,\big)\,+\,
\sigma \big( (G^{-1}+ F^{-1}[\,{}^tH])\,{}^t(dS)\,dS \big)
\right\}\\
& &\ \,-\,
2\, B\cdot \sigma \big( F^{-1}\,{}^tH\,{}^t(dS)\,dR\big).
\end{eqnarray*}
\end{lemma}
\noindent {\it Proof of Lemma 8.3.} First we see that if $Y\in \CP$, then
$$Y^{-1}=\,\begin{pmatrix} F^{-1} & 0 \\ 0 & G^{-1} \end{pmatrix}
\left[ \begin{pmatrix} I_r & -{}^t\!H \\ 0 & I_s \end{pmatrix}\right]
=\,\begin{pmatrix} F^{-1} & -F^{-1}\,{}^tH \\
-HF^{-1} & G^{-1}+F^{-1}[\,{}^tH] \end{pmatrix} ,$$

$$ dY=\,\begin{pmatrix} dF+dG\, [H] +\,{}^t\!(dH)\cdot GH+\,{}^tHG\cdot dH &
{}^t\!(dH)\cdot G\,+\,{}^tH\cdot dG \\
dG\cdot H\,+\,G\cdot dH & dG \end{pmatrix}$$
and
$dV=\,(dR,dS).$
\vskip 0.2cm
For brevity, we put
\begin{equation*}
dY\cdot Y^{-1}=\, \begin{pmatrix} L_0 & L_1 \\ L_2 & L_3 \end{pmatrix}
\end{equation*}
and
\begin{equation*}
Y^{-1}\,{}^t(dV)\,dV=\, \begin{pmatrix} M_0 & M_1 \\ M_2 & M_3 \end{pmatrix}.
\end{equation*}
Here $L_0,\,L_1,\,L_2$ and $L_3$ denote the $r\times r$, $r\times s$,\ $s\times r$ and
$s\times s$ matrix valued differential one forms respectively, and
$M_0,\,M_1,\,M_2$ and $M_3$ denote the $r\times r$, $r\times s$,\ $s\times r$ and
$s\times s$ matrix valued differential two forms respectively.

\vskip 0.2cm\noindent By an easy computation, we get
\begin{eqnarray*}
L_0 &=&\,dF\cdot F^{-1}\,+\,{}^tHG\cdot dH\cdot F^{-1},\\
L_1 &=&\,-dF\cdot F^{-1}\,{}^tH-\,{}^tHG\cdot dH\cdot F^{-1}{}^tH\,+\,{}^t\!(dH)\,+\,
{}^tH\cdot dG\cdot G^{-1},\\
L_2 &=&\,G\cdot dH\cdot F^{-1},\\
L_3 &=&\, dG\cdot G^{-1}\,-\,G\cdot dH\cdot F^{-1}\,{}^tH,\\
M_0 &=&\,F^{-1}\,{}^t(dR)\,dR\,-\,F^{-1}\,{}^t\!H\,{}^t(dS)\,dR,\\
M_3 &=&\,-HF^{-1}\,{}^t(dR)\,dS\,+\,\big( G^{-1}+ F^{-1}[\,{}^tH] \big)\,{}^t(dS)\,dS.
\end{eqnarray*}

\noindent Therefore we have
\begin{eqnarray*}
ds_{g,h;A,B}^2 &=&\, A\cdot \sigma\big( (dY\cdot Y^{-1})^2 \big)\,+\,
B\cdot \sigma\big( Y^{-1}\,{}^t(dV)\,dV\,\big)\\
&=&\,A\cdot \left\{\, \sigma \big(L_0^2\,+\,L_1L_2\big)\,+\,\sigma
\big( L_2L_1\,+\,L_3^2\,\big)\right\}\\
& &\ +\,B\cdot \left\{\, \sigma (M_0)\,+\,\sigma (M_3)\,\right\}\\
&=&\,
A\cdot\left\{ \sigma \big( (F^{-1}dF)^2\big)\,+\,
\sigma \big( (G^{-1}dG)^2\big)\,+\,2\,\sigma \big( F^{-1}
{}^t\!(dH)\,G\,dH\big)\right\}\\
& &\ +\,B\cdot\left\{ \sigma \big( F^{-1}\,{}^t\!(dR)\,dR\,\big)\,+\,
\sigma \big( (G^{-1}+ F^{-1}[\,{}^tH])\,{}^t(dS)\,dS \big)
\right\}\\
& &\ \,-\,
2\, B\cdot \sigma \big( F^{-1}\,{}^tH\,{}^t(dS)\,dR\big).
\end{eqnarray*}
\hfill $\square$

\vskip 0.2cm
Let $s\big( (Y_0,V_0),(Y_1,V_1)\big)$ be the length of the geodesic
$\gamma(t)=\,(\alpha (t),\beta (t))$ with $0\leq t\leq 1.$
We put
\begin{equation*}
\alpha (t)=\,\begin{pmatrix} F(t) & 0 \\ 0 & G(t) \end{pmatrix}
\left[ \begin{pmatrix} I_r & 0 \\ H(t) & I_s \end{pmatrix}\right],\quad
\beta (t)=\,(R(t),S(t)),\quad 0\leq t\leq 1,
\end{equation*}
where $F(t),\,G(t),\,H(t),\,R(t)$ and $S(t)$ are the uniquely determined curves
in ${\mathcal P}_r,\,{\mathcal P}_s,\,\BR^{(s,r)},\,\BR^{(h,r)}$ and $\BR^{(h,s)}$
respectively.

\vskip 0.2cm then we have
\begin{eqnarray*}
& &\, s\big( (Y_0,V_0),(Y_1,V_1)\big)\\
&=&\, A\cdot \int_0^1 \left\{ \, \sigma\left( \left( F^{-1} {{dF}\over{dt}}\right)^2\right)
\,+\,\sigma\left( \left( G^{-1} {{dG}\over{dt}}\right)^2\right)\,+\,
2\,\sigma \left( F^{-1} {}^{{}^{{}^{{}^\text{\scriptsize $t$}}}}\!\!\!\left(
{{dH}\over {dt}}\right) G\,{{dH}\over {dt}}\,\right)\right\}^{1/2}\,dt\\
& &\ +\, B\cdot \int_0^1 \left\{\, \sigma \left( \gamma(t)^{-1}\,
{}^{{}^{{}^{{}^\text{\scriptsize $t$}}}}\!\!\!\left( {{dV}\over {dt}}\right)
{{dV}\over {dt}}\,\right)\right\}^{1/2}\,dt\\
&\geq& \,A\cdot \int_0^1
\left\{ \, \sigma\left( \left( F^{-1} {{dF}\over{dt}}\right)^2\right)
\,+\,\sigma\left( \left( G^{-1} {{dG}\over{dt}}\right)^2\right)\right\}^{1/2}\,dt\\
& &\ +\,B\cdot
\int_0^1 \left\{\, \sigma \left( F^{-1}\,
{}^{{}^{{}^{{}^\text{\scriptsize $t$}}}}\!\!\!\left( {{dR}\over {dt}}\right)
{{dR}\over {dt}}\,\right)\,+\,
\sigma \left( G^{-1}\,
{}^{{}^{{}^{{}^\text{\scriptsize $t$}}}}\!\!\!\left( {{dS}\over {dt}}\right)
{{dS}\over {dt}}\,\right)\,
\right\}^{1/2}\,dt.\\
\end{eqnarray*}
The reason is that the quadratic form $\sigma\big( F^{-1}\,{}^t\!(dH)G\,dH\big)$ is positive
definite. Indeed, if $M,N\in GL(g,\BR)$ such that $F=\,{}^tMM$ and $G=\,{}^tN N,$ then
$$ \sigma\big( F^{-1}\,{}^t\!(dH)G\,dH\big)=\,\sigma\big( \,{}^tWW),\quad W:=\,N\cdot dH\cdot M^{-1}.$$

\noindent If
$$ \sigma \left( F^{-1} {}^{{}^{{}^{{}^\text{\scriptsize $t$}}}}\!\!\!\left(
{{dH}\over {dt}}\right) G\,{{dH}\over {dt}}\,\right)\,=\,0,$$
then ${{dH}\over {dt}}=0$ and hence $H(t)$ is constant in the interval $[0,1].$
Since $H(0)=0,\ H(t)=0\ (0\leq t\leq 1).$
\vskip 0.2cm Moreover the curve $\alpha(t)$ must be diagonal, that is,
\begin{equation*}
\alpha(t)\,=\,\left( \delta_{\mu\nu}\,e^{\chi_\nu (t)}\right),\quad \chi_\nu(0)=0,\ \chi_\nu (1)
\,=\,\ln t_\nu,\quad 1\leq \nu\leq g,
\end{equation*}
where $g_\nu (t)\ (1\leq \nu\leq g)$ are continuously differentiable in [0,1]. Thus we have
$$ {{d\alpha}\over {dt}}\,=\,\left( \delta_{\mu\nu}\,e^{\chi_\nu (t)}
{{d\chi_\nu}\over {dt}}\right)$$
and hence
$$ \alpha(t)^{-1}\,{{d\alpha}\over {dt}}\,=\,\left( \delta_{\mu\nu}\,
{{d\chi_\nu}\over {dt}}\right).$$
Therefore we have
\begin{eqnarray*}
& &\,\int_0^1 \left\{ \,\sigma\left( \left( \alpha(t)^{-1}\,{{d\alpha}\over {dt}}\right)^2\right)
\right\}^{1/2}\,dt\\
&=&\,\int_0^1
\left\{ \, \sigma\left( \left( F^{-1} {{dF}\over{dt}}\right)^2\right)
\,+\,\sigma\left( \left( G^{-1} {{dG}\over{dt}}\right)^2\right)\right\}^{1/2}\,dt\\
&=&\, \int_0^1 \left\{\,\sum_{j=1}^n \left( {{d\chi_j}\over{dt}}\right)^2\right\}^{1/2}\,dt.
\end{eqnarray*}
The minimum value of $\sum_{j=1}^g \left( {{d\chi_j}\over{dt}}\right)^2$ is obtained if the curve
$\alpha (t)$ is the straight line, i.e., $\chi_j(t)=\,t\,\ln t_j\ (1\leq j\leq g),\ 0\leq t\leq 1$
in the $(\chi_1,\cdots,\chi_g)$-space. Thus we get
\begin{equation*}
\int_0^1 \left\{ \,\sigma\left( \left( \alpha(t)^{-1}\,{{d\alpha}\over {dt}}\right)^2\right)
\right\}^{1/2}\,dt\,=\,\left\{\,\sum_{j=1}^n \,(\ln t_j)^2\,\right\}^{1/2}.
\end{equation*}
We put
$$\beta (t)\,=\,\big(\, \beta_{kj}(t)\,\big)\qquad \textrm{with}\quad
0\leq t\leq 1,\ 1\leq k\leq h,\ 1\leq j\leq g.$$

Then we obtain
\begin{eqnarray*}
& &\,\int_0^1 \left\{ \,\sigma\left( \alpha (t)^{-1}
{}^{{}^{{}^{{}^\text{\scriptsize $t$}}}}\!\!\!\left(
{{d\beta}\over {dt}} \right)\,{{d\beta}\over {dt}}\right)\right\}^{1/2}\,dt\\
&=&\,\int_0^1 \left\{ \,\sigma\left( F^{-1}
{}^{{}^{{}^{{}^\text{\scriptsize $t$}}}}\!\!\!\left(
{{dR}\over {dt}} \right)\,{{dR}\over {dt}}\right)
\,+\,
\sigma\left( G^{-1}
{}^{{}^{{}^{{}^\text{\scriptsize $t$}}}}\!\!\!\left(
{{dS}\over {dt}} \right)\,{{dS}\over {dt}}\right)
\right\}^{1/2}\,dt\\
&=&\, \int_0^1 \left\{\,\sum_{j=1}^g\sum_{k=1}^h\,e^{-t\,\ln t_j}
\left( {{d\beta_{kj}}\over {dt}}\right)^2 \right\}^{1/2}\,dt.\\
\end{eqnarray*}
Each curve $\beta_{kj}(t)\ (0\leq t\leq 1)$ is a curve in $\BR$ such that
$\beta_kj(0)=0$ and $\beta_{kj}(1)=\,{\widetilde v}_{kj}.$ Thus each
curve $\beta_{kj}(t)$ must be a straight line, that is,
for all $k,j$ with $1\leq k\leq h$ and $1\leq j\leq g$,
$$\beta_{kj}(t)=\,{\widetilde v}_{kj}\,t,\qquad 0\leq t\leq 1.$$
Therefore we have
\begin{eqnarray*}
& &\,\int_0^1 \left\{ \,\sigma\left( \alpha (t)^{-1}
{}^{{}^{{}^{{}^\text{\scriptsize $t$}}}}\!\!\!\left(
{{d\beta}\over {dt}} \right)\,{{d\beta}\over {dt}}\right)\right\}^{1/2}\,dt\\
&=&\, \int_0^1 \left\{\,\sum_{j=1}^n\,e^{-t\,\ln t_j}
\left( \sum_{k=1}^h {\widetilde v}_{kj}^2 \right) \right\}^{1/2}\,dt\\
&=&\,\int_0^1 \left( \,\sum_{j=1}^g\,\Delta_j\,e^{-t\,\ln t_j}\right)^{1/2}
\,dt.
\end{eqnarray*}
Finally we obtain
\begin{equation}
s\big((Y_0,V_0),(Y_1,V_1)\big)=\,A\,\left\{ \sum_{j=1}^g (\ln t_j)^2 \right\}^{1/2}\,+\,B\,
\int_0^1 \left( \sum_{j=1}^g \Delta_j\,e^{-(\ln t_j)\,t}\right)^{1/2}\,dt.
\end{equation}
Hence we complete the proof.
\end{proof}

\vskip 0.5cm
For a fixed element $(A,a)\in \Gnm,$ we let $\Theta_{A,a}:\PR\lrt\PR$ be the
mapping defined by
$$\Theta_{A,a}(Y,V):=(A,a)\cdot (Y,V),\qquad (Y,V)\in\PR.$$
We consider the behaviour of the differential map $d\Theta_{A,a}$
of $\Theta_{A,a}$ at $(I_g,0)$. Then
$d\Theta_{A,a}$ is given by
$$d\Theta_{A,a}(u,v)=(Au\,^tA,\,v\,^tA),$$
where $(u,v)$ is a tangent vector of $\PR$ at $(I_g,0).$

\vskip 0.11cm
We let ${\tilde \theta}$ be the involution of $\Gnm$ defined by
$${\tilde \theta}((A,a)):=(\,^tA^{-1},-a),\qquad (A,a)\in \Gnm.$$
Then the differential map of ${\tilde \theta}$ at $(I_g,0)$,
denoted by the same notation ${\tilde \theta}$ is given by
$${\tilde \theta}:\fg\lrt\fg,\quad {\tilde
\theta}(X,Z)=(\,-\,^tX,-Z),$$ where $X\in \BR^{(g,g)}$ and $Z\in\Rmn.$
We note that $\fk$ is the (+1)-eigenspace of ${\tilde \theta}$ and
$$\fp=\left\{ (X,Z)\,|\ X\in\BR^{(g,g)},\ X=\,^tX,\ Z\in
\Rmn\,\right\}$$ is the (-1)-eigenspace of ${\tilde \theta}$.
\vskip 0.12cm
Now we consider some differential forms on $\PR$
which are invariant under the action of $GL(g,\BZ)\ltimes
\BZ^{(h,g)}$. We let
$${\mathfrak G}_{g,h}:=GL(g,\BZ)\ltimes \BZ^{(h,g)}$$
be the discrete subgroup of $\Gnm$. Let
$$\al_*=\sum_{\mu\leq
\nu}f_{\mu\nu}(Y,V)\,dy_{\mu\nu}+\sum_{k=1}^h\sum_{l=1}^g\phi_{kl}(Y,V)\,dv_{kl}$$
be a differential 1-form on $\PR$ that is invariant under the
action of ${\mathfrak G}_{g,h}$. We put
$$e_{\mu\nu}=\begin{cases} 1 & \text{if $\mu=\nu$}\\
{\frac 12} &\text{otherwise}.\end{cases}$$ We let
$$f(Y,V)=(e_{\mu\nu}f_{\mu\nu}(Y,V))\quad\text{and}\quad
\phi(Y,V)=\,^t(\phi_{kl}(Y,V)),$$ where $f(Y,V)$ is a $g\times g$
matrix with entries $f_{\mu\nu}(Y,V)$ and $\phi(Y,V)$ is a
$g\times h$ matrix with entries $\phi_{kl}(Y,V).$ Then
$$\al_*=\s(f\,dY\,+\,\phi\,dV).$$
If ${\tilde \gamma}=(\gamma,\al)\in {\mathfrak G}_{g,h}$ with $\gamma\in
GL(g,\BZ)$ and $\al\in \BZ^{(h,g)}$, then we have the following
transformation relation
\begin{equation}
f(\gamma
Y\,^t\gamma,\,(V+\al)\,^t\gamma)=\,^t\gamma^{-1}f(Y,V)\,\gamma^{-1}\end{equation}
and
\begin{equation}    \phi(\gamma
Y\,^t\gamma,\,(V+\al)\,^t\gamma)=\,^t\gamma^{-1}\phi(Y,V).\end{equation}
\indent We let $$\omega_0=dy_{11}\wedge dy_{12}\wedge\cdots\wedge
dy_{nn}\wedge dv_{11}\wedge \cdots\wedge dv_{hg}$$ be a
differential form on $\PR$ of degree ${\widetilde N}:={{g(g+1)}\over
2}+gh.$ If $\omega=h(Y,V)\omega_0$ is a differential form on $\PR$
of degree ${\tilde N}$ that is invariant under the action of
${\mathfrak G}_{g,h}.$ Then the function $h(Y,V)$ satisfies the
transformation relation
\begin{equation}
h(\gamma
Y\,^t\gamma,\,(V+\al)\,^t\gamma)=(\text{det}\,\gamma)^{-(g+h+1)}\,h(Y,V)\end{equation}
for all $\gamma\in GL(g,\BZ)$ and $\al\in \BZ^{(h,g)}.$
\vskip 0.12cm
We write
$$\omega_1=dy_{11}\wedge dy_{12}\wedge\cdots\wedge
dy_{gg}\quad\text{and}\quad \omega_2=dv_{11}\wedge \cdots\wedge
dv_{hg}.$$ Now we define $$\omega_{ab}=\epsilon_{ab}
\bigwedge_{1\le \mu\le \nu\le g\atop (\mu,\nu)\neq (a,b)}
dy_{\mu\nu}\wedge \omega_2,\quad 1\le a\le b\le g$$ and
$${\tilde\omega}_{cd}={\tilde\epsilon}_{cd}\,\omega_1\wedge
\bigwedge_{1\le k\le h,\, 1\le l\le g\atop (k,l)\neq (c,d)}
dv_{kl},\quad 1\le c\le h,\ 1\le d\le g.$$ Here the signs
$\epsilon_{ab}$ and ${\tilde\epsilon}_{cd}$ are determined by the
relations $\epsilon_{ab}\,\omega_{ab}\wedge dy_{ab}=\omega_0$ and
${\tilde\epsilon}_{cd}\,\omega_{cd}\wedge dv_{cd}=\omega_0$. We
let
$$\beta_*=\sum_{\mu\le\nu}s_{\mu\nu}(Y,V)\,\omega_{\mu\nu}+\sum_{k=1}^h\sum_{l=1}^g\varphi_{kl}(Y,V)\,{\tilde\omega}_{kl}$$
be a differential form on $\PR$ of degree ${\widetilde N}-1$ that is
invariant under the action of ${\mathfrak G}_{g,h}$, where
$s_{\mu\nu}(Y,V)$ and $\varphi_{kl}$ are smooth functions on
$\PR.$ We set
$$s=(\epsilon_{\mu\nu}s_{\mu\nu}),\
\epsilon_{\mu\nu}=\epsilon_{\nu\mu},\ s_{\mu\nu}=s_{\nu\mu}\quad
\text{and}\quad \varphi=({\tilde\epsilon}_{kl}\varphi_{kl}).$$ If
we write
$$\Omega(Y,V)=\begin{pmatrix} s(Y,V)\\ \varphi(Y,V)\end{pmatrix},$$ then we
obtain
$$\beta_* \wedge \begin{pmatrix} dY \\ dV\end{pmatrix}=\Omega\,\omega_0.$$
If ${\tilde \gamma}=(\gamma,\al)\in {\mathfrak G}_{g,h}$, then we have
the following transformation relations :
\begin{equation}
s(\gamma
Y\,^t\gamma,\,(V+\al)\,^t\gamma)=(\text{det}\,\gamma)^{-(g+h+1)}\,\gamma
\,s(Y,V)\,^t\gamma  \end{equation} and
\begin{equation}
\varphi(\gamma
Y\,^t\gamma,\,(V+\al)\,^t\gamma)=(\text{det}\,\gamma)^{-(g+h+1)}\,\varphi(Y,V)\,^t\gamma.
\end{equation}

${\mathfrak G}_{g,h}$ acts on $\CP\times \BR^{(h,g)}$ properly discontinuously.
The quotient space
\begin{equation}
{\mathfrak G}_{g,h}\backslash \big(\CP\times \BR^{(h,g)} \big)
\end{equation}
may be regarded as a family of principally polarized real tori of dimension $gh$. To each equivalence class
$[Y]\in {\mathfrak G}_g\ba \CP$ with $Y\in \CP$ we associate a principally polarized real torus
$T_Y^{[h]}\,=\,T_Y\times \cdots\times T_Y$ with $T=\BR^g/\La_Y$, where
$\La_Y=\,Y\BZ^g$ is a lattice in $\BR^g$.

\vskip 0.3cm Let $Y_1$ and $Y_2$ be two elements in $\CP$ with $[Y_1]\neq [Y_2]$, that is,
$Y_2\neq A\,Y_1\,{}^t\!A$ for all $A\in {\mathfrak G}_g.$ We put $\La_i\,=\,Y_i\,\BZ^g$ for $i=1,2.$
Then a torus $T_1\,=\,\BR^g/\La_1$ is diffeomorphic to $T_2\,=\,\BR^g/\La_2$ as smooth manifolds but $T_1$
is not isomorphic to $T_2$ as polarized tori.

\begin{lemma} The following set
\begin{equation}
{\mathfrak R}_{g,h}:=\,\left\{\, (Y,V)\,|\ Y\in \Rg,\ |v_{kj}|\leq 1,\ V=(v_{kj})\in\BR^{(h,g)}\,\right\}
\end{equation}
is a fundamental set for ${\mathfrak G}_{g,h}\ba {\mathcal P}_g\times \BR^{(h,g)}. $
\end{lemma}
\noindent {\it Proof.} It is easy to see that ${\mathfrak R}_{g,h}$
is a fundamental set for ${\mathfrak G}_{g,h}\ba {\mathcal P}_g\times \BR^{(h,g)}. $ We leave the detail to the reader.
\hfill $\square$

\vskip 0.3cm
For two positive integers
$g$ and $h$, we consider the Heisenberg group
$$H_{\BR}^{(g,h)}=\{\,(\la,\mu;\ka)\,|\ \la,\mu\in \BR^{(h,g)},\ \kappa\in\BR^{(h,h)},\
\ka+\mu\,^t\la\ \text{symmetric}\ \}$$ endowed with the following
multiplication law
$$(\la,\mu;\ka)\circ (\la',\mu';\ka')=(\la+\la',\mu+\mu';\ka+\ka'+\la\,^t\mu'-
\mu\,^t\la').$$
We define the semidirect product of $Sp(g,\BR)$
and $H_{\BR}^{(g,h)}$
$$G^J=Sp(g,\BR)\ltimes H_{\BR}^{(g,h)}$$
endowed with the following multiplication law
$$
(M,(\lambda,\mu;\kappa))\cdot(M',(\lambda',\mu';\kappa')) =\,
(MM',(\tilde{\lambda}+\lambda',\tilde{\mu}+ \mu';
\kappa+\kappa'+\tilde{\lambda}\,^t\!\mu'
-\tilde{\mu}\,^t\!\lambda'))$$
with $M,M'\in Sp(g,\BR),
(\lambda,\mu;\kappa),\,(\lambda',\mu';\kappa') \in
H_{\BR}^{(g,h)}$ and
$(\tilde{\lambda},\tilde{\mu})=(\lambda,\mu)M'$. Then $G^J$ acts
on the Siegel-Jacobi space $\BH_g\times \BC^{(h,g)}$ transitively by
\begin{equation}
(M,(\lambda,\mu;\kappa))\cdot (\Om,Z)=(M\cdot\Om,(Z+\lambda
\Om+\mu)
(C\Omega+D)^{-1}), \end{equation}
where $M=\begin{pmatrix} A&B\\
C&D\end{pmatrix} \in Sp(g,\BR),\ (\lambda,\mu; \kappa)\in
H_{\BR}^{(g,h)}$ and $(\Om,Z)\in \BH_g\times \BC^{(h,g)}.$ We note
that the Jacobi group $G^J$ is {\it not} a reductive Lie group and
also that the space ${\mathbb H}_g\times \BC^{(h,g)}$ is not a
symmetric space. We refer to \cite{Y2, Y3, Y4, Y5, Y6, Y7}
for more detail on the Siegel-Jacobi space $\BH_g\times \BC^{(h,g)}$.

\vskip 0.2cm We let $$\G_{g,h}:=\G_g\ltimes H_{\BZ}^{(g,h)}$$ be the
discrete subgroup of $G^J$, where
$$H_{\BZ}^{(g,h)}=\{\,(\la,\mu;\ka)\in H_{\BR}^{(g,h)}\,|\ \la,\mu\in \BZ^{(h,g)},\ \
\ka\in \BZ^{(h,h)}\, \}.$$
We define the map $\Phi_{g,h}:{\mathcal P}_g\times \BR^{(h,g)} \lrt \BH_g\times \BC^{(h,g)}$ by
\begin{equation}
\Phi_{g,h}(Y,\zeta):=\,(i\,Y,\zeta),\qquad (Y,\zeta)\in {\mathcal P}_g\times \BR^{(h,g)}.
\end{equation}
We have the following inclusions
$$ {\mathcal P}_g\times \BR^{(h,g)} \stackrel{\Phi_{g,h}}{\lrt}\,\,\HG\times \BC^{(h,g)}\,\,
\hookrightarrow\,\,\BH_g\times \BC^{(h,g)}\,\,\hookrightarrow\,\,\BH_g^*\times \BC^{(h,g)}.$$
${\mathfrak G}_{g,h}$ acts on ${\mathcal P}_g\times \BR^{(h,g)}$, $\G_g^{\star}\ltimes H_{\BZ}^{(g,h)}$
acts on $\HG\times \BC^{(h,g)}$ and $\G_{g,h}$ acts on $\BH_g\times \BC^{(h,g)}$ and $\BH_g^*\times \BC^{(h,g)}.$
It might be interesting to characterize the boundary points of the closure of the image of $\Phi_{g,h}$ in
$\BH_g^*\times \BC^{(h,g)}.$

\end{section}

\vskip 1cm

\begin{section}{{\large\bf Real Semi-Abelian Varieties}}
\setcounter{equation}{0}
\vskip 0.3cm
In this section we review the work of Silhol on semi-abelian varieties \cite{Si3} which is needed in the next section.
\begin{definition}
A complex  $\textsf{semi}$-$\textsf{abelian variety}$ $A$ is the extension of an abelian variety ${\widetilde A}$ by a group of multiplicative type.
A semi-abelian variety is said to be $\textsf{real}$ if it admits an anti-holomorphic involution which is a group homomorphism.
\end{definition}

Let $T$ be a group of multiplicative type. We consider the exponential map $\exp: {\mathfrak t}\lrt T$. The real structure $S$
on $T$ lifts to a real structure $S_{\mathfrak t}$ on ${\mathfrak t}$. Then $L_{\mathfrak t}:=\,\ker\exp$ is a free $\BZ$-module and
$S_{\mathfrak t}$ induces an involution on $L_{\mathfrak t}$. By standard results (cf.\,\cite[I.\,(3.5.1)]{Si2}),
we can find a basis of $L_{\mathfrak t}$ with respect to which the matrix for $S_{\mathfrak t}$ is of the form
\begin{equation*}
\begin{pmatrix} I_s & 0 &  0 & \cdots & 0 & 0\\
0 & B & 0 & \cdots & 0 & 0 \\
0 & 0 & B & \cdots & 0 & 0 \\
0 & 0 & 0 & \ddots & 0 & 0 \\
0 & 0 & 0 & \cdots & B & 0\\
0 & 0 & 0 & \cdots & 0 & -I_t
\end{pmatrix},\qquad
B:=\begin{pmatrix} 0 & 1 \\ 1 & 0 \end{pmatrix}.
\end{equation*}

Since fixing a basis of $L_{\mathfrak t}$ is equivalent to fixing an isomorphism $T\cong (\BC^*)^r$, we get
\begin{equation*}
T=\,T_1\times T_2\times T_3,\qquad r=s' +2p+t',
\end{equation*}
where
\vskip 0.2cm\noindent
(i) $T_1=\,\BC^*\times\cdots\times \BC^*$ ($s'$-times) and $S$ induces on each factor the involution
$z\longmapsto {\overline z}.$ In this case we write $T_1=\,{\mathbb G}_m^0\times\cdots\times
{\mathbb G}_m^0$\,;
\vskip 0.2cm\noindent
(ii) $T_2=\,\BC^*\times\cdots\times \BC^*$ ($t'$-times) and $S$ induces on each factor the involution
$z\longmapsto {\overline z}^{-1}.$ In this case we write $T_1=\,{\mathbb G}_m^\infty\times\cdots\times
{\mathbb G}_m^\infty$\,;
\vskip 0.2cm\noindent
(iii) $T_3=\,(\BC^*\times\BC^*)\times\cdots\times (\BC^*\times\BC^*)$ ($p$-times)
and $S$ induces on each factor $(\BC^*\times\BC^*)$ the involution
$(z_1,z_2)\longmapsto ({\overline z}_2, {\overline z}_1).$ In this case we write
$T_3=\,{\mathbb G}_m^{2*}\times\cdots\times {\mathbb G}_m^{2*}$.

\vskip 0.3cm
Let $\D=\,\{ \zeta\in\BC\,|\ |\zeta|<1\,\}$ be the unit disk and
let $\D^*=\,\{ \zeta\in\BC\,|\ 0<|\zeta|<1\,\}$ be a punctured unit disk. Let $\varphi:Z^*\lrt \D^*$
be a holomorphic family of matrices $\varphi^{-1}(\zeta)=Z(\zeta)$ in $\BH_g$. We have the natural action of
the lattice $\BZ^{2g}$ on $\D^*\times \BC^g$ defined by
\begin{equation}
(\lambda,\mu)\cdot (\zeta,z):=\,(\zeta,z+\lambda+\,Z(\zeta)\,\mu),\qquad
\zeta\in \D^*,\ \lambda,\mu\in\BZ^g,\ z\in\BC^g.
\end{equation}
Then the quotient space
\begin{equation}
\textbf{A}^*:=\,\big( \D^*\times \BC^g\big)/\BZ^{2g}
\end{equation}
is a holomorphic family of principally polarized abelian varieties associated to a holomorphic family
$\varphi:Z^*\lrt \D^*$.

\vskip 0.2cm
Now we write
\begin{equation*}
Z(\zeta)=\,X(\zeta)+\,i\,Y(\zeta)\in \BH_g
\end{equation*}
and
\begin{equation*}
Y(\zeta) =\,{}^tW(\zeta) D(\zeta) W(\zeta)\in {\mathcal P}_g\quad (the\ Jacobi\ decomposition)
\end{equation*}
with ${\rm diag}( d_1(\zeta),\cdots,d_g(\zeta))\in \BR^{(g,g)}$ is a diagonal matrix.
\vskip 0.2cm
Now we assume the following conditions (F1)--(F3)\,: for any $\zeta\in \D_r^*:=\{\,\zeta\in\BC\,
|\ 0< |\zeta|< r\,\}$,

\vskip 0.2cm \noindent
(F1) There exists a positive number $r>0$ such that for any $\zeta\in \D_r^*$,
$Z(\zeta)\in {\mathfrak W}_g(u)$ for some $u>0,$ where
$\D_r^*:=\{\,\zeta\in\BC\,
|\ 0< |\zeta|< r\,\}$\,;
\vskip 0.2cm \noindent
(F2) $X(\zeta)$ converges in $\BR^{(g,g)}$ and $W(\zeta)$ converges in $GL(g,\BR)$ as
$\zeta \rightarrow 0$\,;

\vskip 0.2cm \noindent
(F3) $\lim_{\zeta\rightarrow 0} d_i(\zeta)=\,d_i$ converges for $1\leq i\leq g-t$, and
$\lim_{\zeta \rightarrow 0} d_i(\zeta)=\infty$ for $g-t< i\leq g.$

\vskip 0.3cm
Let
\begin{equation*}
Z(0)=\,\begin{pmatrix}
z_{11} & \cdots &z_{1,g-t} & 0 & \cdots & 0 \\
\vdots & \ddots & \vdots & 0 & \ddots & \vdots \\
z_{g-t,1} & \cdots & z_{g-t,g-t}& 0 & \cdots & 0 \\
\vdots & \ddots & \vdots & \vdots & \cdots & \vdots \\
z_{g,1} & \cdots & z_{g,g-t}& 0 & \cdots & 0
\end{pmatrix},\qquad z_{ij}=\,\lim_{\zeta \rightarrow 0} z_{ij}(\zeta).
\end{equation*}
The action (9.1) extends to the action of $\BZ^{2g}$ on $\D\times \BC^g$ by letting
$Z(0)$ be the fibre at $\zeta=0.$ We take the quotient space
\begin{equation}
\textbf{A}:=\,\big( \D\times \BC^g\big)/\BZ^{2g}.
\end{equation}
Then we see that $\textbf{A}$ is an analytic variety fibred holomorphically over $\D$, and the fibre at $0$
is a semi-abelian variety
\begin{equation}
\textbf{A}_0=\,\BC^g/L_0,\qquad L_0:=\,\BZ^g Z(0)+\BZ^g\subset \BC^g
\end{equation}
of the abelian variety
\begin{equation}
\widetilde{\textbf{A}}_0:=\,\BC^{g-t}/L^{\diamond},\qquad
L^{\diamond}:=\BZ^{g-t}Z^{\diamond}(0)+\BZ^{g-t}\subset \BC^{g-t}
\end{equation}
by $(\BC^*)^t,$ where
\begin{equation*}
Z^{\diamond}(0)=\,\begin{pmatrix}
z_{11} & \cdots &z_{1,g-t} \\
\vdots & \ddots & \vdots  \\
z_{g-t,1} & \cdots & z_{g-t,g-t}
\end{pmatrix}\in \BH_{g-t}.
\end{equation*}
The extension
\begin{equation*}
1\lrt (\BC^*)^t\lrt \textbf{A}_0\lrt \widetilde{\textbf{A}}_0
\lrt 0
\end{equation*}
is defined by the image of
$$z_{g-k}^{\diamond}=(z_{g-k,1},\cdots,z_{g-k,g-t})\in\BC^{g-t},\qquad k=0,\cdots,t-1$$
under the maps
\begin{equation*}
\BC^{g-t}\lrt \widetilde{\textbf{A}}_0 \lrt {\rm Pic}^0( \widetilde{\textbf{A}}_0),
\end{equation*}
where the last map is the isomorphism defined by the polarization.

\vskip 0.2cm These above facts can be generalized as follows.

\begin{proposition}
Let $\varphi:Z^*\lrt \D^*$
be a holomorphic family of matrices $\varphi^{-1}(\zeta)=Z(\zeta)$ in $\BH_g$ such that
$\varphi^{-1}(\zeta)=Z(\zeta)$ converges in $\BH_g^*$ as $\zeta \rightarrow 0$. Then there exists an
analytic variety ${\textbf{A}}(Z^*)\lrt \D$ such that
\vskip 0.2cm\noindent
(i) the fibre at $\zeta(\neq 0)\in\D$ is the principally polarized abelian variety
$\BC^g/L_\zeta$ with the lattice $L_\zeta=\,\BZ^g Z(\zeta)+\BZ^g$\,;
\vskip 0.2cm\noindent
(ii) the zero fibre ${\textbf{A}}(Z^*)_0$ is a semi-abelian variety.
\end{proposition}
\vskip 0.2cm \noindent
{\it Proof.} The proof can be found in \cite[p.\,189]{Si3}.\hfill $\square$

\vskip 0.3cm
\begin{theorem}
Let $\varphi:Z^*\lrt \D^*$
be a holomorphic family of matrices $\varphi^{-1}(\zeta)=Z(\zeta)$ in $\BH_g$ such that
$\varphi^{-1}(\zeta)=Z(\zeta)$ converges in $\BH_g^*$ as $\zeta \rightarrow 0$.
We assume that $Z(\zeta)=\,\varphi^{-1}(\zeta)\in {\mathscr H}_g$ for $\zeta\in \BR\cap \D^*.$
Let $(s,t)$ be such that
$$\lim_{\zeta \rightarrow 0}Z(\zeta)\in \g\,B_M \big( {\mathscr F}_{s,t}\cap
\overline{{\mathscr H}_{\bf 0}}\big)$$
for some $M\in\BZ^{(g,g)}$ and some $\g\in \G_g^{\star}.$
Then
\vskip 0.2cm\noindent
(a) ${\textbf{A}}(Z^*)_0$ has a natural real structure extending the real structures of the
${{\textbf{A}}(Z^*)_\zeta}'s$ for $\zeta\in\BR\cap \D^*$\,;

\vskip 0.2cm\noindent
(b) As a real variety, ${\textbf{A}}(Z^*)_0$ is the extension of a real abelian variety
$\widetilde{\textbf{A}}(Z^*)_0$ by
$$\big( {\mathbb G}_m^0\big)^{s'}\times
\big( {\mathbb G}_m^{2*}\big)^{p}\times
\big( {\mathbb G}_m^{\infty}\big)^{t'},\qquad s=s'+p,\ t=t'+p\,;$$

\vskip 0.2cm\noindent
(c) Let $x\in \XG (s,t)\subset \overline{\XG}$ be the image of $\lim_{\zeta \rightarrow 0} Z(\zeta)$
in $\overline{\XG}$ and let $[x]$ be the image of $x$ under the isomorphism
$\XG (s,t)\,\cong\,{\mathscr X}_\BR^{g-r}$ with $r=s+t.$ Then $[x]$ is the real isomorphism class of
$\widetilde{\textbf{A}}(Z^*)_0$.
\end{theorem}

\vskip 0.2cm \noindent
{\it Proof.} The proof can be found in \cite[pp.\,191--192]{Si3}.\hfill $\square$

\begin{corollary}
Let $\varphi:Z^*\lrt \D^*$ be as in Theorem 9.1. Assume
$$\lim_{\zeta\lrt 0} Z(\zeta)\in {\mathscr F}_{0,t}\ \, (\,resp.\ {\mathscr F}_{s,0}).$$
Then the class of the extension
$$ 0\lrt \big( {\mathbb G}_m^\infty\big)^t \rightarrow {\textbf{A}}(Z^*)_0
\lrt \widetilde{\textbf{A}}(Z^*)_0\lrt 0$$
$$ \big(\, resp.\ \ 0\lrt \big( {\mathbb G}_m^0\big)^s \lrt {\textbf{A}}(Z^*)_0
\lrt \widetilde{\textbf{A}}(Z^*)_0\lrt 0\ \big)$$
is defined by $t$ purely imaginary divisors on ${\textbf{A}}(Z^*)_0$
\big(\,resp. $s$ real divisors ${\textbf{A}}(Z^*)_0$\,\big).
\end{corollary}

\end{section}

\vskip 1cm

\begin{section}{{\large\bf Real Semi-Tori}}
\setcounter{equation}{0}
\vskip 0.3cm
A real semi-torus $T$ of dimension $g$ is defined to be an extension of a real torus ${\widetilde T}$ of dimension $g-t$ by
a real group $(\BR^*)^t$ of multiplicative type, where $\BR^*=\BR-\{ 0\}.$

\vskip 0.2cm
Let $I=\,\{\,\xi\in\BR\,|\ -1< \xi < 1\,\}$ be the unit interval and $I^*=I-\{ 0\}$ be the punctured unit interval.
Let $\varpi:{\mathfrak Y}^*\lrt I^*$ be a real analytic family of matrices $\varpi^{-1}(\xi)=Y(\xi)\in {\mathcal P}_g.$
We have the natural action of the lattice $\BZ^g$ in $\BR^g$ on $I^*\times \BR^g$ defined by
\begin{equation}
\alpha\cdot (\xi,x)=(\xi, x+Y(\xi)\alpha),\quad \alpha\in\BZ^g,\ \xi\in I^*,\ x\in \BR^g.
\end{equation}
The quotient space
\begin{equation}
{\textbf T}^*:=\,(I^*\times\BR^g)/\BZ^g
\end{equation}
is a real analytic family of real tori of dimension $g$ associated to a real analytic family $\varpi:{\mathfrak Y}^*\lrt I^*$.
We let
\begin{equation*}
Y(\xi)=\,{}^tW(\xi)D(\xi)W(\xi)
\end{equation*}
be the Jacobi decomposition of $Y(\xi)$, where $D(\xi)={\rm diag}(d_1(\xi),\cdots,d_g(\xi))$ is a real diagonal matrix and
$W(\xi)$ is a strictly upper triangular real matrix of degree $g$. Now we assume the following conditions (T1)-(T4):
\vskip 0.3cm\noindent
(T1) There exists a positive number $r$ with $0<r,1$ such that for any $\xi\in I_r^*, \ i\,Y(\xi)\in {\mathfrak W}_g(u)$
for some $u>0,$ where $I_r^*:=\{\,\xi\in\BR\,
|\,-r<\xi< r\,\}$\,;
\vskip 0.2cm \noindent
(T2) $W(\xi)$ converges in $GL(g,\BR)$ as
$\xi \rightarrow 0$\,;

\vskip 0.2cm \noindent
(T3) $\lim_{\xi\rightarrow 0} d_i(\xi)=\,d_i$ converges for $1\leq i\leq g-t$, and
$\lim_{\xi \rightarrow 0} d_i(\xi)=\infty$ for $g-t< i\leq g.$

\vskip 0.3cm
Let
\begin{equation*}
Y(0)=\,\begin{pmatrix}
y_{11} & \cdots &y_{1,g-t} & 0 & \cdots & 0 \\
\vdots & \ddots & \vdots & 0 & \ddots & \vdots \\
y_{g-t,1} & \cdots & y_{g-t,g-t}& 0 & \cdots & 0 \\
\vdots & \ddots & \vdots & \vdots & \cdots & \vdots \\
y_{g,1} & \cdots & y_{g,g-t}& 0 & \cdots & 0
\end{pmatrix},\qquad y_{ij}=\,\lim_{\xi \rightarrow 0} y_{ij}(\zeta).
\end{equation*}
The action (10.1) extends to the action of $\BZ^{g}$ on $I\times \BR^g$ by letting
$Y(0)$ be the fibre at $\xi=0.$ We take the quotient space
\begin{equation}
\textbf{T}:=\,\big( I\times \BR^g\big)/\BZ^{g}.
\end{equation}
Then we see that $\textbf{T}$ is a real analytic variety fibred real analytically over $I$, and the fibre at $0$
is a real semi-torus
\begin{equation}
\textbf{T}_0=\,\BR^g/\Lambda_0,\qquad \Lambda_0:=\,\BZ^g Y(0)\subset \BR^g
\end{equation}
of the real torus
$$ \widetilde{\textbf{T}}_0:=\,\BR^{g-t}/\Lambda^{\diamond},\qquad
\Lambda^{\diamond}:=\BZ^{g-t}Y^{\diamond}(0)\ {\rm is\ a\ lattice\ in}\ \BR^{g-t}$$
by $(\BC^*)^t,$ where
\begin{equation*}
Y^{\diamond}(0)=\,\begin{pmatrix}
y_{11} & \cdots &y_{1,g-t} \\
\vdots & \ddots & \vdots  \\
y_{g-t,1} & \cdots & y_{g-t,g-t}
\end{pmatrix}\in {\mathcal P}_{g-t}.
\end{equation*}

\end{section}

\vskip 1cm

\begin{section}{{\large\bf Open Problems and Remarks}}
\setcounter{equation}{0}
\vskip 0.3cm
In this final section we give some open problems related to polarized real tori to be studied in the future.
\vskip 0.5cm\noindent
{\bf Problem 1.} Characterize the boundary points of the closure of $i\, {\mathcal P}_g$ in ${\mathbb H}_g^*$
explicitly.
\vskip 0.5cm\noindent
{\bf Problem 2.} Find the explicit generators of the ring ${\mathbb D}(g,h)$ of differential operators on
the Minkowski-Euclidean space ${\mathcal P}_g\times \BR^{(h,g)}$ which are invariant under the action (8.7) of
$GL_{g,h}=GL(g,\BR)\ltimes \BR^{(h,g)}$.

\vskip 0.5cm\noindent
{\bf Problem 3.} Find all the relations among a complete explicit list of generators of ${\mathbb D}(g,h)$.

\vskip 0.5cm The orthogonal group $O(g)$ of degree $g$ acts on the subspace
\begin{equation*}
{\mathfrak p}=\,\big\{ \, (X,Z)\,|\ X=\,{}^tX\in \BR^{(g,g)},\ Z\in\BR^{(h,g)}\,\big\}
\end{equation*}
of the vector space $\BR^{(g,g)}\times\BR^{(h,g)}$ by
\begin{equation}
k\cdot (X,Z)=\,(kX\,{}^tk,Z\,{}^tk),\quad k\in O(g),\ (X,V)\in {\mathfrak p}.
\end{equation}
The action (11.1) induces the action of $O(g)$ on the polynomial ring $\textrm{Pol}({\mathfrak p})$ on
${\mathfrak p}$. We denote by $I({\mathfrak p})$ the subring of $\textrm{Pol}({\mathfrak p})$
consisting of polynomials on ${\mathfrak p}$ invariant under
the action of $O(g)$. We see that there is a canonical linear bijection
\begin{equation*}
\Theta : I({\mathfrak p})\lrt {\mathbb D}(g,h)
\end{equation*}
of $I({\mathfrak p})$ onto ${\mathbb D}(g,h)$. We refer to \cite{Hel} and \cite{Y1} for more detail.

\vskip 0.5cm
\begin{remark} M. Itoh \cite{It} proved that $I({\mathfrak p})$ is generated by $\alpha_j\ (1 \leq j \leq g)$ and $\beta_{pq}^{(k)}$
($0\leq k\leq g-1,\ 1\leq p\leq q\leq h),$ where
\begin{equation}
\alpha_j (X,Z)= \textrm{tr}   \big( X^j\,\big),\quad 1\leq j\leq g
\end{equation}
and
\begin{equation}
\beta_{pq}^{(k)}(X,Z)=  \big(Z\,X^k\,^tZ\,\big)_{pq},\quad 0\leq k\leq g-1,\ 1\leq
p\leq q\leq h.
\end{equation}
Here $A_{pq}$ denotes the $(p,q)$-entry of a matrix $A$ of degree $h$.
\end{remark}

\vskip 0.5cm\noindent
\begin{remark} M. Itoh \cite{It} found all the relations among the above generators
$\alpha_j\ (1 \leq j \leq g)$ and $\beta_{pq}^{(k)}$
($0\leq k\leq g-1,\ 1\leq p\leq q\leq h)$ of $I({\mathfrak p})$.
\end{remark}

\vskip 0.5cm\noindent
{\bf Problem 4.} Develop the theory of harmonic analysis on the Minkowski-Euclidean space ${\mathcal P}_g\times \BR^{(h,g)}$
with respect to a discrete subgroup of $GL(g,\BZ)\ltimes \BZ^{(h,g)}.$

\vskip 0.5cm\noindent
{\bf Problem 5.} Characterize the boundary points of the closure of the image of $\Phi_{g,h}$ in
${\mathbb H}_g^*\times \BC^{(h,g)}$ (cf. see (8.29)).

\vskip 0.5cm\noindent
{\bf Problem 6.} Find the explicit generators of the ring $\BD \big(\BH_g\times \BC^{(h,g)}\big)$ of differential operators on
the Siegel-Jacobi space ${\mathbb H}_g\times \BC^{(h,g)}$ which are invariant under the action (8.28) of
the Jacobi group $G^J=Sp(g,\BR)\ltimes H_{\BR}^{(g,h)}$. We refer to \cite{Y2} for more detail.

\vskip 0.5cm\noindent
{\bf Problem 7.} Find all the relations among a complete list of generators of $\BD \big(\BH_g\times \BC^{(h,g)}\big)$.

\vskip 0.5cm\noindent
{\bf Problem 8.}  Develop the theory of harmonic analysis on the Siegel-Jacobi space ${\mathbb H}_g\times \BC^{(h,g)}$
with respect to a congruent subgroup of $\G_{g,h}=Sp(g,\BZ)\ltimes H_{\BZ}^{(g,h)}$. We refer to \cite{Y3} for more detail.

\end{section}

\vskip 1cm

\begin{center}{{\large\bf Appendix\,: Non-Abelian Cohomology}}
\setcounter{equation}{0}
\end{center}

\vskip 0.3cm In this section we review some results on the first cohomology set $H^1(<\!\tau\!>,\G)$ obtained by
Goresky and Tai \cite{GT1}, where $<\!\tau\!>=\{ 1,\tau\}$ is a group of order $2$ and $\g$ is a certain arithmetic subgroup.
These results are often used in this article.

\vskip 0.2cm First of all we recall the basic definitions. Let $S$ be a group. A group $M$ is called a $S$-${\mathsf group}$
if there exists an action of $G$ on $M,\ S\times M\lrt M,\ (\s,a)\longmapsto \s (a)$ such that $\s (ab)=\,\s(a)\,\s (b)$
for all $\s\in S$ and $a,b\in M$. From now on we let $1_S$ (resp. $1_M$) be the identity element of $S$ (resp. $M$).
We observe that if $M$ is a $S$-group, then $\s (1_M)=\,1_M$ for all $\s\in S.$

\vskip 0.3cm
\noindent {\bf Definition.}
{\it Let $M$ be a $S$-group, where $S$ is a group. We define
$$H^0(S,M):=\,\left\{\, a\in M\,|\ \s (a)=a \quad {\rm for\ all}\ \s\in S\,\right\}.$$
A map $f:S\lrt M$ is called a $1$-$\textsf{cocycle}$ with values in $M$ if $f(\s\tau)=\,f(\s)\,\s(f(\tau))$ for all
$\s,\tau\in S.$ We observe that if $f$ is a $1$-cocycle, then $f(1_S)=1_M.$
We denote by $Z^1(S,M)$ the set of all $1$-cocycles of $S$ with values in $M$. Let $f_1$ and $f_2$ be two $1$-cocycles
in $Z^1(S,M).$ We say that $f_1$ is $\textsf{cohomologous}$ to $f_2$, denoted $f_1\sim f_2$, if there exists an element
$h\in M$ such that
$$ f_2 (\s)=\,h^{-1} f_1(\s)\,\s(h)\qquad {\rm for\ all} \ \s\in S.$$
Let $f_\flat:S\lrt M$ be the trivial map, i.e., $f_\flat(\s)=\,1_M$ for all $\s\in S.$ A map $f:S\lrt M$ is called a
$1$-$\textsf{coboundary}$ if $f\sim f_\flat$, i.e., if there exists $h\in M$ such that $f(\s)=\,h^{-1}\s(h)$
for all $\s\in S.$}

\vskip 0.3cm
Obviously a 1-coboundary is a $1$-cocycle. It is easy to see that $\sim$ is an equivalence relation on $Z^1(S,M).$
So we define the first cohomology set
$$H^1(S,M):=\,Z^1(S,M)/\sim.$$

\noindent
{\bf Remark.} In general, $H^1(S,M)$ does not admit a group structure. But $H^1(S,M)$ has an identity, that is, the
cohomologous class containing the trivial $1$-cocycle $f_\flat.$

\vskip 0.2cm
\noindent {\bf Example.} Let $L$ be a Galois extension of a number field $K$ with Galois group $G$. A linear algebraic
group defined over $K$ has naturally the structure of $G$-group. It is known that $H^1(G,GL(n,L))$ is trivial for all
$n\geq 1.$ Using the following exact sequence of $G$-groups
$$1\lrt SL(n,L)\lrt GL(n,L)\lrt L^*\lrt 1,\qquad L^*=\,L-\{ 0\},$$
we can show that $H^1(G,SL(n,L))$ is trivial.

\vskip 0.53cm We put $G=Sp(g,\BR)$ and $K=U(g).$ Then ${\mathbf D}=G/K$ is biholomorphic to $\BH_g$.
Let $S_\tau=\,\{ 1,\tau\}$ be a group of order $2$ as before. We define the $S_\tau$-group structure on $G$ via the action
(2.7) of $S_\tau$ on $G$. Let $\G$ be an arithmetic subgroup of $Sp(g,\BQ).$ We let
$$X_\G:=\,\G\ba G/K\,\cong \G\ba \BH_g$$
and let $\pi_\G:{\mathbf D}\lrt X_\G$ be the natural projection. For any $\g\in \G,$ we define the map $f:S_\tau\lrt \G$ by
$$ f_\g(1)=1_\G\qquad {\rm and}\qquad f_\g(\tau)=\,\g,\leqno (1)$$
where $1_\G$ denotes the identity element of $\G.$

\vskip 0.5cm
\noindent{\bf Lemma 1.} {\it Let $\g\in\G.$ Then
\vskip 0.2cm\noindent
(a) $f_\g$ is a $1$-cocycle if and only if $\g\,\tau(\g)=\,1_\G,$ equivalently, $\tau(\g)\,\g=\,1_\G.$
\vskip 0.2cm\noindent
(b) A cocycle $f_\g$ is a $1$-coboundary if and only if there exists $h\in\G$ such that $\g=\,\tau(h)h^{-1}.$}

\vskip 0.2cm \noindent
{\it Proof.} The proof follows immediately form the definition.\hfill $\square$

\vskip 0.3cm
To each such a $1$-cocycle $f_\g$ we associate the $\g$-twisted involution $\tau\g:{\mathbf D}\lrt {\mathbf D}$
and $\tau\g:\G\lrt \G$. Indeed the involution $\tau\g:{\mathbf D}\lrt {\mathbf D}$ is defined by
$$ \tau\g (xK)=\,\tau(\g xK)=\,\tau(\g)\tau(x)K,\qquad x\in G \leqno (2)$$
and the involution $\tau\g:\G\lrt \G$ is defined by
$$ \tau\g (\g_1)=\,\tau(\g \g_1\g^{-1}),\ \quad \g_1\in \G .\leqno (3)$$
Let
$$ {\mathbf D}^{\tau\g}:=\,\left\{ x\in {\mathbf D}\,|\ (\tau\g)(x)=x\ \right\}$$
and
$$\G^{\tau\g}:=\,\left\{ \g_1\in \G\,|\ (\tau\g)(\g_1)=\g_1\ \right\}$$
\noindent
be the fixed point sets.

\vskip 0.5cm
\noindent{\bf Lemma 2.} {\it Let $x\in {\mathbf D}$. Then $\pi_\G(x)\in X_\G^\tau$ if and only if there exists
an element $\g\in\G$ such that $x\in {\mathbf D}^{\tau\g}$.}

\vskip 0.2cm \noindent
{\it Proof.} It is easy to prove this lemma. We leave the proof to the reader. \hfill $\square$

\vskip 0.5cm
\noindent{\bf Theorem A.} {\it Assume $\G$ is torsion free. Let ${\mathscr C}_\G$ be the set of all connected
components of the fixed point set $X_\G^\tau$. Then the map $\Phi_\G:H^1(S_\tau,\G)\lrt {\mathscr C}_\G$
defined by
$$\Phi_\G ([f_\g]):=\,\pi_\G \big({\mathbf D}^{\tau\g}\big)=\, \G^{\tau\g}\ba {\mathbf D}^{\tau\g}$$
determines a one-to-one correspondence between $H^1(S_\tau,\G)$ and ${\mathscr C}_\G$.}

\vskip 0.2cm \noindent
{\it Proof.} The proof can be found in \cite[pp.\,3-4]{GT1}.\hfill $\square$

\vskip 0.3cm
\noindent{\bf Theorem B.}
{\it Let $S_\tau=\{ 1,\tau\}$ be a group of order $2$. Then $Sp(g,\BR)$ has a $S_\tau$-group structure via the action
(2.7) and hence $U(g)$ also admits a $S_\tau$-group structure through the restriction of the action (2.7) to $U(g).$
And $H^1( S_\tau,U(g))$ and  $H^1( S_\tau,Sp(g,\BR))$ are trivial.}
\vskip 0.2cm \noindent
{\it Proof.} The proof can be found in \cite[pp.\,8-9]{GT1}. However I will give a sketchy proof for the reader.
Assume $f_k$ is a $1$-cocycle in $Z^1(S_\tau,U(g))$ with
$k=\begin{pmatrix} A & B \\ -B & A \end{pmatrix}\in U(g).$ Using the fact $\tau(k)k=I_g$, we see that
$$A=\,{}^tA,\quad B=\,{}^tB,\quad AB=BA \quad {\rm and}\quad A^2+B^2=I_g.$$
Therefore we can find $h\in O(g)$ such that $h(A+iB)h^{-1}={\mathfrak D}\in U(g)$ is diagonal.
We take $\mu=\,\sqrt{\overline{\mathfrak D}}\in U(g)$ by choosing a square root of each diagonal entry.
We set $\delta=\,h^{-1}\mu h.$ Then $k=\tau(\delta)\,\delta^{-1}$. By Lemma 1, $f_k$ is a $1$-coboundary.
Hence $H^1( S_\tau,U(g))$ is trivial.
\vskip 0.2cm
Let $G=Sp(g,\BR)$ as before. Suppose $f_M\in Z^1( S_\tau,G)$ with $M=\begin{pmatrix} A & B \\ C & D \end{pmatrix}\in G.$
Then we see that $f_M$ is a $1$-coboundary with values in $G$ if and only if $\BH_g^{M\tau}\neq \emptyset.$
We can find $M_1\in G$ such that $\BH_g^{M_1\tau}\neq \emptyset$ and $f_M\sim f_{M_1}\sim f_\flat.$
Therefore $f_M\sim f_\flat$, that is, $f_M$ is a $1$-coboundary with values in $G$. Hence
$H^1(S_\tau,G)$ is trivial.   \hfill $\square$

\vskip 0.5cm
\noindent{\bf Theorem C.} {\it For all $m\geq 1$, the mapping $H^1(S_\tau,\G_g(4m))\lrt
H^1(S_\tau,\G_g(2,2m))$ is trivial.}
\vskip 0.2cm \noindent
{\it Proof.} The proof can be found in \cite[pp.\,7-10]{GT1}. We will give a sketchy proof for the reader.
In order to prove this theorem, we need the following lemma.

\vskip 0.2cm
\noindent{\bf Lemma 3.}
{\it If $\tau(\g)\,\g\in \G_g(4m)$ with $\g\in \G_g$, then $\g=\,\beta u$ for some $\beta\in \G_g(2,2m)$ and
for some $u\in GL(g,\BZ)$.}

\vskip 0.2cm
\noindent{\bf Lemma 4.} {\it Let $\g\in \G_g(2)$ and suppose $\Om\in\BH_g$ is not fixed by any element of
$\G_g$ other than $\pm I_g$. Suppose $\tau(\Om)=\g\cdot \Om.$ Then there exists an element
$h\in\G_g$ such that $\g=\,\tau(h)\,h^{-1}.$}

\vskip 0.2cm Lemma 4 is a consequence of the theorem of Silhol \cite[Theorem 1.4]{Si3} and Comessatti.
Suppose $f_{\g}$ is a cocycle in  $Z^1(S_\tau,\G_g(4m))$ with $\g\in \G_g(4m)$. According to Theorem B,
its image in $H^1( S_\tau,G))$ is a coboundary and so there exists $h\in G$ with
$\g=\,\tau(h)\,h^{-1}.$ Thus $\BH^{\tau\g}=\,h\cdot (i\,{\mathcal P}_g).$ By Lemma 2.2, there exists
$\Om\in \BH^{\tau\g}$ which are not fixed by any element of $\G_g$ other than $\pm I_{2g}$ and the
set of such points is the complement of a countable union of proper real algebraic subvarieties of
$\BH^{\tau\g}$. According to Lemma 4, $\g=\,\tau(h)\,h^{-1}$ for some $h\in\G_g$. By Lemma 3, we may write
$h=\,\beta\, u$ for some $\beta\in \G_g(2,2m)$ and
for some $u\in GL(g,\BZ)$. Then $\g=\,\tau(h)\,h^{-1}=\,\tau(\beta)\,\beta^{-1}.$ By Lemma 1, the
cohomology class $[f_\g]$ is trivial in $H^1(S_\tau,\G_g(2,2m))$.
\hfill $\square$

\vskip 0.5cm
\noindent{\bf Theorem D.} {\it Let $\G_0=\,\G_g(2,2m)$ and $\G=\,\G_g(4m).$ Then we have the following
results\,:
\vskip 0.2cm
\noindent
(a)  ${\mathbf D}^{\tau}=\,G^\tau/K^\tau\,;$
\vskip 0.2cm
\noindent
(b) For each cohomology class $[f_\g]\in H^1(S_\tau,\G)$, there exists $h\in \G_0$ such that
$\g=\,\tau(h)h^{-1}$, in which case
$$ {\mathbf D}^{\tau\g}=\,h \,{\mathbf D}^{\tau}\qquad {\rm and}\qquad \G^{\tau\g}=\,h\,\G^\tau h^{-1}.$$

\vskip 0.2cm
\noindent
(c) The association $f_\g\lrt h\,$ (cf. see (2)) determines a one-to-one correspondence
between $H^1(S_\tau, \Gamma)$ and $\G\ba \G_0/\G_0^\tau$.

\vskip 0.2cm
\noindent
(d) $$  X_\G^\tau :=\, \coprod_{h\in \G\ba \G_0/\G_0^{\tau} } h\G^\tau h^{-1}\ba h{\mathbf D}^{\tau}.$$}
\vskip 0.2cm \noindent
{\it Proof.} (a) follows from the fact that $H^1( S_\tau,U(g))$ is trivial (cf. Theorem B). (b) follows from Theorem C.
(c) follows from Theorem C, and the facts that $\G$ is a normal subgroup of $\G_0$ and that $\tau$ acts
on $\G\ba \G_0$ trivially. (d) follows from Theorem C.
(c) follows from Theorem C, and the facts that $\G$ is a normal subgroup of $\G_0$ and that $\tau$ acts
on $\G\ba \G_0$ trivially together with the fact that $\G$ s torsion free.
\hfill $\square$

\vskip 0.5cm
\noindent{\bf Corollary.} {\it Let $m$ be a positive integer with $ m\geq 1.$
Let $S_\tau$ be as in Theorem $A$. Let $\G=\,\G_g(4m)$ and $X=\G\ba \BH_g.$ The set $X_\BR$ of real points of $X$
is given by
$$X_\BR=\, \coprod_{h} \G_{[h]}\ba\, h\!\cdot\! (i\,{\mathcal P}_g)\,=\,\G\ba \BH^{\tau\G},$$
where $h$ is indexed by elements
$$ h\in \G_g(4m)\ba \G_g(2,2m)/ \G^{[L]}_g(2)=H^1(S_\tau,\G_g(4m))$$
and $\G_{[h]}:=\,h\,\G_g^{[L]}(4m)\,h^{-1}.$}
\vskip 0.2cm \noindent
{\it Proof.} The proof follows from (c) and (d) in Theorem D.\hfill $\square$

\vspace{1.5cm}

\end{document}